\documentclass[leqno,11pt,oneside]{amsart}

\usepackage{geometry}
\usepackage{physics}
\usepackage{times}

\usepackage{enumerate}
\usepackage[all]{xy} % diagrams

\usepackage{amsmath, amssymb, amsfonts, latexsym, mdwlist, amsthm,amscd}
\usepackage{caption, subcaption}
\usepackage{tikz}
\usetikzlibrary{calc,trees,arrows,chains,shapes.geometric,%
	decorations.pathreplacing,decorations.pathmorphing,shapes,%
	matrix,shapes.symbols}

\tikzset{
	>=stealth',
	punktchain/.style={
		rectangle,
		rounded corners,
		draw=black, thick,
		minimum height=3em,
		text centered,
		on chain},
	line/.style={draw, thick, <-},
	element/.style={
		tape,
		top color=white,
		bottom color=blue!50!black!60!,
		minimum width=8em,
		draw=blue!40!black!90, very thick,
		text width=10em,
		minimum height=3.5em,
		text centered,
		on chain},
	every join/.style={->, thick,shorten >=1pt},
	decoration={brace},
	tuborg/.style={decorate},
	tubnode/.style={midway, right=2pt},
}

\usepackage{cite}
\usepackage{url}

\usepackage{verbatim}
\usepackage{mathtools}

\usepackage{pgf,tikz}
\usepackage{tikz-cd}
\usetikzlibrary{calc,matrix,patterns,shadings,backgrounds,fit}
\pgfdeclarelayer{myback}
\pgfsetlayers{myback,background,main}
\usepackage{stmaryrd}
\usepackage{diagbox}
\usepackage{extarrows}

\usepackage[bookmarks, colorlinks, breaklinks, pdftitle={Moduli spaces on Kuznetsov components of Fano threefolds},
pdfauthor={ Chunyi Li}]{hyperref}
\hypersetup{linkcolor=blue,citecolor=blue,filecolor=black,urlcolor=blue}

\usepackage{longtable}
\usepackage{lipsum}

\usepackage{paralist}
\setdefaultenum{(a)}{(i)}{}{}

\usepackage{pgfplots}
\pgfplotsset{compat=1.15}
\usepackage{mathrsfs}
\usetikzlibrary{arrows}

\definecolor{xdxdff}{rgb}{0.49019607843137253,0.49019607843137253,1}
\definecolor{ududff}{rgb}{0.30196078431372547,0.30196078431372547,1}

\numberwithin{equation}{section}

\def\C{\ensuremath{\mathbb{C}}}

\def\H{\ensuremath{\mathbb{H}}}

\def\P{\ensuremath{\mathbb{P}}}
\def\Q{\ensuremath{\mathbb{Q}}}
\def\R{\ensuremath{\mathbb{R}}}
\def\Z{\ensuremath{\mathbb{Z}}}

\def\bx{\ensuremath{\mathbf{x}}}
\def\bv{\ensuremath{\mathbf{v}}}

\def\ch{\mathop{\mathrm{ch}}\nolimits}

\def\Coh{\mathop{\mathrm{Coh}}\nolimits}
\def\codim{\mathop{\mathrm{codim}}\nolimits}

\def\deg{\mathop{\mathrm{deg}}}

\def\dim{\mathop{\mathrm{dim}}\nolimits}
\def\ev{\mathop{\mathrm{ev}}\nolimits}

\def\End{\mathop{\mathrm{End}}}

\def\Ext{\mathop{\mathrm{Ext}}\nolimits}

\def\Hilb{\mathrm{Hilb}}
\def\Hom{\mathop{\mathrm{Hom}}\nolimits}
\def\id{\mathop{\mathrm{id}}\nolimits}

\def\im{\mathop{\mathrm{im}}\nolimits}

\def\Kn{\mathop{\mathrm{K}_{\mathrm{num}}}}

\def\min{\mathop{\mathrm{min}}\nolimits}

\def\Pic{\mathop{\mathrm{Pic}}}

\def\rk{\mathop{\mathrm{rk}}}
\def\RlHom{\mathop{\mathbf{R}\mathcal Hom}\nolimits}

\def\td{\mathop{\mathrm{td}}\nolimits}

\def\Stab{\mathop{\mathrm{Stab}}\nolimits}

\def\glr{\ensuremath{\mathrm{GL}^+_2(\mathbb R)}}

\def\Db{\mathrm{D}^{b}}

\def\glt{\widetilde{\mathrm{GL}}^+_2(\R)}

\def\AJ{\ensuremath{\text{Abel--Jacobi}} }
\def\IJ{\ensuremath{\Phi} }

\def\ms{M^s_\sigma}
\def\RHom{\mathrm{RHom}}
\def\Rep{\mathrm{Rep}}

\def\Pe#1#2{\ensuremath{\mathbb P_\sigma(#1,#2)}}

\makeatletter
\newtheorem*{rep@theorem}{\rep@title}
\newcommand{\newreptheorem}[2]{%
\newenvironment{rep#1}[1]{%
 \def\rep@title{#2 \ref{##1}}%
 \begin{rep@theorem}}%
 {\end{rep@theorem}}}
\makeatother

\newtheorem{Thm}{Theorem}[section]
\newreptheorem{Thm}{Theorem}
\newtheorem{Prop}[Thm]{Proposition}
\newtheorem{PropDef}[Thm]{Proposition and Definition}
\newtheorem{Lem}[Thm]{Lemma}

\newtheorem{Cor}[Thm]{Corollary}
\newreptheorem{Cor}{Corollary}

\newtheorem{Ques}[Thm]{Question}

\newtheorem{Asp}[Thm]{Assumption}
\newtheorem{SubLem}[Thm]{Sublemma}

\newtheorem{thm-int}{Theorem}

\theoremstyle{definition}
\newtheorem{Def-s}[Thm]{Definition}
\newtheorem{Def}[Thm]{Definition}
\newtheorem{Rem}[Thm]{Remark}

\newtheorem{Ex}[Thm]{Example}

\newtheorem{Not}[Thm]{Notation}

\def\C{\ensuremath{\mathbb{C}}}

\def\H{\ensuremath{\mathbb{H}}}

\def\P{\ensuremath{\mathbb{P}}}
\def\Q{\ensuremath{\mathbb{Q}}}
\def\R{\ensuremath{\mathbb{R}}}
\def\Z{\ensuremath{\mathbb{Z}}}

\def\cA{\ensuremath{\mathcal A}}

\def\cE{\ensuremath{\mathcal E}}
\def\cF{\ensuremath{\mathcal F}}

\def\cH{\ensuremath{\mathcal H}}
\def\cI{\ensuremath{\mathcal I}}

\def\cK{\ensuremath{\mathcal K}}
\def\cL{\ensuremath{\mathcal L}}

\def\cO{\ensuremath{\mathcal O}}
\def\cP{\ensuremath{\mathcal P}}
\def\cR{\ensuremath{\mathcal R}}

\def\cT{\ensuremath{\mathcal T}}
\def\cU{\ensuremath{\mathcal U}}
\def\cV{\ensuremath{\mathcal V}}

\def\TT{\ensuremath{\mathcal T}}
\def\PP{\ensuremath{\mathcal P}}
\def\MM{\ensuremath{\mathcal M}}

\usepackage{todonotes}

\def\kn{\mathrm{K}_{\mathrm{num}}}
\def\Cone{\mathrm{Cone}}
\def\gd{\mathrm{gldim}}
\def\hom{\mathrm{hom}}
\def\Ku{\mathrm{Ku}}
\def\cHom{\mathcal Hom}

\def\be{\mathbf e}

\def\BN{\mathfrak{Z}}
\def\ET{\mathfrak{E}}
\def\HN{\mathrm{HN}}
\def\BNP{\pi}

\newcommand*{\parallelogramm}{%
  \rlap{\rotatebox{-30}{\rule[.05ex]{.4pt}{.77em}}}%
  \kern.04em%
  \rlap{\kern.36em\raisebox{0.649519052835em}{\rule{.6em}{.4pt}}}%
  \rule{.6em}{.4pt}\kern-.04em%
  \rotatebox{-30}{\rule[.05ex]{.4pt}{.77em}}}

\title{Higher dimensional moduli spaces on Kuznetsov components of Fano threefolds}
\author{Chunyi Li}
\address{C. L.:
Mathematics Institute, University of Warwick,
Coventry, CV4 7AL,
United Kingdom}
\email{C.Li.25@warwick.ac.uk}
\urladdr{https://sites.google.com/site/chunyili0401/}

\author{Yinbang Lin}
\address{Y. L.: School of Mathematical Sciences, Tongji University, Shanghai 200092, China
}
\email{yinbang.lin@icloud.com}
\urladdr{https://yinbang-lin.github.io/}

\author{Laura Pertusi} 
\address{L. P.:
Dipartimento di Matematica ``F.\ Enriques'' \\
Via Cesare Saldini 50 \\
Universit\`a degli Studi di Milano \\
20133 Milano, Italy
}
\email{laura.pertusi@unimi.it}
\urladdr{http://www.mat.unimi.it/users/pertusi/index.html}

\author{Xiaolei Zhao}
\address{X. Z.:
Department of Mathematics \\
South Hall 6607 \\
University of California \\
Santa Barbara, CA 93106, USA
}
\email{xlzhao@math.ucsb.edu}
\urladdr{https://sites.google.com/site/xiaoleizhaoswebsite/}

\begin{document}
\begin{abstract}
We study moduli spaces of stable objects in the Kuznetsov components of Fano threefolds. We prove a general non-emptiness criterion for moduli spaces, which applies to the cases of prime Fano threefolds of index $1$, degree $10 \leq d \leq 18$, and index $2$, degree $d \leq 4$. In the second part, we focus on cubic threefolds. We show the irreducibility of the moduli spaces, and that the general fibers of the Abel--Jacobi maps from the moduli spaces to the intermediate Jacobian are Fano varieties. When the dimension is sufficiently large, we further show that the general fibers of the Abel--Jacobi maps are stably birational equivalent to each other. As an application of our methods, we prove Conjecture A.1 in \cite{FGLZ:EPW} concerning the existence of Lagrangian subvarieties in moduli spaces of stable objects in the Kuznetsov components of very general cubic fourfolds.
\end{abstract}
\maketitle
\setcounter{tocdepth}{1}
\tableofcontents

\section{Introduction}

Hilbert schemes parametrizing points or curves, or moduli spaces of stable coherent sheaves on a smooth Fano threefold of Picard rank one do not admit nice geometric properties in general. When the expected dimension is high, such a moduli space is typically reducible, consisting of components with different dimensions, and with singularities hard to control. 

On the other hand, recent research directions investigate residual components in the bounded derived category of Fano varieties, known as Kuznetsov components, arising from semiorthogonal decompositions defined by exceptional collections. There is a rich emerging theory of moduli spaces of stable objects in Kuznetsov components of Fano varieties of Picard rank one, with applications in both classical geometry and categorical Torelli theorems, see \cite{MS_survey, PS_survey} for surveys.

In this paper, we focus on the case of Fano threefolds of Picard rank $1$. We show a non-emptiness result for moduli spaces of semistable objects in the Kuznetsov components, and in the case of cubic threefolds we obtain interesting analogies with moduli spaces of vector bundles on genus $\geq 2$ curves.

\subsection{Kuznetsov components of Fano threefolds}

Let $X$ be a Fano threefold of Picard rank $1$. When $X$ is of index $2$, in other words, the anticanonical divisor satisfies $-K_X=2H$, where $H$ is the ample generator for $\Pic(X)$, the \emph{Kuznetsov component} of $X$ is defined as
\begin{align*}
    \Ku(X):=\{E\in\Db(\Coh(X))\;|\;\Ext^i(\cO_X,E)=\Ext^i(\cO_X(H),E)=0 \text{ for all }i\in \Z\}.
\end{align*}
The definition of the Kuznetsov component in the index $1$ case and more details are discussed in Section \ref{sec:kuz}.

Stability conditions have been constructed on Kuznetsov components of Fano threefolds of index $1$ and $2$ in \cite{BLMS:kuzcomponent}. We denote by $\sigma$ a stability condition on $\Ku(X)$ in the same $\glt$-orbit of the constructed one (see Remark \ref{not:glt}). By the general theory in \cite{liurendapaper}, given a numerical character $v\in\Kn(\Ku(X))$, there is a good moduli space $M_\sigma(\Ku(X),v)$ parametrizing $\sigma$-semistable objects with character $v$ in $\Ku(X)$, having the structure of a proper algebraic space.

Despite of many results, especially in low dimensions, general structure theorems for these moduli spaces are still missing. In particular, the non-emptiness of these moduli spaces was only known in special cases \cite{PY, LiuZhang_note, Arend:cubic3, JLZ, PPZ_Enriques}, or when the Kuznetsov component is in fact equivalent to the bounded derived category of a curve of genus $\geq 2$ (this happens precisely when $X$ has index $2$, degree $4$, or index $1$, degree $12$, $16$, $18$). In the other cases, the main difficulty is that the categories $\Ku(X)$ arising from Fano threefolds within a given deformation class are typically not equivalent to the derived category of a smooth projective variety. This prevents the application of any specialization argument to reduce to a known geometric case, such as done for cubic fourfolds.

In our first result, we settle the non-emptiness problem of the moduli spaces for Fano threefolds of index $2$, and index $1$, degree $\geq 10$.

\begin{Thm}[Theorem \ref{thm:existinKurank2}]\label{thm:existintro}
    Let $X$ be a Fano threefold with Picard rank $1$, index $1$, degree $10\leq d \leq 18$, or index $2$, degree $d\leq 4$. Then for every nonzero character $v\in\kn(\Ku(X))$, the moduli space $M_\sigma(\Ku(X),v)$ of $\sigma$-semistable objects of class $v$ is non-empty.
\end{Thm}

The actual new cases covered by Theorem \ref{thm:existintro} are those of cubic threefolds (index $2$, degree $3$), Fano threefolds with index $1$ and degree $14$, and double Veronese cones (index $2$, degree $1$). For quartic double solids and Gushel--Mukai threefolds (index $2$, degree $2$, and index $1$, degree $10$) this result was proved in \cite{PPZ_Enriques} with different methods which apply in the context of Enriques categories. Theorem \ref{thm:existintro} provides a unified argument for the proof of the non-emptiness of the moduli spaces for all the cases.

Note also that if $X$ has index $1$, degree $22$ (or index $2$, degree $5$), the component $\Ku(X)$ is equivalent to the bounded derived category of finite dimensional representations of the Kronecker quiver with three arrows. In the remaining cases (index $1$, degree $2$, $4$, $6$, $8$) the convention on the Kuznetsov component is floating and they seems to require a different strategy, see Remark \ref{rem:kuoflowgenusfano3}.

Theorem \ref{thm:existintro} is deduced from a general non-emptiness criterion proved in Proposition \ref{prop:existencegeneralcase}, which is of independent interest. We develop an inductive argument that effectively reduces the question to checking the statement in the low dimensional cases, where the moduli spaces are related to low degree curves on Fano threefolds, and the non-emptiness can be verified directly. 

As an application, we obtain the following result about the connected component of the stability manifold of $\Ku(X)$ containing the stability condition $\sigma$.

\begin{Cor}[Corollary \ref{cor:stabmfdofku}] \label{cor:stabmfdofku_intro}
Let $X$ be a Fano threefold with Picard rank $1$, index $1$, degree $\in \{10, 12, 14, 16, 18\}$ or index $2$, degree $\leq 4$. Then the connected component of the stability manifold $\Stab(\Ku(X))$ containing $\sigma$ is isomorphic to $\C\times \H$.     
\end{Cor}

The above result was known for cubic threefolds (index $2$, degree $3$), quartic double solids (index $2$, degree $2$) and Gushel--Mukai threefolds (index $1$, degree $10$) by \cite[Theorem 1.1]{zhiyu:contractKu}. 

\subsection{Moduli spaces and cubic threefolds}

Once the non-emptiness is settled, the next step is to investigate the geometric properties of the moduli spaces. Here we focus on the case of cubic threefolds. 

Let $Y_3$ be a cubic threefold, and $\Ku(Y_3)$ be its Kuznetsov component. The small dimensional moduli spaces on $\Ku(Y_3)$ are well understood \cite{FP, Arend:cubic3}: up to isomorphism, there exists a $2$-dimensional moduli space, isomorphic to the Fano surface of lines on $Y_3$, and a $4$-dimensional space, isomorphic to the blowup of the theta divisor at the unique singular point. More details on these two moduli spaces are reviewed in Section \ref{sec:mssmall}.

The focus of the second part of this paper is on the properties of higher dimensional moduli spaces. We show that they behave similarly to moduli spaces of vector bundles on a curve of genus $g\geq 2$. First we recall the following classical theorem:
\begin{Thm}[\cite{ramanan1973moduli, DrezetNarasimhan:modulioncurves, KingSchofield:rationality}]
    Let $C$ be a smooth projective curve of genus $g\geq 2$, and let $v=(r,d)$ be a primitive character with $r\geq 2$. Then we have the following statements.
    \begin{enumerate}[(1)]
        \item The moduli space $M(v)$ of slope stable vector bundles of class $v$ is a smooth irreducible projective variety of the expected dimension $1-\chi(v,v)$.
        \item The determinant map $\det\colon M(v)\to \Pic^d(C)$ is isotrivial. The fiber $M(r,\cL) $, parametrizing rank $r$ stable vector bundles with fixed determinant $\cL$, is a smooth Fano variety of Picard rank $1$ and index $2$.
        \item The fiber $M(r,\cL) $ is rational.
    \end{enumerate}
\end{Thm}

Recall that the intermediate Jacobian of a cubic threefold $Y_3$ is defined by \[J(Y_3):= \frac{H^{2,1}(Y_3)^*}{H_3(Y_3,\mathbb Z)}.\] This is naturally an abelian fivefold, and plays a central role in the beautiful classic result of Clemens and Griffiths on the irrationality of cubic threefolds \cite{CG:RationalityCubics}.

For every primitive $v\in \Kn(\Ku(Y_3))$, set $M_\sigma(v):=M_\sigma(\Ku(Y_3),v)$ for simplicity. After choosing a base point $F_0$, we can define an \AJ map by the (cycle-theoretic) second Chern class:
\begin{align*}
    \IJ\colon M_\sigma(v)\to J(Y_3):
       F\mapsto c_2(F)-c_2(F_0)
\end{align*}
(see Section \ref{sec:ajmap} for more details and the construction in the case of non-primitive numerical class). The map $\IJ$ is a morphism of varieties. For every $c\in J(Y_3)$, we denote by $M_\sigma(v,c):=\IJ^{-1}(c)$ the fiber of the \AJ map. The main results of this paper can be summarized as follows.

\begin{Thm}\label{thm:Y3moduli}
    Let $Y_3$ be a cubic threefold, $\Ku(Y_3)$ be its Kuznetsov component, and let $v\in\Kn(\Ku(Y_3))$ be a primitive character. Then we have the following statements.
    \begin{enumerate}[(1)]
        \item (Corollary \ref{cor:cub3smoothmoduli}) The moduli space $M_\sigma(v)$ is smooth irreducible projective of the expected dimension $1-\chi(v,v)$.
        \item (Theorem \ref{thm:connectedfiber} and \ref{thm:fanofiber}) If $\dim M_\sigma(v)> 5$, then the \AJ map $\IJ\colon M_\sigma(v)\to J(Y_3)$ is surjective with connected fibers. For a general point $c\in J(Y_3)$, the fiber $M_\sigma(v,c)$ is a smooth Fano variety with primitive canonical divisor class.
        \item (Theorem \ref{thm:connectedfiber}) Assume that $v,w$ are primitive characters such that $\dim M_\sigma(v)\geq 23$ and $\dim M_\sigma(w)\geq 23$. Then for general points $c,c'\in J(Y_3)$, the fibers $M_\sigma(v,c)$ and $M_\sigma(w,c')$ are stably birational equivalent.
    \end{enumerate}
\end{Thm}

Note that in Corollary \ref{cor:msvirreducible} we show in fact the irreducibility of $M_\sigma(v)$ for every (not necessarily primitive) character $v \in\Kn(\Ku(Y_3))$. As a consequence of Theorem \ref{thm:Y3moduli} we obtain the following statement.

\begin{Cor}[Corollary \ref{Cor:mrcquotient}]
Let $Y_3$ be cubic threefold and $\Ku(Y_3)$ be its Kuznetsov component. Then for every primitive $v\in\Kn(\Ku(Y_3))$ with $\chi(v,v)\leq -4$, the maximal rationally connected quotient of $M_\sigma(v)$ is the intermediate Jacobian $J(Y_3)$.
\end{Cor}

Here we briefly explain the difference among the curve case and that of cubic threefolds. In the classical case of curves, the property of being Fano for $M(r,\cL)$ follows from the unirationality and the Picard rank one. However, for cubic threefolds, we do not yet know how to prove these two properties for $M_\sigma(v,c)$. Our argument to show the Fano statement involves observations on the moduli space from different perspectives. Firstly, by Beauville's diagonal trick, as in the curve case, the cohomological ring of $M_\sigma(v)$ is generated by the tautological classes. It follows that the image of $H^2(M_\sigma(v),\Z)\to H^2(M_\sigma(v,c),\Z)$ has rank at most $1$. When $\dim M_\sigma(v)>5$, we can prove that the \AJ map is surjective, and all divisors on $M_\sigma(v,c)$ that are restricted from $M_\sigma(v)$ are proportional to each other. In particular, the canonical divisor $\omega_{M_\sigma(v,c)}$ is proportional to the restriction of an ample divisor. Finally, the (anti)ampleness of $\omega_{M_\sigma(v,c)}$ is determined by its degree on a curve in $M_\sigma(v,c)$. To do this, we choose a curve $C\cong\P^1$ parametrizing objects which are extensions of two objects $E_1$ and $E_2$, and show that the restriction $\omega_{M_\sigma(v,c)}|_C$ only relies on the $\Ext^1(E_i,E_j)$'s. This reduces the computation to a quiver model and concludes that $\omega_{M_\sigma(v,c)}|_C$ is of degree $-1$.

The rationality of $M(r,\cL)$ \cite[Theorem 1.2]{KingSchofield:rationality} was proved by building up rational dominant maps $\lambda_F\colon M(r,\cL)\dashrightarrow M(r_1,\cL_1)$ for $r_1<r$ and studying carefully the Brauer class. However, the statement \cite[Proposition 2.3]{KingSchofield:rationality}, which is essential for the construction of $\lambda_F$, does not hold in the $\Ku(Y_3)$ case. This is technically due to the difference between Euler pairings on $\Kn$ in these two cases. Also, in our case, it is known that the general fiber $M_\sigma(2\alpha+\beta,c)$ is birational to $Y_3$, which is irrational (see Example \ref{ex:quartics}).

\subsection{Brill--Noether loci and extension loci} \label{subsec:BNlociandEloci}
To explain the strategy of the proof of Theorem \ref{thm:Y3moduli}, we need to introduce the following key concepts in this paper. 

For a numerical character $v\in\Kn(\Ku(Y_3))$, while the smoothness and irreducibility of the whole moduli space $M_\sigma(v)$ in Theorem \ref{thm:Y3moduli}.(1) follows from the Mukai trick, the irreducibility of the general fiber $M_\sigma(v,c)$ resists classical approaches. Our strategy is to show that there exists $M_\sigma(v_1)$ and $M_\sigma(v_2)$ with strictly smaller dimensions such that every general object in $M_\sigma(v)$ is the extension of two objects in $M_\sigma(v_1)$ and $M_\sigma(v_2)$. This induces a rational map from the relative $\Ext^1$ space over $M_\sigma(v_1) \times M_\sigma(v_2)$ to $M_\sigma(v)$, which is compatible with the \AJ map. The irreducibility of $M_\sigma(v,c)$ then follows by induction.

To achieve the above strategy, the key point is to understand the locus of objects in $M_\sigma(v)$ that are extended by two `smaller' stable pieces. More precisely, for two numerical characters $v_i\in\Kn(\Ku(Y_3))$, we define the extension locus:
\[\ET(v_1,v_2):=\left\{E\in M^s_\sigma(v_1+v_2)\;\middle|\;\begin{aligned}\exists&\text{ distinguished}
\text{ triangle } \\& E_1\to E\to E_2\xrightarrow{+} \\ &\text{for some } E_i\in M^s_\sigma(v_i)
\end{aligned} \right\}\subset M^s_\sigma(v_1+v_2).\]
For a given $F\in\ms(v_1)$, we denote by $\ET(F,v_2)$ the sublocus of $\ET(v_1,v_2)$ where the first factor is fixed as $F$.

This is closely related to the Brill--Noether locus defined as:
\begin{align*}
    \BN(F,v):=\{E_v\in\ms(v)\;|\;  \Hom(F,E_v)\neq 0\}\subset \ms(v).
\end{align*}
 When $v_1+v_2=v$, and $F\in\ms(v_1)$, we have the following key formula on the relation between these two loci:
\begin{align} \label{eq:relation}
    \ET(F,v_2)\subseteq\BN(F,v)\subseteq \ET(F,v_2)\cup \left(\scalebox{1.5}{\ensuremath{\cup}}_{v'\in\triangle^*(v_1,v)}\ET(v',v-v')\right).
\end{align}
Here $\triangle^*(v_1,v)$ stands for all characters in the triangle spanned by $v_1$ and $v$, see Notation \ref{not:triangleandpara} and Proposition \ref{prop:imofprimpliesext} for the precise definition and proof. In practice, we can usually control the dimensions of the remaining terms $\ET(v',v-v')$'s so that they are strictly smaller than that of $\BN(F,v)$. Given this technical interpretation between the extension locus and Brill--Noether locus, we may show a stratification theorem for the moduli spaces. 

 To state the result, we consider the characters $\beta:=[\cI_\ell]=(1,0,-L,0)$, where $\ell$ is a line in $Y_3$, and $\alpha:=3[\cO_{Y_3}]-[\cI_\ell(H)]$, which form a basis for $\Kn(\Ku(Y_3))$. By \cite[Proposition 5.7]{PY}, the Serre functor preserves the $\sigma$-stability of the objects, and up to the Serre functor, every non-zero character is of the form $n\alpha+m\beta$ for some $n\in\Z_{\geq 1}$ and $m\in\Z_{\geq 0}$.

\begin{Thm}[Theorem \ref{thm:stratcubic3} and Proposition \ref{prop:stratofmbeta}]\label{thm:introstrat}
    Let $Y_3$ be a cubic threefold and $\Ku(Y_3)$ be its Kuznetsov component. Then for every  $m,n\in\Z_{>0}$, we have
    \begin{align}
        \ms(n\alpha+m\beta) & =\scalebox{1.5}{\ensuremath{\cup}}_{0\leq i\leq n,0\leq j\leq m,\tfrac{j}{i}<\tfrac{m}n}\ET(i\alpha+j\beta,(n-i)\alpha+(m-j)\beta), \label{eq:11}\\
        \ms(m\beta) & =\ET(\alpha,m\beta-\alpha).\nonumber
    \end{align}
    Furthermore, we have $\overline{\ET(n\alpha,m\beta)}= \ms(n\alpha+m\beta)$ and the other terms in \eqref{eq:11} are of strictly smaller dimensions.
\end{Thm}
Our proof of the stably birationality of the fibers is to directly build up birational maps between different $\Pe E v$'s, where  $\Pe E v$ is the projectivization of relative $\Ext^1(E,-)$ space over $\ms(v)$, see Section \ref{sec4.2} for details on the definition. When $v$ is primitive, the space $\Pe Ev$ is birational to $\ms(v)\times \P^r$ for some $r\in\Z_{\geq0}$. Birational maps between $\Pe Ev$'s are constructed in Propositions \ref{prop:stabbir1}, \ref{prop:stabbir2}, \ref{prop:stabbir3}, and \ref{prop:stabbir4}. These constructions also imply stably birationality (resp.\ birationality) between moduli spaces as in Corollary \ref{cor:stabbir}.

If $\Ext^1$ has dimension $1$, we can actually prove that $\ms(v,c)$ is birational to $\ms(v,c')$ for some $v$. Such examples are given in Proposition \ref{prop:birfiber}.  

An interesting and mysterious part for us is that when the dimension of $\ms(v)$ is not large enough. More precisely, when $v=n\alpha+m\beta$ with $m,n>0,m+n\leq 5$, using this approach, it seems difficult to get a birational map from $\Pe Ev$ to another $\Pe E{v'}$, where $v'$ is with smaller $m'+n'$ (see Figure \ref{fig:bir} for reference). This explains the dimension bound $23$ in Theorem \ref{thm:Y3moduli}.(3).  

\subsection{Applications}

 One application of Theorem~\ref{thm:introstrat} is the following result, which answers to \cite[Remark 4.10]{heart}.

\begin{Prop}[Proposition \ref{prop:dbofheart}, Corollary \ref{cor:unique_dg}] \label{prop:dbofheart_intro}
Let $Y_3$ be a cubic threefold, $\Ku(Y_3)$ be its Kuznetsov component, and $\cA$ be the heart of the stability condition $\sigma$. Then $\Db(\cA)$ is equivalent to $\Ku(Y_3)$. The Kuznetsov component $\Ku(Y_3)$ has a strongly unique dg enhancement, and every equivalence $\Ku(Y_3) \to \Ku(Y_3')$ for another cubic threefold $Y_3'$ is of Fourier--Mukai type.
\end{Prop}

In a different direction, we use Proposition \ref{prop:existencegeneralcase} to prove \cite[Conjecture A.1]{FGLZ:EPW}. Recall that if $Y_4$ is a cubic fourfold, then its Kuznetsov component $\Ku(Y_4)$ has a stability conditions $\sigma_4$ constructed in \cite{BLMS:kuzcomponent} (see Section \ref{sec:lagrangian} for further details).

\begin{Thm}[Theorem \ref{thm:C2}, \cite{FGLZ:EPW}, Conjecture A.1] \label{thm:C2intro}
Let $Y_4$ be a very general cubic fourfold, and let $j\colon Y_3\to Y_4$ be a smooth hyperplane section. Then for every primitive character $v\in\Kn(\Ku(Y_3))$, there exists a non-empty open subset $U_v\subset M^s_{\sigma}(\Ku(Y_3),v)$ such that for every $E\in U_v$, the projection in $\Ku(Y_4)$ of $j_*E$ is $\sigma_4$-stable. 
\end{Thm}

Combined with \cite[Theorem A.4]{FGLZ:EPW}, Theorem \ref{thm:C2intro} provides the construction of Lagrangian subvarieties inside hyperk\"ahler manifolds arising as moduli spaces of stable objects in $\Ku(Y_4)$ (see Theorem~\ref{thm:lagrangian_existence} for the statement). This result could have potential applications to the construction of atomic objects supported on Lagrangian subvarieties as in \cite{Bottini, GL:newpaper}.

\subsection*{Further directions}
We explain in Section \ref{sec:hilb_curves} some potential applications of our results to Hilbert schemes of curves on cubic threefolds. In particular, up to showing the stability of a certain twist of the ideal sheaf of a smooth 
irreducible curve $C$ on $Y_3$ (or its projection in the Kuznetsov component), one would deduce a birational description of the component of the Hilbert scheme containing $C$ as a moduli space of stable objects, and that its mrc quotient is the intermediate Jacobian. 

Finally, we expect our method to apply to other examples, for instance the quartic double solid case. However, the actual computations heavily depend on the numeric in different cases, so we leave these to future projects.

\subsection*{Organization of the paper}
In Section \ref{sec2}, we review the notion of Bridgeland stability conditions. Then we prove a general theorem on the non-emptiness of moduli spaces in Proposition \ref{prop:existencegeneralcase}. In Section \ref{sec3}, we recall the notion of Kuznetsov component of a Fano threefold. Then we apply the general theory in Section \ref{sec2} to prove the non-emptiness of $M_\sigma(\Ku(X),v)$ as in Theorem \ref{thm:existintro}. In Section \ref{sec4}, we introduce some technical notions on the loci of moduli spaces and maps between them. In Section \ref{sec:moduliY3} we consider the case of cubic threefolds. After reviewing some general properties on $\Ku(Y_3)$, we introduce the \AJ map, we recall the known results on moduli spaces of small dimension and we prove some technical lemmas involving dimension estimates for moduli spaces. Section \ref{sec6} is devoted to the construction of the birational maps among the projectivized $\Ext^1$ spaces, and contains the proof of Proposition \ref{prop:dbofheart_intro}. In Section \ref{sec7} we conclude the proof of Theorem \ref{thm:Y3moduli}. In Section \ref{sec:lagrangian}, we prove Theorem \ref{thm:C2intro} on the existence of Lagrangian subvarieties in hyperk\"ahler moduli spaces associated to cubic fourfolds. In Section \ref{sec:hilb_curves}, we look at some applications to the classical geometry of curves on cubic threefolds. Finally in Section \ref{sec:questions}, we raise some questions for further study.

\subsection*{Acknowledgements}
We would like to thank Arend Bayer, Hannah Dell, Lie Fu, James Hotchkiss, Zhiyu Liu, Emanuele Macr\`i, Alex Perry, Junliang Shen, Ruijie Yang for many useful discussions related to this work.

Part of this work was written while the third and fourth named authors were attending the Junior Trimester program “Algebraic geometry: derived categories, Hodge theory, and Chow groups” at the Hausdorff Institute for Mathematics in Bonn, funded by the Deutsche Forschungsgemeinschaft (DFG, German Research Foundation) under Germany Excellence Strategy– EXC-2047/1– 390685813, and visiting the Simons Laufer Mathematical Sciences Institute (formerly MSRI) with the support of the NSF grant DMS-1928930. We are pleased to thank these institutions for the warm hospitality and the wonderful working environment.

C.L.\ is supported by the Royal Society URF$\backslash$R1$\backslash$201129 “Stability condition and application in algebraic geometry”.
Y.L.\ is supported by Applied Basic Research Programs of Science and Technology Commission Foundation of Shanghai Municipality.
L.P.\ is a member of the Indam group GNSAGA. 
X.Z.\ is partially supported by NSF grant DMS-2101789, and NSF FRG grant DMS-2052665.\\

\noindent\textbf{Notation}

\begin{longtable}{cp{0.85\textwidth}}
$\Db(X)$ & bounded derived category of coherent sheaves on $X$\\
$\cT$ &  a $\C$-linear triangulated category, an SOD factor of $\Db(X)$ \\
$\Kn$ &  (lifted) numerical Grothendieck group, see Remark \ref{rem:Knum}\\
$v,w$ & characters in $\Kn(\cT)$\\
$\chi(-,-)$ & Euler paring \\
$\sigma$ &   stability condition \\
$\cA$ &   the heart of a bounded $t$-structure / an abelian category \\
$Z$ &  central charge; (weak) stability function on a bounded heart \\
$\phi_\sigma$ &   phase function of a stability condition $\sigma$ \\
$\mu_Z$ &  slope of a (weak) stability function $Z$ \\
$\gd$ &  global dimension function\\
$\ms(-)(M_\sigma(-))$ &  moduli space of (semi)stable objects \\
$\Lambda$ &  rank $2$ lattice with a fixed isomorphism with $\Z^{\oplus 2}$\\
$\triangle(-,-)$ &  lattice points in a triangle, see Notation \ref{not:triangleandpara}  \\
$E_f$ &  object $\Cone(f)[-1]$\\
$\mathsf S$ &  Serre functor\\
$\Ku(X)$ & the Kuznetsov component of a variety $X$\\
$Y_3$ &  smooth cubic threefold\\

 $\alpha,\beta,\gamma$    &characters in $\Kn(\Ku(Y_3))$, see Notation \ref{not:abc}\\

$\BN^*(-,-)$ &  Brill--Noether jumping locus, see Definition \ref{def:BNlocus}\\
$\ET^*(-,-)$ &  extension locus, see Definition \ref{def:ETlocus}\\
$\ms(v,w)^\dag$ &  non-jumping locus of $\ms(v)\times \ms(w)$, see Definition \ref{def:non-jumping}\\
$\Pe vw$ & projectivization of the relative $\Ext^1$ over $\ms(v,w)^\dag$, see Definition \ref{def:proj_nj}\\
$e_{v,w}^*$ & rational maps from $\Pe vw$, see Definition \ref{def:evw} and \ref{def:evwR}\\
$J_v(Y_3)$ & (twisted) intermediate Jacobian of $Y_3$\\
\end{longtable}

\section{Non-emptiness of the moduli spaces}
\label{sec2}

In this section we recall some definitions and properties about numerical stability conditions and moduli spaces. Then we assume the stability condition has discrete central charge and in Lemma \ref{lem:extendstabobj} we show that the extension of two stable objects with the ``smallest possible space'' among their characters is stable. We apply this in Proposition \ref{prop:existencegeneralcase} to prove the non-emptiness of moduli spaces assuming the existence of stable objects in smaller dimensional moduli. 

\subsection{Review: Stability conditions on triangulated categories}
Let $X$ be a smooth projective variety defined over the field of complex numbers $\C$. Assume that $\cT$ is a full admissible subcategory of  $\Db(X)=\Db(\Coh(X))$, the bounded derived category of coherent sheaves on $X$. In other words, the inclusion functor $\cT \to \Db(X)$ is fully faithful and has left and right adjoints. The Grothendieck group $\mathrm{K}_0(\TT)$ is equipped with the Euler pairing $\chi\colon \mathrm K_0(\TT)\times \mathrm K_0(\TT)\to \Z$ defined as:
\[\chi([E],[F]):=\sum_{n\in\Z}(-1)^n\dim\Hom(E,F[n]).\]
The numerical Grothendieck group $\mathrm{K}_{\mathrm{num}}(\TT):=\mathrm{K}_0(\TT)/\ker(\chi)$ is isomorphic to a subgroup of $\mathrm{K}_{\mathrm{num}}(X)$ which is a finitely generated free abelian group.

We recall the definition and first properties of stability conditions on $\cT$, introduced by Bridgeland in \cite{Bridgeland:Stab}. 

\begin{Def}[(Weak) Stability function]\label{def:weakstabfunction}
    Let $\cA$ be an abelian category. A \emph{weak numerical stability function} on $\cA$ is a group homomorphism $Z\colon\Kn(\cA)\to \C$ such that for any non-zero object $ E \in \cA$ we have $\Im Z(E) \geq 0$, and in the case that $\Im Z(E) = 0$, we have $\Re Z(E) \leq 0$. 

    We call $Z$ a \emph{stability function} if, moreover, $Z(E)\neq 0$ for every $0\neq E\in\cA$.
\end{Def}

 \noindent The \emph{slope} of $0\neq E \in \cA$ is defined as 
    \[\mu_{Z}(E)= 
    \begin{cases}
    -\frac{\Re Z(E)}{\Im Z(E)} & \text{if } \Im Z(E) > 0, \\
    + \infty & \text{otherwise.}
    \end{cases}\]
    An object $E \in \cA$ is $\mu_Z$-\emph{(semi)stable} if for every nonzero proper subobject $F $ of $E$ in $\cA$ we have $\mu_{Z}(F) < (\leq) \ \mu_{Z}(E/F)$.

\begin{Def} [Stability condition]\label{def:stabilitycond}
Let  $\cA$ be the heart of a bounded t-structure on $\cT$ and $Z$ be a (weak) numerical stability function on $\cA$. Denote by $\sigma=(\cA, Z)$ the pair of these data. 
A non-zero object $E\in\cT$ is called $\sigma$-\emph{(semi)stable} if $E[n]\in\cA$ for some $n\in\Z$ and $E[n]$ is $\mu_Z$-(semi)stable.

We call $\sigma$ a (weak) numerical \textit{stability condition}  on $\cT$ when it  satisfies the following properties:
\begin{enumerate}  
    \item (Harder--Narasimhan Filtration) Every non-zero object $E \in \cA$ has a unique filtration 
    \begin{align}\label{eq:hnfil}
    0=E_0 \hookrightarrow E_1 \hookrightarrow \dots E_{m-1} \hookrightarrow E_m=E 
    \end{align}
    where $A_i:=\Cone(E_{i-1}\hookrightarrow E_i)$ is $\sigma$-semistable and $\mu_Z(A_1) > \dots > \mu_Z(A_m)$. 
    \item (Support Property) There exists a quadratic form $Q$ on $\Kn(\cT) \otimes \R$ such that \begin{itemize}
    \item $Q|_{\ker Z}$ is negative definite;
    \item $Q([E],[E]) \geq 0$ 
    for every $\sigma$-semistable object $E$.
\end{itemize}
    \end{enumerate}
\end{Def}
In this paper, we will only consider numerical stability conditions and will omit the term numerical for simplicity. 
\begin{Not}[Phase and slicing] \label{notion:slicing}
Let $\sigma=(\cA, Z)$ be a stability condition  on $\TT$. The stability function $Z$ is also called the \emph{central charge} of the stability condition.

For a non-zero object $E \in \cA$, its \emph{phase} is defined as
\[\phi_\sigma(E)= 
\begin{cases}
\frac{1}{\pi}\text{Arg}(Z(E)) & \text{if } \Im Z(E)>0, \\
1 & \text{otherwise}. 
\end{cases}
\]
If an object $F=E[n]$ for some $E \in \cA$ and $n\in \Z$, then we define its phase as $\phi_\sigma(F)=\phi_\sigma(E)+n$. 

We can associate $\sigma$ with a \emph{slicing} \[\cP_\sigma\colon \R\to \{\,\text{full additive subcategories in }\cT\,\}\] on $\cT$ as follows. For $\theta\in\R$,
\begin{enumerate}
\item if $\theta \in (0,1]$, the subcategory $\PP_\sigma(\theta)$ is the union of the zero object and all $\sigma$-semistable objects with phase $\theta$;
\item otherwise, set $\PP_\sigma(\theta):=\PP_\sigma(\theta-n)[n]$ for $\theta-n \in (0,1]$ and $n \in \Z$.
\end{enumerate}
 For every non-zero object $E\in \cT$, there is a unique filtration as that in \eqref{eq:hnfil} with $A_i\in\cP_\sigma(\theta_i)$ and $\theta_1>\dots>\theta_m$. The objects $A_i$'s are called the \emph{Harder--Narasimhan factors} of $E$. We denote by $\HN^+_\sigma(E):=A_1$ (resp. $\HN^-_\sigma(E):=A_m$) the Harder--Narasimhan factor with the largest (resp. smallest) phase.
 
 For an interval $I \subset \R$, we use $\cP_\sigma(I)$ to denote the extension-closed subcategory of $\TT$ generated by the subcategories $\PP_\sigma(\theta)$ with $\theta \in I$. In particular, it is clear from the definition that $\PP_\sigma((0, 1])= \cA$. 

The support property implies that $\PP_\sigma(\theta)$ has finite length for every $\theta \in \R$.\footnote{See \cite[Lemma A.4]{BMS:stabCY3s} for the equivalent definitions of support property. Definition A.2 in \cite{BMS:stabCY3s} implies the slice is Artinian directly.} In particular, every object $E \in \PP_\sigma(\theta)$ admits a (non-unique) finite filtration with $\sigma$-stable factors  of the same phase $\theta$, which are called \emph{Jordan--H\"older factors}.
\end{Not}

\begin{Def}[Global dimension]\label{def:gd} 
Let $\sigma$ be a stability condition on $\TT$. We define its \emph{global dimension}  as 
\begin{equation}\label{def:gldimvalue}
\gd (\sigma) := \sup\{ \theta_2-\theta_1 \mid
    \Hom(E_1,E_2)\neq0 \text{ for some } E_i\in\cP(\theta_i)\}.
\end{equation}
\end{Def}

\begin{Not}[Moduli space]\label{not:moduli}
Let $\sigma=(\cA, Z)$ be a stability condition on $\TT$. For $v \in \mathrm{K}_{\mathrm{num}}(\TT)$, and $\theta\in \R$ satisfying $Z(v)\in\R_{>0}\cdot e^{\pi i\theta}$, we  consider the functor 
\[\MM_\sigma(\TT, (v,\theta)) \colon (\mathrm{Sch})^{\mathrm{op}} \to \mathrm{Gpd}\]
from the category of schemes over $\C$ to the category of groupoids. This functor associates to a scheme $S$ the groupoid $\MM_\sigma(\TT, (v,\theta))(S)$ of all perfect complexes $E \in \text{D}(X \times S)$, such that for every $s \in S$, the restriction $E_s$ of $E$ to the fiber $X \times \lbrace s \rbrace$ is in $\cP_\sigma(\theta)$, in other words, $\sigma$-semistable of phase $\theta$, and with $[E_s]=v$. 

In the examples we will consider in this paper, the functor $\MM_{\sigma}(\TT, (v,\theta))$ admits a good moduli space $M_\sigma(\TT, (v,\theta))$, in the sense of \cite{Alper:goodmoduli}, which is a proper algebraic space over $\C$. We will denote by $M_\sigma^s(\TT, (v,\theta))$ the locus of classes of $\sigma$-stable objects in $M_\sigma(\TT, (v,\theta))$. Whenever the category $\cT$ and the phase $\theta$ are clear from the context, we will drop them in the notation and denote the moduli space (resp. the stable locus) as $M_\sigma(v)$ (resp. $\ms(v)$).   
\end{Not} 

\begin{Rem}[Stability manifold and $\glt$-action]\label{not:glt}
Denote by $\text{Stab}(\TT)$ the set of stability conditions on $\TT$. By Bridgeland's Deformation Theorem \cite{Bridgeland:Stab}, the set $\text{Stab}(\TT)$ (given that it is non-empty) admits a complex manifold structure of dimension equal to the rank of $\mathrm{K}_{\mathrm{num}}(\TT)$.

Denote by $\glr:=\{M\in\mathrm{GL}_2(\mathbb R)\;|\; \det(M)>0\}$, and let $\glt$ be the universal cover of $\glr$. We have the following right group action of $\glt$ on $\Stab(\TT)$.  Given $\widetilde{g}=(g,M) \in \glt$ with $M \in \glr$ and $g\colon  \mathbb{R} \to \mathbb{R}$ an increasing function with $g(\phi +1 )= g(\phi) +1$, the action on $\sigma = (\mathcal{P}_\sigma((0,1]), Z) \in \Stab(\TT)$ is given by 
\[\sigma \cdot \widetilde{g} = (\mathcal{P}_\sigma((g(0), g(1)]), M^{-1} \circ Z).\]
In particular, stability conditions $\sigma$ and $\sigma \cdot \widetilde{g}$ have the same set of (semi)stable objects (with possibly different phases). Their moduli spaces $\ms(v)$ and $M_{\sigma\cdot \tilde g}^s(v)$ are isomorphic to each other.   
\end{Rem}

\subsection{Harder--Narasimhan factors of the extension of stable objects}
The aim of this subsection is to prove the key Lemma \ref{lem:extendstabobj} saying that the extension between two stable objects with the ``smallest possible gap'' is stable. To make sense of this, we need to assume that there are no other characters in the triangle spanned by these two characters. 
\begin{Asp} [Discrete central charge]  \label{asp2}
 In this paper, we will always assume that  $\cT$ is a full admissible subcategory of  $\Db(X)=\Db(\Coh(X))$, where $X$ is a smooth projective variety defined over the field of complex numbers $\C$. We will always assume that the image of the central charge is discrete.  More precisely, the central charge \[Z\colon\kn(\cT)\xrightarrow{\lambda}\Lambda\hookrightarrow \C\] factors via a rank $2$ lattice $\Lambda$.
\end{Asp}

\begin{Not}[Basic notion on lattice points]
In this section, we fix an isomorphism $\Lambda\cong \Z^{\oplus 2}$. On $\Z^{\oplus2}$ we define the norm $|\bullet|$ as $|(m,n)|:=\sqrt{m^2+n^2}$, and the cross-product as $(a,b)\times (c,d):= ad-bc$. We further assume that the orientation on $\Lambda$ induced by the counterclockwise orientation on $\C$ is compatible via the isomorphism with the orientation on $\Z^{\oplus2}$ induced by the cross-product.
\end{Not}

\begin{Not}[Lattice points in a triangle]\label{not:triangleandpara}
    For $v,w\in \Lambda$ satisfying $v\times w>0$, we denote by  $\triangle(v,w)$ the set of all lattice points in the triangle spanned by $v$ and $w$. To be precise, \begin{align*}
        \triangle(v,w)&:=\{av+bw\; |\; a,b\geq 0, a+b\leq 1]\}\cap \Lambda.
    \end{align*}
\end{Not}

\begin{Def}[Extension object]\label{def:extobj}
    Let $v,w\in \Lambda$ satisfying $v\times w>0$. Let $E_v$ and $E_w$ be two $\sigma$-semistable objects such that $\lambda(E_v)=v$, $\lambda(E_w)=w$, and $\phi_\sigma(E_w)-\phi_\sigma(E_v)\in(0,1)$.  For $0\neq f\in\Hom(E_w,E_v[1])$, we define the extension object \[E_f:=\Cone(E_w\xrightarrow[]{f}E_v[1])[-1].\]
\end{Def}

The following easy lemma provides an essential tool to control the HN factors of extensions of stable objects.

\begin{Lem}\label{lem:hnfactorextobj}
In the setup of Definition~\ref{def:extobj}, we have
\[\lambda(\HN^+_\sigma(E_f))\in \triangle(v+w,w) \text{ and } \lambda(\HN^-_\sigma(E_f))\in \triangle(v,v+w).\]
\end{Lem}

\begin{proof}
    Denote by $E_1,\dots, E_m$ the Harder--Narasimhan factors of $E_f$ with phases from high to low. Then $\phi_\sigma(E_i)\in[\phi_\sigma(E_v),\phi_\sigma(E_w)]$. As $0< \phi_\sigma(E_w)-\phi_\sigma(E_v)<1$, the lattice points $\lambda(E_i)=a_i v +b_i w$ for some $a_i,b_i\geq 0$. 

     Because $v+w=\lambda(E_f)=\sum\lambda(E_i)$, we have $a_i,b_i\leq 1$. 
     As \[\phi_\sigma(E_w)\geq \phi_\sigma(E_1)\geq\phi_\sigma(v+w)\geq \phi_\sigma(E_m\geq \phi_\sigma(E_v),\] where we informally write $\phi_\sigma(v+w)$ for the value in $[\phi_\sigma(E_v),\phi_\sigma(E_w)]$ satisfying $Z(v+w)\in\R_{>0}\cdot e^{\pi i\phi_\sigma(v+w)}$. The statement follows.
\end{proof}

Using this lemma, we have the following stability result for extensions of two stable objects, when their characters are ``closest to each other''.

\begin{Lem}\label{lem:extendstabobj}
In the setup of Definition~\ref{def:extobj}, we further assume that $v\times w=1$. Then $E_f$ is always $\sigma$-stable.
\end{Lem}
\begin{proof}
    As $v\times w=1$, $v\times(v+w)=1$. In particular, the lattice point $\lambda(E_f)=v+w$ is primitive. So we only need to show that $E_f$ is $\sigma$-semistable.

    Suppose that $E_f$ is not $\sigma$-semistable.  By Lemma \ref{lem:hnfactorextobj}, there are only two Harder--Narasimhan factors $E_1$ and $E_2$ of $E_f$, with characters $\lambda(E_1)=w$ and $\lambda(E_2)=v$. 
    
    Applying $\Hom(E_1,-)$ to the distinguished triangle $E_v\to E_f\to E_w\xrightarrow[]{f}E_v[1]$, we get \begin{align}\label{eq21}
        0=\Hom(E_1,E_v)\to \Hom(E_1,E_f)\to \Hom(E_1,E_w)\xrightarrow[]{f\circ -}\Hom(E_1,E_v[1]).
    \end{align}
    Because $E_1$ is the first Harder--Narasimhan factor of $E_f$, we have $\Hom(E_1,E_f)\neq 0$. By \eqref{eq21}, we have $\Hom(E_1,E_w)\neq 0$. As $w$ is primitive and $E_w$ is $\sigma$-stable,  this can only happen when $E_1\cong E_w$. However, as $f\neq 0$, the map $f\circ - $ is injective from $\Hom(E_1,E_w)=\C$ to $\Hom(E_1,E_v[1])$. It follows that $\Hom(E_1,E_f)= 0$, which leads to a contradiction. Therefore, $E_f$ is $\sigma$-semistable, hence stable.
\end{proof}

\subsection{Existence of stable objects: a general result} \label{sec22}
In this subsection, we apply Lemma \ref{lem:extendstabobj} to prove Proposition \ref{prop:existencegeneralcase} saying that under certain assumptions, the non-emptiness of $M_\sigma(v)$ with small dimensions implies the non-emptiness of other $M_\sigma(v)$.

We review the following form of Pick's Theorem which will be essential in the inductive proof for Proposition \ref{prop:existencegeneralcase}. 
\begin{PropDef}[Pick's Theorem]\label{pd:pick}
    For every primitive vector $v\in\Z^{\oplus 2}$ with $|v|>1$, there exists a unique pair of vectors $v_\pm$ satisfying:
\[v_-\times v=v\times v_+=v_-\times v_+=1,\text{ and }|v_\pm|<|v|
.\]
This makes the pair satisfy $v=v_-+v_+$. We denote by $\delta(v)$ the difference between the `phases' of $v_\pm$. More precisely,   \[\delta(v):=\tfrac{1}{\pi}\arg(v_-,v_+)=\tfrac{1}{\pi}\arcsin{\frac{1}{|v_+||v_-|}}\in(0,\tfrac{1}{2}].\] 
\end{PropDef}

\begin{proof}
For every primitive vector $v=(n,m)\in\Z^{\oplus 2}$ with $m\geq 2$, one may write $\frac{n}{m}$ as a continued fraction 
\[[a_0,a_1,\dots,a_i]:=a_0+\frac{1}{a_1+\frac{1}{\ddots+\frac{1}{a_i}}},\]
where $i\geq 1$, $a_j\in\Z$ for all $0\leq j\leq i$, and $a_i \geq 2$. Then it is easy to see that the rational number $\tfrac{n_1}{m_1}:=[a_0,a_1,\dots,a_{i-1}]$ satisfies $mn_1-m_1n=(-1)^i$, and the rational number $\tfrac{n_2}{m_2}:=[a_0,a_1,\dots,a_i-1]=\tfrac{n-n_1}{m-m_1}$ satisfies $mn_2-m_2n=(-1)^{i+1}$. Also $m=m_1+m_2$, $n=n_1+n_2$, and $|(m,n)|>|(m_j,n_j)|$ for $j=1,2$. We may define $v_\pm=(n_1,m_1)$ or $(n_2,m_2)$ accordingly so that $v_-\times v=v\times v_+=1$.

For $v=(n,1)$ with $n> 0$, we may define $v_+:=(n-1,1)$ and $v_-:=(1,0)$; for $v=(n,1)$ with $n< 0$, we may define $v_+:=(-1,0)$ and $v_-:=(n+1,1)$. For primitive $v=(n,m)\neq(0,-1)$ with $m\leq -1$, we may define $v_+:=-(-v)_+$ and $v_-:=-(-v)_-$. \\

For the uniqueness, assume that another pair $v'_\pm$ satisfies the assumption, then as $(v'_--v_-)\times v=v'_-\times v-v_-\times v=0$, the vector $v'_-=v_-+av$ for some $a\in\Z$. As $|v'_-|<|v|$ and $\arg(v_-,v)<\pi$, the vector $v'_-$ can only be $v_-$ or $v_--v=-v_+$. 

Similarly, the vector $v'_+$ can only be $v_+$ or $v_+-v=-v_-$. As $v'_-\times v'_+=1$, we must have $v'_\pm=v_\pm$.
\end{proof}

\begin{Prop}\label{prop:existencegeneralcase}
Let $\sigma$ be a stability condition on a $\C$-linear triangulated category $\cT$ satisfying Assumption \ref{asp2}. Let $S_0\subset S\subset \kn(\cT)$ be non-empty subsets such that 
\begin{enumerate}
    \item the moduli space $M^s_\sigma(v)\neq \emptyset$ for any $v\in S_0$;
    \item for any $v\in S\setminus S_0$, $\lambda(v)$ is primitive, $|\lambda(v)|>1$, and $\gd(\sigma)<3-\delta(\lambda(v))$;
    \item for any $v\in S\setminus S_0$, there exist $v_1,v_2\in S$ such that $v=v_1+v_2$, $\lambda(v_1)=\lambda(v)_+$, $\lambda(v_2)=\lambda(v)_-$, and $\chi(v_1,v_2)<0$.
\end{enumerate}
    Then the moduli space $M^s_\sigma(v)\neq \emptyset$ for any $v\in S$.
\end{Prop}
\begin{proof}
    Suppose that there exists $v\in S$ such that $M^s_\sigma(v)=\emptyset$. As the values of $|\lambda(-)|$ are discrete, we may assume that $v$ has the smallest $|\lambda(-)|$ among all such vectors. 

    By \emph{(a)}, $v\notin S_0$. Let $v_1$ and $v_2$ be the two characters as in assumption \emph{(c)}. In particular, $|\lambda(v_i)|<|\lambda(v)|$. By the minimality assumption on $|\lambda(v)|$, $M^s_\sigma(v_i)\neq \emptyset$. 
    
    Let $E_i$ be an object in $M^s_\sigma(v_i)$, then by \emph{(b)}, $\Hom(E_1,E_2[m])= 0$ when $m\geq 3$ or $\leq -1$. It follows that \[0>\chi(v_1,v_2)=\hom(E_1,E_2)-\hom(E_1,E_2[1])+\hom(E_1,E_2[2]).\]
    Hence, $\hom(E_1,E_2[1])\neq 0$. Choosing any $0\neq f\in \Hom(E_1,E_2[1])$, by Lemma \ref{lem:extendstabobj}, the object $E_f:= \Cone(E_1\xrightarrow[]{f}E_2[1])[-1]$ is $\sigma$-stable with character $v$, which contradicts $M^s_\sigma(v)=\emptyset$.

    Therefore, $M^s_\sigma(v)\neq \emptyset$ for any $v\in S$.
\end{proof}

To illustrate how to apply this result, we use it to prove the following classical statement.
\begin{Ex}[Non-emptiness of the moduli space in the curve case]\label{eg:nonempDbC}
    Let $C$ be a smooth projective curve with genus $g\geq 1$. Let $\sigma=(\Coh(C),Z=-\deg+i\rk)$ be a stability condition on $\Db(C)$, which coincides with slope stability for coherent sheaves. Then for every primitive character $v\in \kn(C)$, the moduli space $M^s_\sigma(v)\neq \emptyset$. 
\end{Ex}
\begin{proof}
   Apply Proposition \ref{prop:existencegeneralcase} by letting $S_0=\{(\rk,\deg)\mid (\rk,\deg)=(0,\pm 1)\text{ or }|\rk|=1\}$ and $S=\{$all primitive characters in $\kn(C)\}$. Assumption \emph{(a)} and \emph{(b)} hold automatically. For every $v=(r,d)\in S\setminus S_0$, we may assume that $r\geq 2$. Choose $v_i=(r_i,d_i)$ accordingly, then $\frac{d_1}{r_1}>\frac{d_2}{r_2}$. By Riemann--Roch, we have $\chi(v_1,v_2)=(r_1d_2-r_2d_1)-(g-1)r_1r_2<0$. So Assumption \emph{(c)} holds. By Proposition \ref{prop:existencegeneralcase}, the statement holds.
\end{proof}

For the main cases we consider in this paper, the Kuznetsov components of Fano threefolds, the following extra assumptions are satisfied.

\begin{Asp}\label{rem:1}
     In addition to Assumption \ref{asp2}, we assume that the stability condition $\sigma$ satisfies the following conditions:
    \begin{enumerate}
        \item $\gd(\sigma)<\tfrac{5}{2}$;
        \item $\lambda\colon\kn(\cT)\xrightarrow{\cong}\Lambda \cong \Z^{\oplus 2}$ and we will identify  numerical characters as lattice points in $\Z^{\oplus 2}$;
        \item $\chi(v,v)\leq 0$ for every $v\in\Kn(\cT)$;
        \item there exists $N_{\chi}\in\Z$, depending on the Euler form $\chi$ on $\Kn(\cT)$, such that for every primitive $v\in\Kn(\cT)$ with $\chi(v,v)< -N_\chi$ and $|v|>1$, one has $\chi(v_+,v_-)<0$.
    \end{enumerate} 
\end{Asp}

The computation of $N_\chi$ for the Kuznetsov components of various Fano threefolds is included in Appendix \ref{app:bounds}. As we will see in Section \ref{sec:nonempty_fano3}, the following corollary reduces the non-emptiness problem to that of small dimensional moduli spaces.

\begin{Cor} \label{cor:existence}
    Let $\sigma$ be a stability condition on $\cT$ satisfying Assumption \ref{asp2} and \ref{rem:1}. Assume that $M^s_\sigma(v)\neq \emptyset$ for every primitive character $v$ satisfying either $\chi(v,v)\geq -N_\chi$ or $|v|=1$. Then the moduli space $M^s_\sigma(v)\neq \emptyset$ for every primitive character $v$. 
\end{Cor}
\begin{proof}
 Apply Proposition \ref{prop:existencegeneralcase} by letting $S$ be the set of all primitive characters and $S_0$ be the subset of primitive characters with $\chi(v,v)\geq -N_\chi$ or $|\lambda(v)|=1$. Note that the condition $\gd(\sigma)<3-\delta(\lambda(v))$ in assumption ({\em b}) holds automatically when $\gd(\sigma)<\tfrac{5}{2}$. It follows by Proposition \ref{prop:existencegeneralcase} that $M^s_\sigma(v)\neq \emptyset$ for every primitive $v$.
\end{proof}

Before ending this section, we note the following application on the stability manifolds.

\begin{Cor}\label{cor:stab}
    In the setting of Corollary \ref{cor:existence}, the set of stability conditions $\sigma\cdot\glt$ forms a connected component of the stability manifold $\Stab(\cT)$.   
\end{Cor}
\begin{proof}
   As the dimension of $\Stab(\cT)$ is $\rk(\Kn(\cT))$, which is $2$ by Assumption \ref{rem:1}.(b), the subset $\sigma\cdot\glt$ is open in $\Stab(\cT)$.  Suppose that $\sigma\cdot\glt$ is not closed, and let $\sigma'=(Z',\cP')$ be a stability condition on its boundary. Then its central charge $Z'$ is degenerate, in other words, the image of $Z'$ is contained in a real line in $\C$. In this case, either $\ker Z'$ contains a primitive $v\in\Kn(\cT)$, or  $0\in\overline{\{Z'(\Kn(\cT))\}\setminus \{0\}}$. Note that for every primitive $v\in\Kn(\cT)$, the space $\ms(v)\neq \emptyset$. It follows that $M_{\sigma'}(v)\neq \emptyset$. Therefore, the central charge $Z'$ cannot satisfy the support property. This leads to the contradiction. So $ \sigma\cdot\glt$ is a connected component of the stability manifold.
\end{proof}

\section{Kuznetsov components of Fano threefolds}\label{sec3}
In this section, we apply the non-emptiness criteria in the previous section to the Kuznetsov component of a Fano threefold, proving Theorem \ref{thm:existintro}, Corollary \ref{cor:stabmfdofku_intro} and Theorem \ref{thm:Y3moduli}.(1).

\subsection{Review: Stability conditions on Kuznetsov components of Fano threefolds}\label{sec:kuz}

\begin{Def}[Semiorthogonal decomposition]
Let $\cT$ be a $\C$-linear triangulated category. A \textit{semiorthogonal decomposition} (SOD) for $\TT$, denoted by $\TT = \langle \TT_1, \dots, \TT_m \rangle$, is a sequence of full triangulated subcategories $\TT_1, \dots, \TT_m$ of $\TT$ such that: 
	\begin{enumerate}
		\item $\Hom_{\TT}(E, F) = 0$, for all $E \in \TT_i$, $F \in \TT_j$ and $i>j$;	\item For any non-zero object $E \in \TT$, there is a sequence of morphisms
		\begin{equation*}  
		0 = E_m \to E_{m-1} \to \cdots \to E_1 \to E_0 = E,
		\end{equation*}
		such that $\mathrm{Cone}(E_i \to E_{i-1}) \in \TT_i$ for $1 \leq i \leq m$.  
	\end{enumerate}
\end{Def}

\begin{Def}[Exceptional collection]
An object $E\in\TT$ is \textit{exceptional} if $\Hom_{\TT}(E,E[k])=0$ for all integers $k \neq0$, and $\Hom_{\TT}(E,E)= \C$. An \textit{exceptional collection} is a collection of exceptional objects $E_1,\dots,E_m$ in $\TT$ such that $\Hom_{\TT}(E_i,E_j[k])=0$ for all $k\in\Z$ and $1\leq j<i\leq m$.
\end{Def}

Assume that $\TT$ is a proper $\C$-linear triangulated category, in other words, for every $A$, $B \in \TT$ the vector space $\oplus_i \Hom(A, B[i])$ is finite-dimensional. Given an exceptional collection $E_1,\dots,E_m$ in $\TT$, we have the semiorthogonal decomposition
\[\TT=\langle \cK, E_1,\dots,E_m \rangle,\]
where $\cK:=\langle E_1,\dots,E_m \rangle^{\perp}= \lbrace F \in \TT\;|\; \Hom_{\TT}(E_i, F[k])=0  \text{ for all $k\in\Z$ and } i=1, \dots, m\rbrace$.\\

\noindent Let $X$ be a smooth Fano threefold of Picard rank one. Let $H$ be the primitive ample divisor. In particular, the Picard group $\mathrm {Pic}(X)=\Z H$. Denote by $K_X$ the canonical divisor of $X$.  The \emph{index} $i_X$ is the integer satisfying $K_X=-i_XH$. The \emph{degree} of $X$ is defined as $d_X:=H^3$.

Deformation types of Fano threefolds of Picard rank one are classified by their index and degree \cite{Is, MU}. More precisely, the index $i_X\in\{1,2,3,4\}$. When $i_X=4$, the threefold is the projective space; when $i_X=3$, the threefold is the quadric. In each of these two cases, the category  $\Db(X)$ admits a full exceptional collection.\\

\noindent When $i_X=2$, there are five deformation types classified by $d_X\in\{1,2,3,4,5\}$. The Kuznetsov component has a straightforward definition in this case.
\begin{Def}[Kuznetsov component for index $2$ case]\label{def:kuzind2}
    Let $X$ be a smooth Fano threefold with $\Pic(X)=\Z H$ and $K_X=-2H$. The \emph{Kuznetsov component} $\Ku(X)$ is defined by the semiorthogonal decomposition
    \[\Db(X)=\langle \Ku(X),\cO_X,\cO_X(H)\rangle.\]
\end{Def}

\noindent When $i_X=1$, there are ten deformation types classified by $d_X\in \{2,4,6,8,10,12,14, 16,18,22\}$. The degree can be rewritten as $d_X=2g_X-2$. Here $g_X$ is called the \emph{genus} of $X$. In this case, the definition of the Kuznetsov component relies on the existence of certain exceptional vector bundles. 

\begin{Thm}[{\cite[Theorem 6.2, Proposition and Definition 6.3]{BLMS:kuzcomponent}}, \cite{BKM}]\label{thm:KuofF3}
    Let $X$ be a Fano threefold of Picard rank $1$, index $1$,
and even genus $g\geq 6$ (equivalently, $d\in\{10,14,18,22\}$) over any algebraically closed field. Then there exists a stable vector bundle $\cE_2$ on $X$ of rank $2$, with $\ch_1(\cE_2)=-H$ and $\ch_2(\cE_2)=(\tfrac{g}{2}-2)L$, where $L$ is the class of a line on $X$. The pair $(\cE_2,\cO_X)$ is exceptional, and the \emph{Kuznetsov component} $\Ku(X)$ of $X$ is defined by the semiorthogonal decomposition \[\Db(X)=\langle \Ku(X),\cE_2,\cO_X\rangle.\]
\end{Thm}

When  the genus of $X$ is $9$ (resp. $7$), there exists a rank $3$ (resp. $5$) stable bundle $\cE_3$ (resp. $\cE_5$) such that $(\cE_3,\cO_X)$ (resp. $(\cE_5,\cO_X)$) is exceptional. The \emph{Kuznetsov component} $\Ku(X)$ of $X$ is defined as $\langle \cE_*,\cO_X\rangle^\perp$. By \cite[Section 6]{Kuznetsov:Hyperplane}, the triangulated category $\Ku(X)$ is equivalent to $\Db(C_3)$ (resp. $\Db(C_7)$) for some smooth projective curve $C_g$ with genus $g$.\\

Stability conditions have been constructed on these Kuznetsov components. Here we only recall the following theorem, some details of the construction can be found in Appendix \ref{app:small_moduli}.

\begin{Thm}[{\cite[Theorem 1.1]{BLMS:kuzcomponent}}]\label{thm:stabonku}
    Let $X$ be a smooth Fano threefold over an algebraic closed field. Assume that $X$ is of index two, or index one and genus $\geq 6$, then its Kuznetsov component $\Ku(X)$ admits a Bridgeland stability condition $\sigma=(\cA,Z)$ with central charge satisfying Assumption \ref{asp2}.
\end{Thm}
\begin{Rem}[Kuznetsov component in the low genus case]\label{rem:kuoflowgenusfano3}
    When the Fano threefold is of index $1$ and $g\leq 5$ (in other words, $d\in\{2,4,6,8\}$), the choice of the Kuznetsov component is more complicated. In some cases, there exists a stable exceptional vector bundle $\cE$ in $\cO_X^\perp$, but it is unknown if there exists a stability condition on $\langle \cE,\cO\rangle^\perp$. Indeed one can show that there is no Serre invariant stability condition in these cases \cite[Section 1.4]{KuzPerry_serre}.
    
    Alternatively, in these cases the Kuznetsov components can be simply defined as $\cO_X^\perp$, and \cite[Theorem 1.1]{BLMS:kuzcomponent} also constructs stability conditions on these categories. Here the non-emptiness problem for the moduli spaces is more complicated, and we leave it for a future project.
\end{Rem}

\subsection{Non-emptiness of the moduli spaces}\label{sec:nonempty_fano3}
 
\begin{Thm}\label{thm:existinKurank2}
    Let $X$ be a smooth Fano threefold with Picard rank $1$, index $1$, and genus $\in \{6,7,8,9,10\}$ or index $2$, and degree $\leq 4$. Then for every nonzero character $v\in\kn(\Ku(X))$, the moduli space $M_\sigma(v)$ of $\sigma$-semistable objects of class $v$ is non-empty.
\end{Thm}

The proof of Theorem \ref{thm:existinKurank2} consists in checking that the stability condition $\sigma$ satisfies the conditions in Assumption \ref{rem:1} and Corollary \ref{cor:existence}. All assumptions are straightforward to check, except for the bound $N_\chi$ and some of the non-emptiness of $\ms(v)$ for $v$ with $\chi(v,v)\geq -N_\chi$, which follows from a case-by-case study. The computation of $N_\chi$ is completely elementary and is included in Appendix \ref{app:bounds} for completeness. The non-emptiness of several small dimensional moduli spaces that is not covered in the literature is included in Appendix \ref{app:small_moduli}.

\begin{proof}[Proof of Theorem \ref{thm:existinKurank2}]
    First note that it is enough to show the non-emptiness for primitive $v$. Indeed, if $\ms(v)\neq \emptyset$, then direct sums of objects in $\ms(v)$ gives (strictly) semistable objects in $M_\sigma(mv)$ for $m>1$.\\

    When the Fano threefold is of index $2$ degree $4$ (resp. index $1$ genus $7,9,10$), the Kuznetsov component is known to be equivalent to $\Db(C_g)$ for some curve with genus $g=2$ (resp. $7,3,2$) \cite[Tables in Section 6]{BLMS:kuzcomponent}. By \cite{Macri:curves}, the Bridgeland stability condition $\sigma$ on $\Db(C_g)$, up to a $\glt$-action which does not affect the moduli space, is the slope stability condition. The statement follows from the classical result or Example \ref{eg:nonempDbC}. \\

    When the Fano threefold is of index $2$ and degree $3$ (or index $1$ and genus $8$), the Serre functor $\mathsf S$ on the Kuznetsov component satisfies $\mathsf S^3=[5]$. The stability condition $\sigma$ is Serre invariant by \cite[Proposition 5.7]{PY}, so it has global dimension $<2$.  We apply Lemma \ref{lem:22matrix} by letting $Q$ be the Euler pairing $\chi(-,-)$ and $D$ to be the isometry given by the Serre functor. As $\mathsf S^6=[10]$, which acts trivially on $\kn(\Ku(X))$, the Euler pairing is of the form  $I_\pm$ as that in Lemma \ref{lem:22matrix}. As $\phi_\sigma(\mathsf S \bx)-\phi_\sigma(\bx)\in(1,2)$, the ordered set $\{\bx,-\mathsf S\bx\}$ forms a right-hand oriented $\Z$-linear basis of $\Lambda$. So the Euler form is of the form $I_+$.
    
    By Lemma \ref{lem:22matrix} and Corollary \ref{cor:existence}, we only need to check that  for all primitive character $v\in \kn(\Ku(X))$ with $\chi(v,v)\geq  -1$, the space $M^s_\sigma(v)\neq \emptyset$. The only case is when $\chi(v,v)=  -1$. The Serre  functor $\mathsf S$ permutes all such $v$'s. As the ideal sheaf of a line is $\sigma$-stable \cite{PY}, we have $M^s_\sigma(v)\neq \emptyset$. The statement follows in this case.\\

    When the Fano threefold is of index $2$ and degree $2$ (or index $1$ and genus $6$), the Serre functor $\mathsf S$ on the Kuznetsov component satisfies $\mathsf S^2=[4]$. The stability condition $\sigma$ is Serre invariant by \cite[Proposition 5.7]{PY} and \cite{PR}, so it is with global dimension $2$. By \cite[Section 3]{Kuz:Fano3fold}, the Euler pairing under certain basis is of the form $\begin{pmatrix}
        -1 & -1 \\ -1 & -2
    \end{pmatrix}$. So there exist $\Z$-linear independent $v_1, v_2\in\kn(\Ku(X))$ with $\chi(v_i,v_j)=-\delta_{ij}$. Using $\{v_1,v_2\}$ as basis of $\kn(\Ku(X))$, it is clear that for every primitive character $v$ with $\chi(v,v)<-2$, we have $\chi(v_+,v_-)=-v_+\cdot v_-<0$. 
    
    By  Corollary \ref{cor:existence}, we only need to check that  for all primitive $v\in \kn(\Ku(X))$ with $\chi(v,v)\geq  -2$, the space $M^s_\sigma(v)\neq \emptyset$. This follows from the more general results in \cite{PPZ_Enriques}, but in fact the non-emptiness for these small dimensional moduli spaces is known by previous works. 
    
    More precisely, if $X$ has index $2$ and degree $2$, then the non-emptiness of $\ms(v)$ when $\chi(v,v)=-1$, resp.\ $\chi(v,v)=-2$, follows from \cite[Proposition 4.1]{PY}, resp.\ \cite[Proposition 4.7]{APR}, together with the fact that the functor $\mathbb{L}_{\cO_X}(- \otimes \cO_X(H))$ preserves the stability by \cite[Proposition 5.7]{PY}. 
    
    If $X$ has index $1$ and genus $6$, then $\ms(v) \neq \emptyset$ when $\chi(v,v)=-1$ by \cite[Proposition 7.11, Theorem 8.9]{JLLZ}. Moreover, there are two classes in $\kn(\Ku(X))$ satisfying $\chi(v,v)=-2$ with Chern character 
    \[1 - \frac{3}{10}H^2+ \frac{1}{2}P \quad \text{ and } \quad 3-2H+\frac{3}{10}H^2+\frac{7}{6}P.\]
    In the first case, the non-emptiness of the corresponding moduli space follows from \cite[Lemma A.7]{JLLZ}, using that this class corresponds to the class of the ideal sheaf of a general twisted cubic curve in $X$. The remaining case is discussed in Lemma \ref{lemma_lastcaseGM3}.\\

     When the Fano threefold is of index $2$ and degree $1$, the Serre functor $\mathsf S$ on the Kuznetsov component satisfies $\mathsf S^3=[7]$. The stability condition $\sigma$ is Serre invariant \cite[Proposition 5.7]{PY}, so it is with global dimension $<\tfrac{5}{2}$.  By Lemma \ref{lem:22matrix}, the Euler pairing is of the form  $I_\pm$. As $\phi_\sigma(\mathsf S \bx)-\phi_\sigma(\bx)\in(2,3)$, the ordered set $\{\bx,\mathsf S\bx\}$ forms a right-hand oriented $\Z$-linear basis of $\Lambda$. So the Euler form is of the form $I_-$.
    
    By Lemma \ref{lem:22matrix} and Corollary \ref{cor:existence}, we only need to check that  for all primitive character $v\in \kn(\Ku(X))$ with $\chi(v,v)\geq  -3$, the space $M^s_\sigma(v)\neq \emptyset$. The only cases are when $\chi(v,v)=  -1$ or $-3$. The functors $\mathsf S$ permutes all $v$ with $\chi(v,v)=-1$. As the ideal sheaf of a line is $\sigma$-stable, the space $M^s_\sigma(v)\neq \emptyset$ for every $v$ with $\chi(v,v)=-1$. The functor $\mathsf S$ permutes all $v$ with $\chi(v,v)=-3$. Thus the statement follows from Lemma \ref{lemma_lastcasedVc}.
\end{proof}

Before ending this section, we point out an immediate corollary on the stability manifold of $\Ku(X)$. 

\begin{Cor}\label{cor:stabmfdofku}
    Let $X$ be a Fano threefold with Picard rank $1$, index $1$, genus $\in \{6,7,8,9,10\}$ or index $2$, degree $\leq 4$. Then the set of stability conditions $\sigma\cdot\glt$ forms a connected component of the stability manifold $\Stab(\Ku(X))$. 
\end{Cor}

\begin{proof}
    As shown in the proof of Theorem \ref{thm:existinKurank2}, the conditions of Corollary \ref{cor:stab} are satisfied. Now the result follows directly.
\end{proof}

Among the cases studied in Theorem \ref{thm:existinKurank2}, the most interesting one is that of cubic threefolds (index $2$, degree $3$), which will be the focus of our study from Section \ref{sec:moduliY3}.

\begin{Cor}\label{cor:cub3smoothmoduli}
    Let $X$ be a smooth cubic threefold or a smooth Fano threefold with index $1$ and genus $8$. Then for every primitive character $v\in\kn(\Ku(X))$, the moduli space $M^s_\sigma(v)$ is a smooth irreducible projective fine moduli space of dimension $1-\chi(v,v)$.
\end{Cor}

\begin{proof}
    By \cite{Kuz:V14}, it is enough to show this for cubic threefolds. Essentially the only new statement is the non-emptiness of the moduli spaces, which is covered by Theorem \ref{thm:existinKurank2}. The remaining properties follow from standard arguments and we only briefly explain them. 
    
    Smoothness is proved in \cite[Theorem 1.2]{PY}. For the projectivity, there exists a Bayer--Macr\`i divisor on $M^s_\sigma(v)$, which is positive on each effective curve \cite{BM:projectivity}. Now by \cite[Corollary 3.4]{villalobos:proj_moishezon}, the space $M^s_\sigma(v)$ is indeed projective.
    
    To see that  $M^s_\sigma(v)$ is a fine moduli space, it is enough to observe that, by elementary calculation, for every primitive $v$ there exists a character $w$ such that $\chi(w,v)=1$. Take an object $E_w$ of character $w$, and consider the object $\mathscr E:=p_{1,*}\mathscr{H}om(p_2^*(E_w), \mathscr U)$ on the moduli stack $\mathscr M^s_\sigma(v)$. Here $\mathscr U$ denotes the universal object on $\mathscr M^s_\sigma(v) \times X$ and $p_i$'s denote the projections to the two factors. Note that $\mathscr U$ defines an $\alpha$-twisted object on $M^s_\sigma(v)\times X$ of a certain Brauer class $\alpha$, and $\mathscr E$ defines an $\alpha$-twisted object on $M^s_\sigma(v)$ of rank $\chi(w,v)=1$. As the rank of $\mathscr E$ is a multiple of the order of $\alpha$, we see that $\alpha$ is trivial and we have a universal family.
    
    Finally irreducibility follows from a standard ``Mukai's trick", see for example \cite[Proof of Theorem 2.12, Step 3]{LMS}.
\end{proof}

\section{Brill--Noether locus and extension locus} \label{sec4}

In this section we define rigorously the loci introduced in Section \ref{subsec:BNlociandEloci} of the introduction and we prove their relation as in \eqref{eq:relation}. Then, under suitable assumptions, we define the projectivized relative $\Ext^1$, denoted by $\Pe vw$, over the locus in $\ms(v) \times \ms(w)$ of pairs of objects $(E_v,E_w)$ such that $\Hom(E_w,E_v[i])=0$ for every $i \neq 1$, and construct birational maps among different $\Pe vw$'s. As a consequence, we compute in Corollary \ref{cor:codimofETlocus} the codimension of the extension locus, under suitable assumptions.

\subsection{Brill--Noether locus and extension locus}

Let $\sigma$ be a numerical stability condition on $\cT$ satisfying Assumption \ref{asp2}.

In this subsection, we introduce some notation and general results for the ``jumping locus'' of a moduli space. We are mainly interested in two types of jumping loci in this paper. The first type is where two stable objects with very close phases have unexpected $\Hom$'s. We call this the Brill--Noether locus, as this is in the same spirit as the Brill--Noether locus on the Jacobian of a curve. The second type is where an object is an extension of two stable objects with very close phases. Our main goal is to show that these two types of loci are closely related (Proposition \ref{prop:imofprimpliesext}). The estimate of the dimension of the Brill--Noether locus is the key to the following sections.

\begin{Rem}[Lifted numerical Grothendieck group]\label{rem:Knum}
We define $\Kn(\cT)^*_\sigma$, the \emph{lifted numerical Grothendieck group} of $\cT$, by 
\[\Kn(\cT)^*_\sigma:=\{(v,a)\;|\; Z_\sigma(v)\in\R_{>0}\cdot e^{a\pi i}\}\subset \Kn(\cT)\times \R.\]
We denote by $\phi_\sigma((v,a))=a$. For $n\in \Z$, we denote by $(v,a)[n]:=((-1)^n v,a+n)$. 

To simplify the notation, we will omit $\lambda$ in Assumption \ref{asp2} when there is no confusion. Operations such as $v+w$ and $\triangle(v,w)$ make sense when $|\phi_\sigma(v)-\phi_\sigma(w)|<1$. More precisely, $\triangle(v,w)$ stands for the characters in $\Kn(\cT)^*_\sigma$ with phase in $[\phi_\sigma(w),\phi_\sigma(v)]$ and image of $\lambda$ in the triangle of $v$ and $w$ as that in Notation \ref{not:triangleandpara}.

We will also write $\Kn(\cT)$ (resp. $v[n]$) for $\Kn(\cT)^*_\sigma$ (resp. $((-1)^nv,a+n)$) throughout the rest of the paper when there is no confusion.
\end{Rem}

\begin{Def}[Brill--Noether locus]\label{def:BNlocus}
Let $v,w\in\kn(\cT)$ with $\phi_\sigma(w)>\phi_\sigma(v)$. We define the \textit{Brill--Noether locus} $\BN(v,w)$ in $M^s_\sigma(v)\times M^s_\sigma(w)$ as follows:
\[\BN(v,w):=\{(E_v,E_w)\in M^s_\sigma(v)\times M^s_\sigma(w)\;|\;\Hom(E_v,E_w)\neq 0\}\subset M^s_\sigma(v)\times M^s_\sigma(w).\]
We denote by $\BNP_w$ (resp. $\BNP_v$) the projection from $\BN(v,w)$ to $M^s_\sigma(w)$ (resp. $M^s_\sigma(v)$). For a positive integer $r$ and an object $E\in \cT$, we introduce the following notation:
\begin{align*}
    &\BN^r(E,w):=\{E_w\in  M^s_\sigma(w)\;|\;\hom(E,E_w)\geq r\}\subset M^s_\sigma(w); \\
    &\BN^r(v,E):=\{E_v\in  M^s_\sigma(v)\;|\;\hom(E_v,E)\geq r\}\subset M^s_\sigma(v);\\
    &\BN^{=r}(v,E):= \BN^{r}(v,E)\setminus \BN^{r+1}(v,E).
\end{align*}
\end{Def}
\begin{Def}[Extension locus] \label{def:ETlocus}
When $\phi_\sigma(w)-\phi_\sigma(v)\in (0,1)$, we denote by $\ET(v,w)$ the sublocus of objects in $M^s_\sigma(v+w)$ that are extended by objects in $M^s_\sigma(v)$ and $M^s_\sigma(w)$, in other words,
\[\ET(v,w):=\left\{E\in M^s_\sigma(v+w)\;\middle|\;\begin{aligned}
    E\cong\Cone(E_w[-1]\to E_v) \text{ for some }\\ E_v\in M^s_\sigma(v) \text{ and } E_w\in M^s_\sigma(w)
\end{aligned} \right\}\subset M^s_\sigma(v+w).\]

For every $i\geq 0$, we will denote by $\ET^i(v,w)$ the sublocus of $\ET(v,w)$ where $E_v$ and $E_w$ further satisfy $\hom(E_w,E_v[1])=-\chi(w,v)+i$.  For subloci $A\subset \ms(v)$ and $B\subset \ms(w)$, we will denote by $\ET(A,B)$ the sublocus  of $\ET(v,w)$ where $E_v\in A$ and $E_w\in B$.
\end{Def}

The next proposition says that the Brill--Noether locus $\BN(E,w)$ consists of objects that can be extended from $E$ or objects with a smaller gap of phases. In practice, this allows us to establish dimension bound on the Brill--Noether locus via an estimate of the dimensions of the extension loci. This method is the key to our stable birationality result.

\begin{Prop}\label{prop:imofprimpliesext}
    Assume that $\phi_\sigma(w)-\phi_\sigma(v)\in(0,1)$, $E_w\in \ms(w)$, and $\BN(v,E_w)\neq \emptyset$. Then  there exists $v'\in\triangle(v,w)\setminus\{\R w\}$ such that $E_w\in\ET(v',w-v')$. 
    
    More precisely, let $E_v\in\ms(v)$, then
    \begin{align}\label{eq:bnet}
         \BN(E_v,w)\subseteq \ET(E_v,w-v)\cup \left(\bigcup_{v'\in\triangle(v,w)\setminus\{\R w,v\}}\ET(v',w-v')\right).
    \end{align}
In particular, if $\overline{\BNP_w(\BN(v,w))}=M^s_\sigma(w)$, then \[M^s_\sigma(w)=\bigcup_{v'\in\triangle(v,w)\setminus\{\R w\}}\ET(v',w-v').\]
A similar formula as that in \eqref{eq:bnet} holds for $\BN(v,E_w)$.
\end{Prop}

\begin{proof}
    For all $u$ in $\triangle(v,w)\setminus\{\R w\}$ satisfying $\BN(u,E_w)\neq \emptyset$, let $u_0$ be the one with minimum $u\times w$, which is the area of the parallelogram spanned by $u$ and $w$. In particular, for every $u'$ in $\triangle(u_0,w)\setminus\{u_0,w\}$, the object $E_w\not\in \BNP_w(\BN(u',w))$. We claim that $E_w\in\ET(u_0,w-u_0)$, which reduces to the following Lemma \ref{lem:cokerstable}.
\begin{Lem}\label{lem:cokerstable}
    Let $0\neq f\in\Hom(E_v,E_w)$ for some $E_v\in M^s_\sigma(v)$, $E_w\in M^s_\sigma(w)$ and $\phi_\sigma(E_w)-\phi_\sigma(E_v)\in(0,1)$. Assume that $\BN(u,E_w)=\emptyset$ for every $u\in \triangle(v,w)\setminus\{v,w\}$, in other words, \begin{align}
        \Hom(E_{u},E_w)=0 \label{eq:41}
    \end{align}
    for every $E_{u}\in M^s_\sigma(u)$. Then the object $E_f:=\Cone(E_v\xrightarrow{f}E_w)$ is $\sigma$-stable.

    Similarly, if $\BN(E_v,u)=\emptyset$ for every $u\in \triangle(v,w)\setminus\{v,w\}$, then $E_f$ is $\sigma$-stable.
\end{Lem}
\begin{proof}[Proof of Lemma \ref{lem:cokerstable}]
    Suppose that $E_f$ is not $\sigma$-stable, then it is either strictly $\sigma$-semistable or unstable. In the first case, we may choose a Jordan--H\"older factor $E_u$ of $E_f$ so that $\Hom(E_f,E_u)\neq 0$. In the second case, by Lemma \ref{lem:hnfactorextobj}, the character of  $\HN^-_\sigma(E_f)$ is in $\triangle(w,w-v)\setminus\{\R(w-v)\}$. As $f\neq 0$, the character of  $\HN^-_\sigma(E_f)$ is not in $\R w$. By passing to a Jordan--H\"older factor we can further assume that this object is $\sigma$-stable.

    In either case, there exists a $\sigma$-stable object $E_u$ with phase $\phi_\sigma(E_u)\in(\phi_\sigma(E_w),\phi_\sigma(E_{f})]$ and character $u\in \triangle(w,w-v)\setminus\{w-v, \R w\}$ such that $\Hom(E_f,E_u)\neq 0$. In particular, \[w\times u<w\times (w-v)=v\times w.\]
    Also, applying $\Hom(-,E_u)$ to the distinguished triangle
    \[E_{v}\xrightarrow[]{f}E_w \to E_{f} \to E_{v}[1],\]
   it follows that $\Hom(E_w,E_u)\neq 0$ since $\phi_\sigma(E_{v}[1])>\phi_\sigma(E_f)$. 
    
    Let $0\neq g\in\Hom(E_w,E_u)$ and denote by $E_g:=\Cone(E_w\xrightarrow{g} E_u)[-1]$. Then the character of $E_g$ is $w-u$, which is in $\triangle(v,w)\setminus\{\R v,\R w\}$. Note that $\Hom(E_g,E_w)\neq 0$ as $E_u$ is indecomposable. By Assumption \eqref{eq:41}, the object $E_g$ cannot be $\sigma$-stable. So it is either strictly $\sigma$-semistable or unstable. As $g\neq 0$, by a similar argument as that for $E_f$, in either case, there exists a $\sigma$-stable object $E_{v'}$ with phase $\phi_\sigma(E_{v'})\in[\phi_\sigma(E_g),\phi_\sigma(E_{w}))$ and character \[v'\in \triangle(w-u,w)\setminus\{w-u, \R w\}\subset \triangle(v,w)\setminus\{v,\R w\}\] such that $\Hom(E_{v'},E_g)\neq 0$. Similarly, apply $\Hom(E_{v'},-)$ to the distinguished triangle
    \[E_u[-1]\to E_g\to E_w \xrightarrow{g}E_u.\]
    It follows that $\Hom(E_{v'},E_w)\neq 0$. This leads to a contradiction to Assumption \eqref{eq:41}. So the object $E_f$ must be $\sigma$-stable.
\end{proof}
  Back to the proof of Proposition \ref{prop:imofprimpliesext}, the first statement follows immediately by taking $u_0=v$ and applying Lemma \ref{lem:cokerstable} to the object $\Cone(E_{u_0}\to E_w)$. 
  
  The second statement is just rephrasing the first statement with the term for $v'=v$ singled out.
  
    For the last statement, if $E_w\in \BNP_w(\BN(v,w))$, then it follows from the first statement. If $E_w\in \overline{\BNP_w(\BN(v,w))}\setminus {\BNP_w(\BN(v,w))}$, by a semi-continuity argument, there exists a strictly semistable object $E_v$ such that $\Hom(E_v,E_w)\neq 0$. Then $v=mv_0$ must be non-primitive, and by passing to a Jordan--H\"older factor of $E_v$, we see that there exists some $m_0<m$ such that $E_w\in \BNP_w(\BN(m_0v_0,w))$. Note that $\triangle(m_0v_0,w)\subset \triangle(v,w)$, the statement follows.
\end{proof}
\subsection{Birational map between fiber bundles over moduli spaces}
\label{sec4.2}
Throughout this subsection, we will make the following assumption:
\begin{Asp}\label{asspdag}
Assume that $\gd (\sigma)\leq 2$. Let $v,w\in\Kn(\cT)$, further assume that
\begin{align*}
    \phi_\sigma(w)- \phi_\sigma(v)\in(0,1),\; \chi(w,v)<0\text{ and }\BN(w,v[2])\neq \ms(w)\times \ms(v[2]). 
\end{align*}
The notation $v[2]$ is explained in Remark \ref{rem:Knum}.
\end{Asp}

\begin{Def}[Non-jumping locus]\label{def:non-jumping}
Under Assumption \ref{asspdag}, we define  $\ms(v,w)^\dag$ by
\[\ms(v,w)^\dag:=\{(E_v,E_w)\in M^s_\sigma(v)\times M^s_\sigma(w)\;|\;\Hom(E_w,E_v[i])=0 \text{ when } i\neq 1\}.\]
\end{Def}

In other words, the space $\ms(v,w)^\dag$ is the complement of \[\{(E_v,E_w)\;|\;\Hom(E_w,E_v[2])\neq 0\}\]
in $M^s_\sigma(v)\times M^s_\sigma(w)$.

We denote by $\mathscr U_v$ the universal object on the moduli stack $\mathscr M^s_\sigma(v)\times \cT$, and $\mathscr U_w$ similarly. Now consider the object on $\mathscr M^s_\sigma(v)\times \mathscr M^s_\sigma(w)$ given by
\[\mathscr E:=p_{12,*}\mathscr{H}om(p_{23}^*\mathscr U_w, p_{13}^*\mathscr U_v)[1].\]
By definition, $\mathscr E$ defines a twisted object over $\ms(v)\times \ms(w)$, of a possibly nontrivial Brauer class. We denote by $\cE^\dag$ the restriction of this twisted object to $\ms(v,w)^\dag$. In particular, $\cE^\dag$ is a twisted locally free sheaf of rank $-\chi(w,v)$.

\begin{Def}[Projectivized relative $\Ext^1$]\label{def:proj_nj}
    We denote by $\Pe vw$ the projectivization of $\cE^\dag$ defined above over $\ms(v,w)^\dag$. 

   Denote by $\pi_v$ and $\pi_w$ the natural projections from $\Pe vw$ to $\ms(v)$ and $\ms(w)$ respectively. For every $E\in \ms(v)$ and $F\in\ms(w)$, we denote by $\Pe Ew$ (resp. $\Pe vF$) the fiber $\pi^{-1}_v(E)$ (resp. $\pi^{-1}_w(F)$). For every $i\geq 0$, we denote by $\Pe{\BN^{=i}(E[-2],w)}E$ the projectivization of the relative $\Ext^1(E,-)$ over $\BN^{=i}(E[-2],w)\subset \ms(w)$. In particular, when $i=0$, we have $\Pe{\BN^{=0}(E[-2],w)}E =\Pe Ew$. 
\end{Def}

Note that $\Pe vw$ naturally has the structure of a Brauer--Severi variety over $\ms(v,w)^\dag$.

Recall that for an element $(E_v,E_w,f)\in \Pe vw$, the extension object is denoted as $E_f:=\Cone(E_w\xrightarrow{f}E_v[1])[-1]$. The following property is straightforward.

\begin{Lem}\label{lem:openextobjstable}
The subset in $\Pe vw$ where $E_f$ is $\sigma$-stable is open.
\end{Lem}

\begin{Def}[Extension map]\label{def:evw}
We define a rational map
    \begin{align*}
     e_{v,w}\colon\Pe vw&\dashrightarrow \ET^0(v,w)\subset M^s_\sigma(v+w) \\ (E_v,E_w,f)&\mapsto  E_f:=\Cone(E_w\xrightarrow{f}E_v[1])[-1],\end{align*}
\noindent whenever the image $E_f$ is $\sigma$-stable.
\end{Def}

Recall from Definition \ref{def:ETlocus} that  for $i\in\Z_{\geq0}$, the space $\ET^i(v,w)$ is the sublocus of $\ET(v,w)$ where $E_v$ and $E_w$ further satisfy $\hom(E_w,E_v[1])=-\chi(w,v)+i$. Under Assumption \ref{asspdag}, this is exactly when $\hom(E_v,E_w[2])=i$. In particular, the space $\ET^0(v,w)$ consists of extensions from $(E_v,E_w)$ satisfying $\Hom(E_v,E_w[2])=0$, in other words, $(E_v,E_w)\in\ms(v,w)^\dag$.
It is clear from the definition that the image of $e_{v,w}$ is equal to $\ET^0(v,w)$ whenever non-empty.

Finally, we have the following rational maps which are crucial for our stable birationality statement.

\begin{Def}\label{def:evwR}
We further assume that Assumption \ref{asspdag} holds for the pair $(v+w,v[1])$, or explicitly, \begin{align*}
    \chi(v,v+w)>0\text{ and }\BN(v,(v+w)[1])\neq \ms(v)\times \ms((v+w)[1]).
\end{align*}  Equivalently, it follows that the space $\Pe{v+w}{v[1]}$ is non-empty. We define a rational map
\begin{align*}
    e^R_{v,w}\colon\Pe vw&\dashrightarrow \Pe{v+w}{v[1]} \\ (E_v,E_w,f)&\mapsto  (E_f,E_v[1],f^R)\end{align*}
whenever $(E_f,E_v[1])\in\ms(v+w,v[1])^\dag$. The morphism $f^R\in\Hom(E_v[1],E_f[1])$ is given by the distinguished triangle \[E_w\xrightarrow{f}E_v[1]\xrightarrow{f^R}E_{f}[1]\xrightarrow[]{+}.\] Note that all objects are simple in the distinguished triangle, hence the morphism $f^R$ is uniquely determined up to a scalar.

Similarly, further assuming that Assumption \ref{asspdag} holds for the pair $(w[-1],v+w)$, in other words, \begin{align*}
    \chi(v+w,w)>0\text{ and } \BN(v+w,w[1])\neq \ms(v+w)\times \ms(w[1]),
\end{align*}  we define the rational map
\begin{align*}
    e^L_{v,w}\colon\Pe vw&\dashrightarrow \Pe{w[-1]}{v+w} \\ (E_v,E_w,f)&\mapsto  (E_w[-1],E_f,f^L)\end{align*}
whenever $(E_w[-1],E_f)\in \ms(w[-1],v+w)^\dag$. The morphism $f^L\in\Hom(E_f,E_w)$ is given (uniquely up to a nonzero scalar) by the distinguished triangle \[E_f\xrightarrow{f^L}E_w\xrightarrow{f}E_v[1]\xrightarrow{+}.\]
\end{Def}

For a given $E_v\in M^s_\sigma(v)$ (resp. $E_w\in M^s_\sigma(w)$), we define $e_{E_v,w}$ (resp. $e_{v,E_w}$) as the restriction of the map $e_{v,w}$ to the fiber $\Pe{E_v}w$ (resp. $\Pe {v}{E_w}$). We define $e^R_{E_v,w}$ and $e^L_{v,E_w}$ similarly. 

\begin{Rem}\label{rem:projectivization}
Note that $\Pe vw$ is only defined over the non-jumping locus $\ms(v,w)^\dag$. This turns out to be the most convenient and most suitable notation for our purpose. Under this definition, in order to verify that the rational map $e^R_{v,w}$ is well-defined on an open set, we need to show that:
\begin{enumerate*}
    \item The pairs $(v,w)$ and $(v+w,v[1])$ satisfy Assumption \ref{asspdag}. In other words, the spaces $\Pe vw$ and $\Pe{v+w}{v[1]}$ are non-empty.
    \item For a general $(E_v,E_w,f)\in \Pe vw$, the extended object $E_f$ is stable.
    \item Moreover, a general image point $(E_f,E_v[1])$ is in $\ms(v+w,v[1])^\dag$, in other words, $\Hom(E_v,E_f[1])=0$.
\end{enumerate*}
\end{Rem}

The following proposition gives a criterion on when the rational map $e^R_{E_v,w}$ is indeed birational.

\begin{Prop}\label{prop:extbirpair}
    Let $w,v\in\Kn(\cT)$ with  $\phi_\sigma(w)-\phi_\sigma(v)\in(0,1)$, and $E_v\in\ms(v)$. Assume that $e^R_{E_v,w}$  is well-defined at a point, then  $e^R_{E_v,w}$ is a birational equivalence from the irreducible component containing the point to the irreducible component containing its image, with the inverse map given by $e^L_{v+w,E_v[1]}$. More specifically, the map $e^R_{E_v,w}$ gives an isomorphism between dense open subsets of the well-defined loci.
    
    Similar statement holds for $E_w\in\ms(w)$ and the pair of maps $e^L_{v,E_w}$ and $e^R_{E_w[-1],v+w}$.
 \end{Prop}
 \begin{proof}
  As  $e^R_{v,w}$ is well-defined,  Assumption \ref{asspdag} holds for both the pairs $(v+w,v[1])$ and $(v,w)$. So the spaces $\Pe{v+w}{v[1]}$ and  $\Pe vw$ are non-empty.

     The map $e^L_{v+w,v[1]}$ can be defined from $\Pe{v+w}{v[1]}$ to $\Pe vw$.
     
     Assume that $e^R_{v,w}$ is well-defined on the point $(E_v,E_w,f)$ with the image point $(E_f,E_v[1],f^R)$ in $\Pe{v+w}{v[1]}$. Then $e^R_{E_v,w}$ is well-defined. We have the distinguished triangle:
     \[E_f\xrightarrow{}E_w\xrightarrow{f}E_v[1]\xrightarrow{f^R}E_f[1].\]
     By definition, we have $e^L_{v+w,E_v[1]}(E_f,E_v[1],f^R)=(E_v,E_w,f)$. As $(E_v,E_w)\in \ms(v,w)^\dag$, the map $e^L_{v+w,E_v[1]}$ is well-defined at $(E_f,E_v[1],f^R)$.
     
     By Lemma \ref{lem:openextobjstable}, both maps $e^R_{E_v,w}$ and $e^L_{v+w,E_v[1]}$ are well-defined in non-empty  open subsets. Whenever well-defined, we have $e^L_{v+w,E_v[1]}\circ e^R_{E_v,w}=\id$ and $e^R_{E_v,w}\circ e^L_{v+w,E_v[1]}=\id$ by definition. So they induce a birational equivalence between the two irreducible components.
     \end{proof}

Finally, we have the following result. This together with Corollary \ref{cor:codimofETlocus} will be used to control the dimension of extension loci.

\begin{Prop}\label{prop:extmapfinite}
    Assume that $\phi_\sigma(w)-\phi_\sigma(v)\in(0,1)$ and $\BN(v,w)$ does not contain any irreducible component of $\ms(v)\times \ms(w)$. Then $e_{v,w}$ is generically finite whenever well-defined. 
 \end{Prop}
 \begin{proof}
     Let $(E_v,E_w,f)$ be a general point in $\Pe vw$. As $\BN(v,w)$ does not contain any irreducible component of $\ms(v)\times \ms(w)$, we may assume that
     \begin{align}
         \Hom(E_v,E_w)= 0. \label{eq42}
     \end{align}
     Suppose that an irreducible curve $C$ through the point $(E_v,E_w,f)$ is contracted by $e_{v,w}$, then for every point $(E_v',E_w',f')$ on $C$, there is an isomorphism $g$ between $E_f$ and $E_f'$:
     \begin{equation}\label{eqdiagram4}
	\begin{tikzcd}
		E_w[-1]\arrow{r}{f[1]}\arrow[d,dashed]{}{d} & E_v \arrow{r}{a} \arrow[d,dashed]{}{c} & E_f
		\arrow{d}{g} \arrow{r}{b} & E_w  \arrow{r}{f} & E_v[1]\\
		E_w'[-1]\arrow{r}{f'[1]} & E_v' \arrow{r}{a'} & E_{f'}
		 \arrow{r}{b'} & E_w'  \arrow{r}{f'} & E_v'[1].
	\end{tikzcd}
\end{equation}
If $E_v'\not\cong E_v$ for a general $E_v'$ on $C$, then $\Hom(E_v,E_v')=0$.  Apply $\Hom(E_v,-)$ to the distinguished triangle $E_v' \to  E_f
		 \to E_w'  \to E_v'[1]$, we get $\Hom(E_v,E_w')\neq 0$ for a general $E_w'$. By the semi-continuity of $\Hom(E_v,-)$ on the moduli space $\ms(w)$, we have $\Hom(E_v,E_w)\neq 0$, which contradicts the assumption \eqref{eq42}. 

   It follows that $E_v'\cong E_v$ for every point on $C$. By the same argument, $E_w'\cong E_w$ for every point on $C$. So the curve $C$ must parameterize a family of morphisms in $\Hom(E_w,E_v[1])$.\\

   As $b' \circ g\circ a\in\Hom(E_v,E_w)=0$ by \eqref{eq42}, there exists $c\in \Hom(E_v,E_v)$ such that $g\circ a = a'\circ c$. As $g$ is an isomorphism and $a\neq 0$, we have $c\neq 0$. It follows that the morphism $c$ is an isomorphism as well. Similarly there exists an isomorphism $d\in \Hom(E_w[-1],E_w[-1])$ such that $f'[1]\circ d = c\circ f[1]$.
   
   Note that both $E_v$ and $E_w$ are $\sigma$-stable, the isomorphisms $c$ and $d$ are just identities up to a scalar. As $c\circ f[1]=f'[1]\circ d$, we have $f=f'$ up to a scalar. Hence the curve $C$ consists of just one point $(E_v,E_w,f)$ in $\Pe vw$. 
   
   So the map $e_{v,w}$ does not contract any curve through a general point $(E_v,E_w,f)$ satisfying \eqref{eq42}. On its well-defined locus, the map $e_{v,w}$ is generically finite.
 \end{proof}

 \begin{Cor}\label{cor:codimofETlocus}
  Let $v,w\in\Kn(\cT)$ such that $\phi_\sigma(w)-\phi_\sigma(v)\in(0,1)$. Assume that
  \begin{enumerate*}
      \item the moduli spaces $M^s_\sigma(u)$ is of pure dimension $1-\chi(u,u)$ for $u=v,w,v+w$;
      \item the space $\BN(v,w)$ does not contain any irreducible component of $\ms(v)\times \ms(w)$; and
      \item the space $\BN(w,v[i])=\emptyset$ for every $i\geq 2$.
  \end{enumerate*} 
 \noindent Then the sublocus  $\ET(v,w)$, whenever non-empty, is of codimension $-\chi(v,w)$ in $\ms(v+w)$.
 \end{Cor}
 \begin{proof} Note that by (c), $\dim \Ext^1(E_w,E_v)=-\chi(w,v)$, so $\ET(v,w)=\ET^0(v,w)$ is in the image of $e_{v,w}$. By Proposition \ref{prop:extmapfinite}, the dimension of $\ET(v,w)$, whenever non-empty, is the same as that of $\Pe vw$.
 
      By a direct computation, the dimension of $\Pe vw$ is equal to \begin{align*}
         & \dim M^s_\sigma(v)+\dim M^s_\sigma(w)+\dim \Ext^1(E_w,E_v) -1 \\ 
         =&\ 1-\chi(v,v) +1-\chi(w,w)-\chi(w,v) -1 \\
         =&\ 1-\chi(v+w,v+w) +\chi(v,w)=\dim M^s_\sigma(v+w)-(-\chi(v,w)).
     \end{align*}
  The conclusion follows.
 \end{proof}

\section{Moduli spaces on the Kuznetsov components of cubic threefolds} \label{sec:moduliY3}
In the rest of the paper, we explore the geometry of higher dimensional moduli spaces on the Kuznetsov component of a smooth cubic threefold $Y_3$. 

\subsection{General properties of $\Ku(Y_3)$}

Let $Y_3$ be a smooth cubic threefold. Recall from Definition \ref{def:kuzind2} that $\Ku(Y_3)$ is the full triangulated subcategory consisting of objects right orthogonal to $\cO_{Y_3}$ and $\cO_{Y_3}(H)$: 
\begin{align*}
   \Ku(Y_3)=\{E\in\Db(Y_3)\;|\; \RHom(\cO_{Y_3},E)=\RHom(\cO_{Y_3}(H),E)=0\}.
\end{align*}

\begin{Not}[Serre functor on $\Ku(Y_3)$]\label{not:serrey3}
For an object $F\in \Ku(Y_3)$, the inverse of the Serre functor is given by \begin{align}\label{eq:serrefunctorY3}
    \mathsf S^{-1}_{\Ku(Y_3)}(F)=\mathbb L_{\cO}\mathbb L_{\cO(H)}(\mathsf S^{-1}_{Y_3}(F))= \mathbb L_{\cO}\mathbb L_{\cO(H)}(F\otimes \cO_{Y_3}(2H))[-3],
\end{align}
where we write $\mathbb L_E$ for the left mutation functor of the exceptional object $E$: for every object $G\in\Db(Y_3)$,
\[\mathbb L_E(G):=\Cone(E\otimes \mathrm{RHom}(E,G)\xrightarrow{\mathsf{ev}}G).\]

\noindent For simplicity we write $\mathsf S$ for $\mathsf S_{\Ku(Y_3)}$. We also write $\mathsf L$ for the functor $\mathsf L(-):=\mathbb L_{\cO}(-\otimes \cO(H))$. It is then clear from formula \eqref{eq:serrefunctorY3} that $\mathsf S^{-1}=\mathsf L^2[-3]$ and $\mathsf S=\mathsf L^{-2}[3]$. By \cite[Lemma 7.1.29]{Huybrechts:cubichypersurfacebook}, the functor $\mathsf L^3=[2]$. It follows that  $\mathsf S=\mathsf L[1]$. The category $\Ku(Y_3)$ is $\tfrac{5}{3}$-Calabi--Yau, in the sense that $\mathsf{S}^3=[5]$, see \cite{Kuznetsov:fracCY}.
\end{Not}
\begin{Not}[Characters in $\Kn(\Ku(Y_3))$]\label{not:abc}
   The ideal sheaf $\cI_\ell$ of a line is in $\Ku(Y_3)$.  Denote by $\beta\in \Kn(\Ku(Y_3))$  its numerical class. Denote by $\alpha=\mathsf S\beta[-2]$ and $\gamma=\mathsf S^{-1}\beta[2]=\mathsf S\alpha[1]$. It is clear that $\Kn(\Ku(Y_3))$ is generated by $\alpha$ and $\beta$. For the relation, we have $\alpha+\gamma=\beta$.

   More explicitly, one may interpret $\alpha$, $\beta$, and $\gamma$ as Chern characters in $\Kn(Y_3)$ as
\begin{align}\label{eqcharofabc}
    \alpha=(2,-H,-\tfrac{L}{2},\tfrac{P}{2}), \;\;\beta=(1,0,-L,0), \mbox{ and} \;\;\gamma=(-1,H,-\tfrac{L}{2},-\tfrac{P}{2}),
\end{align}
where $H$ stands for the class of a hyperplane section, $L$ for the class of a line, and $P$ for the class of a point.

It is worth recalling the Euler form on these characters:
\begin{align*}
\chi(\alpha,\alpha)=\chi(\beta,\alpha)=\chi(\gamma,\beta)=-1;\;\; \chi(\alpha,\beta)=\chi(\beta,\gamma)=\chi(\gamma,\alpha)=0; \;\;\chi(\alpha,\gamma)=1.\\
\end{align*}
\end{Not}

\begin{Not}[Hexagonal coordinate for $\Kn(\Ku(Y_3))$]
In this paper, we apply an action of $\glt$ on the stability conditions on $\Ku(Y_3)$ as that constructed in \cite{BLMS:kuzcomponent}, so that the central charge is of the form
\begin{align}
    Z(E)=e^{\frac{\pi i}{3}}\rk(E)+\tfrac{\sqrt 3}{3}i H^2\ch_1(E)\text{ and } \phi_\sigma(\cI_\ell)=\tfrac{1}{3}.
\end{align} 
Recall that the action of $\glt$ does not change the moduli spaces.

By doing so, the central charge maps $\Kn(\Ku(Y_3))$ to the integer lattice points under the hexagonal coordinates (see Figure \ref{fig:hex}). Compared with the usual Euclidean coordinates for which $\alpha$ and $\beta$ span the first quadrant, the lattice points in the two coordinate systems are interchanged by applying the transform $\begin{pmatrix} 1 & 0\\ \tfrac{1}{2} & 1\end{pmatrix}$.
Visualizing elements of $\Kn(\Ku(Y_3))$ under the hexagonal coordinate is more convenient for us to keep track of the phases of characters and the relations between different moduli spaces. 

 In particular, for every $0\neq v\in \Kn(\Ku(Y_3))$, we have 
\begin{align}\label{eq:phaseofserre}
\phi_\sigma(\mathsf Sv)-\phi_\sigma(v)=\tfrac{5}{3}; \\
    \phi_\sigma(\alpha)=0;\;\;\phi_\sigma(\beta)=\tfrac{1}{3};\;\;\phi_\sigma(\gamma)=\tfrac{2}{3}.
\end{align}   
\end{Not}
A crucial property of the Serre functor is that it preserves the stability of objects. This property was used in the proof of Theorem \ref{thm:existinKurank2}, and will appear several times in the next two sections of the paper. 

\begin{Prop}[{\cite[Proposition 5.7]{PY}}]\label{prop:serreinv}
    The stability condition $\sigma$ is Serre invariant. More precisely, $\mathsf S\sigma=\sigma[\tfrac{5}{3}]$. An object $E$ is $\sigma$-stable with phase $\theta$ if and only if $\mathsf S(E)$ is $\sigma$-stable with phase $\theta+\tfrac{5}{3}$.
\end{Prop}
\begin{figure}
\centering
\begin{tikzpicture}[line cap=round,line join=round,>=triangle 45,x=1.7cm,y=1.7cm]

\clip(-2,-0.3) rectangle (6.420,4.3);
\draw (0,0) node[anchor=north west] {O};
%\beta and \gamma axix
\draw [domain=0:2.2] plot(\x,{(-0--1.7240222753698782*\x)/1.0032964027047553}) node[above]{$\beta$-axis};
\draw [domain=-1.7:0] plot(\x,{(-0--1.732303518008477*\x)/-0.9759205879203182});
\draw (-1.3,3) node {$\gamma$-axis};

\draw (0.26098380957792494,3.73) node[anchor=north west] {$(m-1)\beta+\gamma$};
\draw (2,3.578561176063406) node[anchor=north west] {$m\beta$};
\draw [dashed] (1,3.422)-- (0,0);
\draw [dashed] (1,3.422)-- (2,3.437);
\draw [dashed] (2,3.437)-- (1,0);

\draw [dashed] (0,0)-- (1,0) node[below]{$\alpha$};
\draw (1,0)--(5.5,0) node[below]{$\alpha$-axis};
\draw [shift={(0,0)},line width=0.7pt]  plot[domain=0.3:0.66,variable=\t]({1*4.79*cos(\t r)+0*4.79*sin(\t r)},{0*4.79*cos(\t r)+1*4.79*sin(\t r)});
\draw (4.35,2.466) node[anchor=north west] {$\mathsf S^{-1}[2]$ rotates};
\draw (4.4,2.2) node[anchor=north west] {the lattices by $\frac{\pi}{3}$};

\begin{scriptsize}
\draw [fill=ududff] (0,0) circle (2.5pt);

\draw [fill=xdxdff] (1,0) circle (2.5pt);

\draw [fill=xdxdff] (2,0) circle (2.5pt);

\draw [fill=xdxdff] (3,0) circle (2.5pt);

\draw [fill=xdxdff] (4,0) circle (2.5pt);

\draw [fill=xdxdff] (5,0) circle (2.5pt);

\draw [fill=ududff] (1,1.724) circle (2.5pt);
\draw (1,1.9) node{Bl$_{F(Y_3)}J(Y_3)$};

\draw [fill=ududff] (-1,1.732) circle (2.5pt);

\draw [fill=xdxdff] (0.5,0.858) circle (2.5pt);
\draw (0.5,1)node{$\ms(\beta)\cong F(Y_3)$};
\draw [fill=ududff] (-2.516,0.854) circle (2.5pt);

\draw (1.5,0.7) node{Bl$_p\Theta$};
\draw [fill=xdxdff] (1.5,0.86) circle (2.5pt);

\draw (2,1.6) node{$\pi^{-1}\cong V_{14}$};
\draw (2.5,1) node{$\sim $Hilb$^{4,0}(Y_3)$};
\draw (3.8,1.55) node{$\sim $Hilb$^{10,6}(Y_3)$};
\draw [fill=xdxdff] (2.5,0.86) circle (2.5pt);

\draw [fill=xdxdff] (3.52,0.862) circle (2.5pt);

\draw [fill=xdxdff] (1.5,2.57) circle (2.5pt);

\draw [fill=ududff] (2.49,2.577) circle (2.5pt);

\draw [fill=ududff] (3.5,2.585) circle (2.5pt);

\draw [fill=xdxdff] (2,3.437) circle (2.5pt);

\draw [fill=xdxdff] (3,1.724) circle (2.5pt);

\draw [fill=xdxdff] (4.,1.724) circle (2.5pt);

\draw [fill=xdxdff] (4.5,0.864) circle (2.5pt);

\draw [fill=xdxdff] (2,1.724) circle (2.5pt);

\draw [fill=ududff] (1,3.422) circle (2.5pt);

\draw  (3.77,2.927) -- (3.77,2.8);
\draw  (3.77,2.927) -- (3.9,2.92);
\end{scriptsize}
\end{tikzpicture}
    \caption{Characters in $\Kn(\Ku(Y_3))$ under the hexagonal coordinate.}
    \label{fig:hex}
\end{figure}
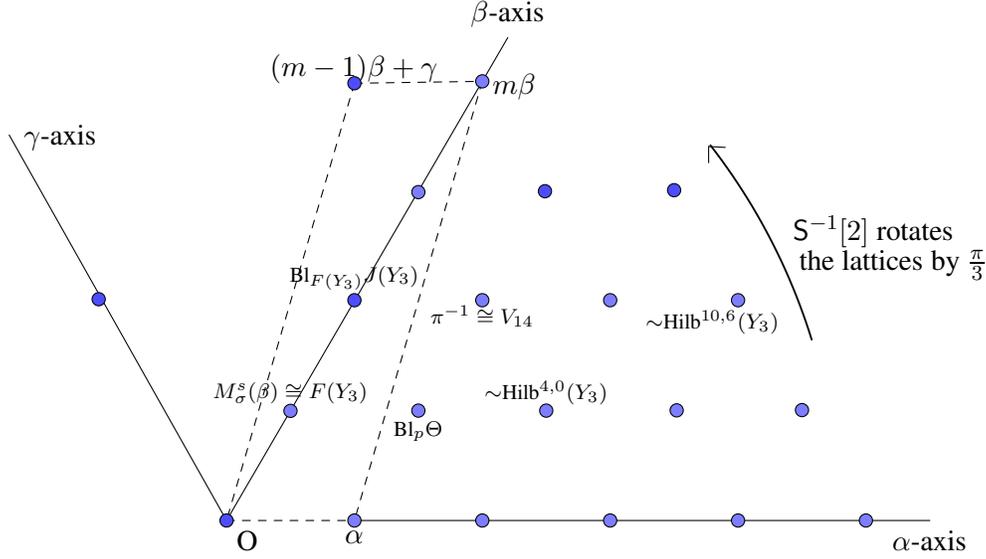

\begin{Rem}[Serre functor permutes the sextants]\label{rem:Sswapabc}
    Note that the action of the Serre functor $\mathsf S$ permutes the ``sextants" of $\Kn(\Ku(Y_3))$. More precisely, by Proposition \ref{prop:serreinv}, an object $E\in \ms(n\beta+m\gamma)$ for some $m,n\geq 0$ if and only if $\mathsf S(E)\in \ms((n\alpha+m\beta)[2])$. So if a statement for lattice points in one of the sextants holds, then it holds for all lattice points. For simplicity we usually state properties only for moduli space of characters in the $(\alpha,\beta)$-sextant, in other words, characters of the form $n\alpha+m\beta$ with $m,n\geq 0$.
\end{Rem}

\subsection{\AJ map}\label{sec:ajmap}

In this subsection, we recall the intermediate Jacobian and Abel--Jacobi map for a cubic threefold. This will be useful for our study of moduli spaces.

Recall that the intermediate Jacobian of a cubic threefold $Y_3$ is defined by $J(Y_3):= \frac{H^{2,1}(Y_3)^*}{H_3(Y_3,\mathbb Z)}$. This has the structure of an abelian fivefold and plays a central role in the proof of irrationality of cubic threefolds \cite{CG:RationalityCubics}. For a family of 1-cycles $Z\to T$ over a smooth irreducible base scheme $T$, we can define an \AJ map $\Phi\colon T\to J(Y_3)$, which is a morphism of schemes. Here $J(Y_3)$ is viewed as a subgroup of CH$_1(Y_3)$ of a given algebraic equivalence class.

For every $v\in \Kn(\Ku(Y_3))$, now we construct the \AJ map by the (cycle-theoretic) second Chern class:
\begin{align*}
    \IJ_v\colon M_\sigma(v)\to J_v(Y_3):
       F\mapsto c_2(F),
\end{align*}
where we use $J_v(Y_3)$ to denote the component of CH$_1(Y_3)$ receiving the image of $\Phi_v$. In fact this can be defined also when $v$ is non-primitive and $M_\sigma(v)$ is singular and does not admit a universal family: the moduli stack $\mathscr M_\sigma(v)$ is a smooth Artin stack. Choose a smooth presentation $\mathscr U\to \mathscr M_\sigma(v)$. Now $\mathscr U$ is a smooth scheme and by considering the pullback of the universal object, we can define the \AJ map $\mathscr U \to J_v(Y_3): F\mapsto c_2(F)$ as a morphism of schemes. This descends to a morphism $\mathscr M_\sigma(v) \to J_v(Y_3)$. Recall that $\mathscr M_\sigma(v) \to M_\sigma(v)$ is a good moduli space, which is universal for maps to algebraic spaces \cite[Theorem 6.6]{Alper:goodmoduli}. Hence this factors through a morphism $\IJ_v\colon M_\sigma(v)\to J_v(Y_3)$. For every $c\in J_v(Y_3)$, we denote by \[\ms(v,c):=\IJ_v^{-1}(c)\cap \ms(v)\]
the stable locus of the fiber of the \AJ map.

Note that the group structure on the Chow group gives us a map $+\colon J_v(Y_3)\times J_w(Y_3)\to J_{v+w}(Y_3)$. For $E_v\in \ms(v)$, $E_w\in\ms(w)$ and a $\sigma$-stable extension $E_f=\Cone(E_v\xrightarrow{f}E_w[1])[-1]$, we have $\IJ_v(E_v)+\IJ_w(E_w)=\IJ_{v+w}(E_f)$.

\subsection{Moduli spaces $\ms(v)$ of small dimensions}\label{sec:mssmall}

In this subsection, we recall the first several examples of moduli spaces on $\Ku(Y_3)$.

\begin{Ex}[Stable objects with character $\alpha$, $\beta$, and $\gamma$]   \label{eg511}
The moduli space $M^s_\sigma(\beta)$ consists of ideal sheaves $\cI_\ell$ of lines. The moduli spaces $\ms(\alpha)$, $\ms(\beta)$, and $\ms(\gamma)$ are identified by the Serre functor, or equivalently the functor $\mathsf L$ as that in Notation \ref{not:serrey3}. These moduli spaces are all isomorphic to the Fano variety of lines on $Y_3$, see \cite[Theorem 1.1]{PY} and \cite[Theorem 1.3]{FP}.

The moduli space $M^s_\sigma(\alpha)$ consists of rank two vector bundles $\mathsf L(\cI_\ell)[-1]$. More precisely, every object $\cE_\ell$ in $M^s_\sigma(\alpha)$ fits in a short exact sequence:
\begin{align}\label{eq415}
    0\to\cE_\ell\to \cO\otimes \Hom(\cO,\cI_\ell(H))\xrightarrow{\mathrm{ev}}\cI_\ell(H)\to 0.
\end{align}
The moduli space $\ms(\gamma)$ consists of objects $\cF_\ell:=\Cone(\cO_\ell(-H)[-1]\xrightarrow{\ev}\cO(-H)[1])$.

We may compute the Brill--Noether locus $\BN(\alpha,\beta)$ as follows. Apply $\Hom(-,\cI_{\ell'})$ to \eqref{eq415}, then 
\begin{align}\hom(\cE_\ell,\cI_{\ell'})=\hom(\cI_\ell(H),\cI_{\ell'}[1])=\begin{cases}
    0 & \text{when }\ell'\cap \ell=\emptyset; \\
    1 & \text{when }\ell', \ell \text{ intersect at a point}; \\
    2 & \text{when }\ell'= \ell. \\
\end{cases}
  \label{eq4111}  
\end{align}

\noindent In particular, $\BN(\alpha,\beta)$ is of codimension one in the four-dimensional space $\ms(\alpha)\times\ms(\beta)$.     
\end{Ex}

\begin{Ex}[Stable objects with character $\beta+\gamma$, see {\cite[Theorem 7.1]{Arend:cubic3}} for more details]\label{eg:ab}
    Denote by $U_\beta$ and $U_\gamma$ the universal family on $\ms(\beta)$ and $\ms(\gamma)$. Consider the relative $\Ext^1$ sheaf H$^1(p_{12,*}\mathscr{H}om(p_{23}^* U_\gamma, p_{13}^* U_\beta))$, and we denote its projectivization over $\ms(\beta)\times \ms(\gamma)$ by $\tilde\P_\sigma(\beta,\gamma)$. Note that \[\hom(E_\gamma,E_\beta[1])=\begin{cases}
        2 & \text{when } E_\beta=\mathsf L(E_\gamma)[-1];\\
        1 & \text{otherwise.}
    \end{cases}\]
    So the natural map $\pi_{\beta,\gamma}\colon\tilde\P_\sigma(\beta,\gamma) \to \ms(\beta)\times \ms(\gamma)$ is one-to-one on general points and has a $\P^1$-fiber on the diagonal $\Delta$ of $\ms(\beta)\times \ms(\gamma)$, while identifying $\ms(\gamma)$ and $\ms(\beta)$ via $\mathsf L[-1]$. As each irreducible component is with at least the expected dimension, the space $\tilde \P_\sigma(\beta,\gamma)$ is irreducible. As a variety, $\tilde \P_\sigma(\beta,\gamma)$ is isomorphic to $\mathrm{Bl}_\Delta(F(Y_3)\times F(Y_3))$.
    
     By Lemma \ref{lem:extendstabobj}, for every $E_\beta\in\ms(\beta)$, $E_\gamma\in\ms(\gamma)$, and $0\neq f\in\Hom(E_\gamma,E_\beta[1])$, the object $E_f:=\Cone(E_\gamma\xrightarrow{f}E_\beta[1])[-1]$ is $\sigma$-stable. Hence we get a well-defined morphism $\tilde{e}_{\beta,\gamma}$    from $\tilde\P_\sigma(\beta,\gamma)$ to $\ms(\beta+\gamma)$.  This morphism extends the map $e_{\beta,\gamma}$ to the jumping locus of the $\Ext^1$ group, or equivalently, the exceptional locus over the diagonal $\Delta$. By Proposition~\ref{prop:extmapfinite}, the morphism $\tilde{e}_{\beta,\gamma}$ is dominant and generically finite.

    For each pair of objects $(\cI_\ell,\cF_{\ell'})\in \ms(\beta)\times \ms(\gamma)$, the extended objects depend on the position of $\ell$ and $\ell'$.

    When $\ell\neq  \ell'$, as $\hom(\cO_{\ell'}(-H)[-1],\cI_\ell)=0$, we have the following commutative diagrams of distinguished triangles for the unique non-trivial extension $e(\cI_\ell,\cF_{\ell'})$:
      \begin{center}
	\begin{tikzcd}
   \cO_{\ell'}(-H)[-2] \arrow{r}\arrow{d}{\ev}& 0 \arrow{r}\arrow{d}  &  \cO_{\ell'}(-H)[-1] \arrow[r,equal]\arrow{d}{\ev}
		&   \cO_{\ell'}(-H)[-1] \arrow{d}{\ev}\\
	\cO(-H) \arrow{r}{g}\arrow{d}{\ev}	& \cI_\ell \arrow{r} \arrow[d,equal]& E_g
		\arrow{d} \arrow{r} & \cO(-H)[1] \arrow[d] \\
	\cF_{\ell'}[-1] \arrow{r}{\ev} \arrow{d}	&	\cI_\ell \arrow{r}\arrow {d}  & e(\cI_\ell,\cF_{\ell'}) \arrow{r}{}\arrow{d}{} & \cF_{\ell'} \arrow{d}\\
\cO_{\ell'}(-H)[-1] \arrow{r}& 0\arrow{r} &\cO_{\ell'}(-H) \arrow[r,equal] & \cO_{\ell'}(-H).
	\end{tikzcd}
\end{center}
In this diagram, we label $\ev$ on an arrow suggesting that the homomorphism group between the two objects is one-dimensional. Denote by $\mathbf P_{\ell,\ell'}$ the projective subspace spanned by $\ell$ and $\ell'$ in $\mathbf P^4$ when $\ell\cap\ell'=\emptyset$, or the tangent space of $Y_3$ at $\ell\cap\ell'$ when $\ell\cap\ell'$ is a point. In particular, the space $\mathbf P_{\ell,\ell'}\cong\mathbf P^3$.  Denote by $S_{\ell,\ell'}:=\mathbf P_{\ell,\ell'}\cap Y_3$ and $\iota\colon S_{\ell,\ell'}\to Y_3$ the embedding.  Then the map $g$ is determined up to a scalar by the property that $E_g$ is supported on $S_{\ell,\ell'}$. When $\ell\cap \ell'=\emptyset$, the object $E_g\cong\iota_*\cO_{S_{\ell,\ell'}}(-\ell)$ and $e(\cI_\ell,\cF_{\ell'})\cong \iota_*\cO_{S_{\ell,\ell'}}(\ell'-\ell)$.

In other words, a general object in $\ms(\beta+\gamma)$ is of the form $\iota_*\cO_{S_{\ell,\ell'}}(\ell'-\ell)$. This induces a rational map from $\ms(\beta+\gamma)$ to $(\mathbf P^4)^*$ of degree $72$. In the general case that $S_{\ell,\ell'}$ is a smooth cubic surface, there are exactly six ordered pairs of lines $(\ell_i,\ell_i')$ on $S_{\ell,\ell'}$ such that $[\ell_i-\ell'_i]=[\ell-\ell']$. So the degree of the morphism $\tilde{e}_{\beta,\gamma}$ is $6$.\\

When $\ell'=\ell$, we have \begin{align}\label{eq4213}
    \Hom(\cF_\ell[-1],\cI_\ell)\cong\Hom(\cO_{\ell}(-H)[-1],\cI_\ell)\cong\Hom(\cO_\ell(-H),\cO_\ell)\cong\C^2.
\end{align}
For every $0\neq f\in\Hom(\cF_\ell[-1],\cI_\ell)$, we have the commutative diagram of distinguished triangles for the extension  $\Cone(f)$:
   \begin{center}
	\begin{tikzcd}
	\cO(-H) \arrow{r}{g}\arrow{d}{\ev}	& 0 \arrow{r} \arrow[d]& \cO(-H)[1]
		\arrow{d} \arrow[r,equal] & \cO(-H)[1] \arrow[d] \\
	\cF_{\ell}[-1] \arrow{r}{f} \arrow{d}{\ev}	&	\cI_\ell \arrow{r}\arrow[d,equal]  & \Cone(f) \arrow{r}{}\arrow{d}{} & \cF_{\ell} \arrow{d}\\
\cO_{\ell}(-H)[-1] \arrow{r}{f_p}& \cI_\ell\arrow{r} &\cI_p \arrow[r] & \cO_{\ell}(-H).
	\end{tikzcd}
\end{center}
Via \eqref{eq4213}, the morphisms $f$ and $f_p$ in the diagram correspond to those in $\Hom(\cO_\ell(-H),\cO_\ell)$ with cokernel $\cO_p$ for some point $p\in \ell$. The object $\Cone(f)$ is isomorphic to \[\Cone(\cI_p[-1]\xrightarrow{\ev}\cO(-H)[1])=\R_{\cO(-H)}(\cI_p).\] 

In summary, we have the following commutative diagram of morphisms for $\tilde e_{\beta,\gamma}$:
     \begin{center}
	\begin{tikzcd}
    & \mathrm{Bl}_1^{-1}(\Delta_{F(Y_3)}) \arrow[r]{}{6:1}\arrow[d,hook]{d}\arrow[dl]
		&  Y_3 \arrow[d,hook]{}{\mathsf{pr}}\arrow[dr]\\
		\Delta_{F(Y_3)} \arrow[d,hook]{}& \mathrm{Bl}_\Delta(F(Y_3)\times F(Y_3)) \arrow{r} {\tilde{e}_{\beta,\gamma}}\arrow{dl}{\mathrm{Bl}_1}& M^s_\sigma(\beta+\gamma)
		\arrow{d} \arrow{dr}{\mathrm{Bl}_2} & \{\text{pt}\}\arrow[d,hook]\\
		F(Y_3)\times F(Y_3) \arrow{rr}{\AJ} & & J(Y_3)\arrow[r, phantom, sloped, "\supset"] \arrow[u,phantom]{}{} & \Theta_{J(Y_3)}.
	\end{tikzcd}
\end{center}
    \end{Ex}

\subsection{Bounds on dimensions of extensions}
The following computational lemmas will be useful in the proof of several propositions in the next section.
\begin{Lem}\label{lem:vwchicomputing}
    Assume that $v,w\in \Kn(\Ku(Y_3))$ satisfy $\phi_\sigma(w)-\phi_\sigma(v)\in(-1,1)$. Then $\chi(v,w)<0$ (resp. $=0$) if and only if $\phi_\sigma(w)-\phi_\sigma(v)\in(-\tfrac{2}{3},\tfrac{1}{3})$ (resp. $=-\tfrac{2}{3}$ or $\tfrac{1}{3}$).
\end{Lem}
The proof is quite elementary and left to the reader.
   
\begin{Lem}\label{lem:extspaceofsmallgapvw}
    Assume that $v,w\in \Kn(\Ku(Y_3))$ satisfy $\phi_\sigma(w)-\phi_\sigma(v)\in(0,\tfrac{1}{3})$. Then 
    \begin{enumerate}[(1)]
        \item $\dim \ET(v,w)\leq \dim \ms(v+w)+\chi(v,w)<\dim M^s_\sigma(v+w)$;
        \item $\overline{\BNP_{v+w}(\BN(v,v+w))}\neq M^s_\sigma(v+w)$. In other words, for a general $E_{v+w}\in M^s_\sigma(v+w)$, we have $\Hom(E_v,E_{v+w})=0$ for every $E_v\in M^s_\sigma(v)$.
    \end{enumerate}
\end{Lem}
\begin{proof}
    {\em (1)} For $E_w\in M^s_\sigma(w)$ and $E_v\in M^s_\sigma(v)$, as $\phi_\sigma(w)>\phi_\sigma(v)$, we have $\Hom(E_w,E_v[i])=0$ for all $i\leq 0$. By \eqref{eq:phaseofserre}, we have $\phi_\sigma(E_v[2])=\phi_\sigma(E_v)+2>\phi_\sigma(E_w)+\tfrac{5}{3}=\phi_\sigma(\mathsf {S}(E_w))$. So by Proposition \ref{prop:serreinv}, we have 
    \[\Hom(E_w,E_v[i])= \Hom(E_v[i],\mathsf {S}(E_w))^*=0,\]
    for every $i\geq 2$. So $\ms(v,w)^\dag=M^s_\sigma(v)\times M^s_\sigma(w)$. 
    
    It follows that $\ET(v,w)=\ET^0(v,w)$ which is the image of $e_{v,w}$. Hence \[\dim\ET(v,w)=\dim\ET^0(v,w)\leq \dim\Pe vw.\] Similar to the proof of Corollary \ref{cor:codimofETlocus}, we have $\dim\Pe vw=\dim M^s_\sigma(v+w)+\chi(v, w)$.
    As $\phi_\sigma(w)-\phi_\sigma(v)\in(0,\tfrac{1}{3})$, it follows by Lemma \ref{lem:vwchicomputing} that \[\dim M^s_\sigma(v+w)+\chi(v, w)<\dim M^s_\sigma(v+w).\] 
 The statement holds.\\
 
 \noindent  {\em (2)} Suppose that $\overline{\BNP_{v+w}(\BN(v,v+w))}= M^s_\sigma(v+w)$, then by Proposition \ref{prop:imofprimpliesext}, we have
\begin{align}\label{eq44}
M^s_\sigma(v+w)=\bigcup_{v'\in\triangle(v,v+w)\setminus\{\R (v+w)\}}\ET(v',v+w-v').
\end{align}
Note that for every $v'\in\triangle(v,v+w)\setminus\{\R (v+w)\}$, we have \[\phi_\sigma(v)\leq \phi_\sigma(v') <\phi_\sigma(v+w)<\phi_\sigma(v+w-v')\leq \phi_\sigma(w).\]
So $\phi_\sigma(v+w-v')-\phi_\sigma(v')\in(0,\tfrac{1}{3})$. 

It follows by {\em (1)} that $\dim(\ET(v',v+w-v'))<\dim M^s_\sigma(v+w)$, which contradicts \eqref{eq44}. Hence the statement holds.
\end{proof}

\begin{Lem}\label{lem:extspaceof13gap}
    Let $v,w\in \Kn(\Ku(Y_3))$ satisfying $\phi_\sigma(w)-\phi_\sigma(v)=\tfrac{1}{3}$. Assume that $\BN(v,w)$ does not contain any irreducible component of $\ms(v)\times \ms(w)$ and $e_{v,w}$ is well-defined on a non-empty locus. Then 
    $e_{v,w}$ is generically finite and dominant onto the irreducible component containing its image. 
\end{Lem}
\begin{proof}
    First note that by Proposition \ref{prop:extmapfinite}, the rational map $e_{v,w}$ is generically finite. Also as $\phi_\sigma(w)-\phi_\sigma(v)=\tfrac{1}{3}$, by Lemma \ref{lem:vwchicomputing}, we have $\chi(v,w)=0$. By Corollary \ref{cor:codimofETlocus}, we have $\dim\Pe vw=\dim M^s_\sigma(v+w)+\chi(v, w)=\dim M^s_\sigma(v+w)$, hence the domination statement follows.
\end{proof}

\begin{Lem}\label{lem:extbyonealpha}
    Let $w=n\alpha +m\beta\in \Kn(\Ku(Y_3))$ with $n\geq0$, $m\geq1$. Then for every general $E_w\in M^s_\sigma(w)$, every $E_\alpha\in M^s_\sigma(\alpha)$ and $0\neq f\in\Hom(E_w,E_\alpha[1])$, the object $E_f:=\Cone(E_w\xrightarrow[]{f}E_\alpha[1])[-1]$ is $\sigma$-stable.
\end{Lem}
\begin{proof}
    For every lattice point $u\in \triangle(w+\alpha,w)\setminus \{w+\alpha,w,0\}$, we have $u=n_1\alpha+m_1\beta$ for some $0\leq n_1\leq n$ and $0<m_1<m$. So $w-u=(n-n_1)\alpha+(m-m_1)\beta$ is in the $(\alpha,\beta)$-sextant as well. It follows that \[\phi_\sigma(w-u)-\phi_\sigma(u)\in(0,\tfrac{1}{3}).\]

    By Lemma \ref{lem:extspaceofsmallgapvw}, $\overline{\BNP_{w}(\BN(u,w))}\neq M^s_\sigma(w)$ for every such $u$. Let $E_w$ be an object in the non-empty open set
    \begin{align}\label{eq45}
    M^s_\sigma(w)\setminus \bigcup_{u\in \triangle(w+\alpha,w)\setminus\{w+\alpha,w,0\}}\overline{\BNP_{w}(\BN(u,w))}.
    \end{align}

    Now for every $E_\alpha$ and $f$, by Lemma \ref{lem:hnfactorextobj}, the character of $F=\HN^+_\sigma(E_f)$ is in $\triangle(w+\alpha,w)$. As $f\neq 0$, the character cannot be $w$. (If $\HN^+_\sigma(E_f)$ is strictly semistable, then we may choose one of its Jordan--H\"older factors $F$ such that $\Hom(F,E_f)\neq 0$.)

    Applying $\Hom(F,-)$ to the distinguished triangle $E_\alpha\to E_f\to E_w\xrightarrow[]{f}E_\alpha[1]$, we get $0\neq \Hom(F,E_f)\hookrightarrow\Hom(F,E_w)$. On the other hand, by the choice of $E_w$ in \eqref{eq45}, for every $\sigma$-stable object $F$ with character in $\triangle(w+\alpha,w)\setminus\{w+\alpha,w\}$, $\Hom(F,E_w)=0$. Hence the character of $F$ must be $w+\alpha$ which is the same as that of $E_f$. It follows that $E_f$ is $\sigma$-stable.
\end{proof}

\section{Stably birational equivalence between moduli spaces}\label{sec6}
In this section, we prove that rational maps from $\Pe vw$ to other $\Pe{v'}{w'}$ or $\ms(v+w)$ are well-defined, dominant, and even birational under certain assumptions on the characters. 

\subsection{Birational maps between $\Pe vw$'s}
We first estimate the codimension of the Brill--Noether locus involving the character $\alpha$. This is an essential step to show some of the rational maps $e_{\alpha,-}^*$ (see Definitions \ref{def:evw} \& \ref{def:evwR}) are well-defined.
\begin{Lem}\label{lem:homb}
    Let  $m,n$ be positive integers, $E_\alpha\in\ms(\alpha)$ and $E_{\beta}\in\ms(\beta)$. Then
   \begin{enumerate}[(1)]
       \item $\dim \BN(n\alpha+m\beta,E_{\beta}) \leq \dim\ms(n\alpha+m\beta)-m-1;$\\
        $\dim \BN(E_\alpha, n\alpha+m\beta) \leq \dim\ms(n\alpha+m\beta)-n-1;$
       \item $\dim\BN(\alpha,n\alpha+m\beta)\leq \dim\ms(n\alpha+m\beta)-n+1$;
       \item $\dim\BN^2(n\alpha+\beta,E_\beta)\leq \dim\ms(n\alpha+\beta)-3$.
   \end{enumerate}
        
\end{Lem}
\begin{proof}
\emph{(1)} By Proposition \ref{prop:imofprimpliesext}, the space $\BN(n\alpha+m\beta,E_\beta)$ is contained in
\begin{align}\label{eq420}
    \ET(n\alpha+(m-1)\beta,E_\beta)\cup(\bigcup_{
        v\in\triangle^*_{n,m}
    }\ET(n\alpha+m\beta-v,v)),
\end{align} 
    where $\triangle^*_{m,n}:=\triangle(n\alpha+m\beta,\beta)\setminus  \{\R(n\alpha+m\beta),\beta\}$. The set $\triangle^*_{m,n}$ is non-empty only when $n\geq 2$. In this case, for every $v\in\triangle^*_{m,n}$, both characters $v$ and $n\alpha+m\beta-v$ are strictly in the $(\alpha,\beta)$-sextant. In particular, we have $0<\phi_\sigma(v)-\phi_\sigma(n\alpha+m\beta-v)<\tfrac{1}{3}$. Denote by $v=n_1\alpha+m_1\beta$, then $1\leq n_1<n$ and $1\leq m_1<m$. By  Lemma \ref{lem:extspaceofsmallgapvw}, the codimension of $\ET(n\alpha+m\beta-v,v)$ is at least 
    \begin{align*}
        -\chi(n\alpha+m\beta-v,v)&=n_1(n-n_1)+(n_1+m_1)(m-m_1)\\&\geq 1+(1+m_1)(m-m_1)\geq 1+m. 
    \end{align*} 
    So we only need to bound the dimension of the first factor $\ET(n\alpha+(m-1)\beta,E_\beta)$ in \eqref{eq420}.
    
    When $(n,m)=(1,1)$, this statement follows from  Example \ref{eg:ab}. In all other cases, the Brill--Noether locus $\BN(\beta,(n\alpha+(m-1)\beta)[2])=\emptyset$. It follows that $\ET(n\alpha+(m-1)\beta,E_\beta)=\ET^0(n\alpha+(m-1)\beta,E_\beta)$. So its dimension is less than or equal to $\dim\Pe{n\alpha+(m-1)\beta}{E_\beta}$. 
    
    By a direct computation, we have $\dim \ms(n\alpha+m\beta)-\dim\Pe{n\alpha+(m-1)\beta}{E_\beta}=2-\chi(n\alpha+(m-1)\beta,\beta)=m+1$. The first inequality follows.

    The second inequality can be proved by a similar argument.\\

    \noindent\emph{(2)}  The statement follows from \emph{(1)} as $\dim\ms(\alpha)=2$.\\

    \noindent\emph{(3)} When $n=1$, the statement follows from  Example \ref{eg:ab}. We may assume $n\geq 2$.
    
    In \eqref{eq420}, every object $E$ in the first term $\ET(n\alpha,E_\beta)$ fits in a distinguished triangle $E_{n\alpha}\to E\to E_{\beta}\xrightarrow{+} $.  If $\hom(E,E_\beta)\geq 2$, then $E_{n\alpha}\in \BN(E_\beta,n\alpha)$. By Lemma \ref{lem:znamb}, the codimension of $\BN(E_\beta,n\alpha)$ is at least one. As $n\geq 2$, we have $\BN(\beta,n\alpha[2])=\emptyset$.  It follows that \begin{align*}
    & \dim(\ET(n\alpha,E_\beta)\cap \BN^2(n\alpha+\beta,E_\beta)) =  \dim(\ET^0(n\alpha,E_\beta)\cap \BN^2(n\alpha+\beta,E_\beta)) \\  \leq & \dim(\ms(n\alpha))-1-\chi(\beta,n\alpha)-1=n^2+n-1=\dim(\ms(n\alpha+\beta))-3.
    \end{align*}
    
    For the other terms in \eqref{eq420}, it is clear that a lattice point in $\triangle^*_{m,n}$ is of the form $s\alpha+\beta$ for some $1\leq s\leq  n-1$. By Lemma \ref{lem:extspaceofsmallgapvw}, the codimension of $\ET((n-s)\alpha,s\alpha+\beta)$ is $s(n-s)$. When $n\geq 4$, the inequality holds.
    
    The remaining cases are $n=2$ or $3$.

    When $n=2$, every element $E$ in $\ET(\alpha,\alpha+\beta)$ fits in a distinguished triangle $E_{\alpha}\to E\to E_{\alpha+\beta}\xrightarrow{+} $. As in Example \ref{eg:ab}, we have $\dim \BN(\alpha+\beta,E_\beta)=2$, $\dim \BN^2(\alpha+\beta,E_\beta)=1$. As computed in Example \ref{eg511}, we have $\dim\BN(\alpha,E_\beta)=1$, and $\dim\BN^2(\alpha,E_\beta)=0$. 

    Note that $\BN(\beta,(\alpha+\beta)[2])=\emptyset$. It follows that 
    \begin{align*}
        & \dim (\ET(\alpha,\alpha+\beta)\cap \BN^2(2\alpha+\beta,E_\beta))=\dim (\ET^0(\alpha,\alpha+\beta)\cap \BN^2(2\alpha+\beta,E_\beta))\\
        \leq &\dim(\ET^0(\alpha,\BN^2(\alpha+\beta,E_\beta))\cup \ET^0(\BN(\alpha,E_\beta),\BN(\alpha+\beta,E_\beta))\cup\ET^0(\BN^2(\alpha,E_\beta),\alpha+\beta))\\
        = &\ 4-\chi(\beta,\alpha+\beta)-1=5=\dim(\ms(2\alpha+\beta))-3.
    \end{align*}

    When $n=3$, every element $E$ in $\ET(\alpha,2\alpha+\beta)$ fits in a distinguished triangle $E_{\alpha}\to E\to E_{2\alpha+\beta}\xrightarrow{+} $. Note that $\codim \BN(2\alpha+\beta,E_\beta)\geq 1$ and $\codim\BN(\alpha,E_\beta)=1$.  By Lemma \ref{lem:extspaceofsmallgapvw}, $\codim \ET(\alpha,2\alpha+\beta)=2$. So the codimension of $\ET(\alpha,2\alpha+\beta)\cap \BN(3\alpha+\beta,E_\beta)$ is at least $3$. Similarly, the codimension of $\ET(2\alpha,\alpha+\beta)\cap \BN(3\alpha+\beta,E_\beta)$ is at least $3$.
\end{proof}

Now we have the first birational equivalence statement.

\begin{Prop}\label{prop:stabbir1}
    Let $Y_3$ be a smooth cubic threefold and $\Ku(Y_3)$ be its Kuznetsov component. Then for every $n\alpha+m\beta\in\Kn(\Ku(Y_3))$ with $n\geq 1$ and $m\geq2$ and every $E_\gamma\in\ms(\gamma)$, we have a birational equivalence \begin{align}
e^R_{E_\gamma[-1],n\alpha+m\beta}\colon\Pe{E_\gamma[-1]}{n\alpha+m\beta}\dashrightarrow \Pe{(n+1)\alpha+(m-1)\beta}{E_\gamma}.
    \end{align}
    More specifically, the map $e^R_{E_v,w}$ gives an isomorphism between dense open subsets at the well-defined locus.
\end{Prop}

We briefly explain the strategy of the proof. Recall from Remark \ref{rem:projectivization} the three conditions to verify, which correspond to the three steps of the proof. Here Condition \emph{(a)} is straightforward to check. 
For Condition \emph{(b)}, if the extension object is unstable, by considering the destabilizing object we produce a nontrivial morphism for the Brill--Noether locus. Then Proposition \ref{prop:imofprimpliesext} will relate the Brill--Noether locus to extension loci, and Lemma \ref{lem:extspaceofsmallgapvw} gives us bounds on the dimension of extension loci. This shows unstable extensions do not occur generically. 
Condition \emph{(c)} is the most difficult, but its verification follows exactly the same idea: use Proposition \ref{prop:imofprimpliesext} to relate the Brill--Noether locus to extension loci and now use Lemma \ref{lem:homb} to control dimensions of extension loci.

\begin{proof}
\textbf{Step 1.} We first verify Assumption \ref{asspdag} for $(\gamma[-1],n\alpha+m\beta)$ and $((n+1)\alpha+(m-1)\beta,\gamma)$, which is necessary to define the source and target of $e^R_{E_\gamma[-1],n\alpha+m\beta}$.

It is clear that $\phi_\sigma(n\alpha+m\beta)-\phi_\sigma(\gamma)\in(0,1)$, $\chi(n\alpha+m\beta,\gamma)=n>0$, and $\chi(\gamma,(n+1)\alpha+(m-1)\beta)=-m+1<0$.

By Serre duality, the jumping locus $\BN(n\alpha+m\beta,E_\gamma[1])\cong\BN(\mathsf S^{-1}(E_\gamma)[1],n\alpha+m\beta)$. By Proposition \ref{prop:serreinv}, the object $\mathsf S^{-1}(E_\gamma)[1]\in\ms(\alpha)$. By Lemma \ref{lem:homb}.(1), the Brill--Noether locus $\BN(n\alpha+m\beta,E_\gamma[1])$ does not contain any irreducible component of $\ms(n\alpha+m\beta)$.

Similarly, by Serre duality,
\[\BN(E_\gamma,((n+1)\alpha+(m-1)\beta)[2])\cong\BN(((n+1)\alpha+(m-1)\beta)[2],\mathsf S(E_\gamma)).\]
By Proposition \ref{prop:serreinv},  the object $\mathsf S(E_\gamma)[2]\in\ms(\beta)$. By Lemma \ref{lem:homb}.(1), the Brill--Noether locus $\BN(((n+1)\alpha+(m-1)\beta)[2],\mathsf S(E_\gamma))$ does not contain any irreducible component of $\ms(((n+1)\alpha+(m-1)\beta)[2])$. \\
\\
\textbf{Step 2.} We then show that $e_{E_\gamma[-1],n\alpha+m\beta}$ is well-defined on general points in  $\Pe{E_\gamma[-1]}{n\alpha+m\beta}$. More precisely, we will show that for every general $E_{n\alpha+m\beta}$ in $\ms({n\alpha+m\beta})$ and  every $0\neq f\in\Hom(E_{n\alpha+m\beta},E_\gamma)$, the object $E_f:=\Cone(f)[-1]$ is $\sigma$-stable.
    
    Note that every lattice point \[v\in \triangle_{n,m}:=\triangle((n+1)\alpha+(m-1)\beta,n\alpha+m\beta)\setminus\{(n+1)\alpha+(m-1)\beta,n\alpha+m\beta\},\]
    can be written in the form $a\alpha+b\beta$ for some $a\leq n$ and $b<m$. So \[\phi_{\sigma}(n\alpha+m\beta-v)-\phi_\sigma(v)\in(0,\tfrac{1}{3}).\]
    By Proposition \ref{prop:imofprimpliesext} and Lemma \ref{lem:extspaceofsmallgapvw}, 
    \[\dim\bigcup_{v\in\triangle_{n,m}}\BNP_{n\alpha+m\beta}(\BN(v,n\alpha+m\beta))=\dim\bigcup_{v\in\triangle_{n,m}}\ET(v,n\alpha+m\beta)<\dim \ms(n\alpha+m\beta).\]
    So for a general point $E_{n\alpha+m\beta}\in\ms(n\alpha+m\beta)$, we have $\Hom(E_v,E_{n\alpha+m\beta})=0$ for every $E_v\in\ms(v)$ with $v\in \triangle_{n,m}$.

    For every $0\neq f\in\Hom(E_{n\alpha+m\beta},E_\gamma)$,  suppose that the object $E_f$ is not $\sigma$-stable, then by Lemma \ref{lem:hnfactorextobj}, there exists $E_v\in \ms(v)$ with $v\in \triangle_{n,m}$ such that $\Hom(E_v,E_f)\neq 0$. One may apply $\Hom(E_v,-)$ to the distinguished triangle 
    \begin{align}\label{eq417}
        E_\gamma[-1]\to E_f \to E_{n\alpha+m\beta}\xrightarrow{+}
    \end{align} to get a contradiction. \\
\\
\textbf{Step 3.} Next, we show that for general $f\in\Hom(E_{n\alpha+m\beta},E_\gamma)$, the object $E_f$ is not in the Brill--Noether locus $\BN(E_\gamma[-2],((n+1)\alpha+(m-1)\beta))$.  For this purpose, we further assume that $E_{n\alpha+m\beta}\not\in \BN(n\alpha+m\beta,\mathsf{S}(E_\gamma)[-2])$, which has strictly smaller dimension than that of $\ms(n\alpha+m\beta)$ by Lemma \ref{lem:homb}.\emph{(1)}. 

    Applying $\Hom(-,\mathsf{S}(E_\gamma))$ to the distinguished triangle \eqref{eq417}, we get a long exact sequence:
    \begin{align}\label{eq419}
         0\to \Hom(E_f,\mathsf{S}(E_\gamma)[-2])\to \Hom(E_\gamma,\mathsf{S}(E_\gamma)[-1]) \to\Hom(E_{n\alpha+m\beta},\mathsf{S}(E_\gamma)[-1])\to...
    \end{align}
Here the first term $\Hom(E_{n\alpha+m\beta},\mathsf S(E_\gamma)[-2])=0$ as we assumed that $E_{n\alpha+m\beta}\not\in \BN(n\alpha+m\beta,\mathsf{S}(E_\gamma)[-2])$. It follows that \[\hom(E_\gamma[-2],E_f)=\hom(E_f,\mathsf S(E_\gamma)[-2])\leq  \hom(E_\gamma,\mathsf{S}(E_\gamma)[-1])=2.\]
So $E_f\not\in \BN^3(E_\gamma[-2],(n+1)\alpha+(m-1)\beta)$.

Note that the map 
\begin{align*}
\tilde e^{R}_{E_\gamma[-1],n\alpha+m\beta}\colon\Pe{E_\gamma[-1]}{n\alpha+m\beta} & \dashrightarrow \cup_{i\geq 0}\Pe{\BN^{=i}(E_\gamma[-2],(n+1)\alpha+(m-1)\beta)}{E_\gamma}\\
(E_{n\alpha+m\beta},f) & \mapsto (E_f,f^R)
\end{align*} 
is one-to-one whenever well-defined. As we have shown that the image is not contained in the locus $\BN^3(E_\gamma[-2],(n+1)\alpha+(m-1)\beta)$, to prove that the general image $E_f\notin \BN(E_\gamma[-2],(n+1)\alpha+(m-1)\beta)$, we only need to show that the dimension of $\Pe{\BN^{=1,2}(E_\gamma[-2],(n+1)\alpha+(m-1)\beta)}{E_\gamma}$ is strictly less than the dimension of the source $\Pe{E_\gamma[-1]}{n\alpha+m\beta}$, which is $n^2+mn+m^2+n$.

When $m\geq 3$, by Lemma \ref{lem:homb}.(1),
\begin{align*}
    & \dim\Pe{\BN^{=1,2}(E_\gamma[-2],(n+1)\alpha+(m-1)\beta)}{E_\gamma} \\ \leq & \dim(\BN(E_\gamma[-2],(n+1)\alpha+(m-1)\beta))-\chi(\gamma,(n+1)\alpha+(m-1)\beta)+2-1 \\
    = &\dim (\BN((n+1)\alpha+(m-1)\beta,\mathsf S E_\gamma[-2]))+m\\
    \leq &\ (n+1)^2+(n+1)(m-1)+(m-1)^2+1-(m-1)-1+m\\
    =&\ n^2+n+mn+m^2-m+2.
\end{align*}
When $m=2$, we may apply Lemma \ref{lem:homb}.(1) to compute the $\BN^{=1}$ part and Lemma \ref{lem:homb}.(3) to compute the $\BN^{=2}$ part. 

In each case, the dimension of $\Pe{\BN^{=1,2}(E_\gamma[-2],(n+1)\alpha+(m-1)\beta)}{E_\gamma}$ is strictly less than that of $\Pe{E_\gamma[-1]}{n\alpha+m\beta}$. Therefore, the image of a general point in every irreducible component of $\Pe{E_\gamma[-1]}{n\alpha+m\beta}$ is not in $\BN(E_\gamma[-2],((n+1)\alpha+(m-1)\beta))$. Hence the map $e^R_{E_\gamma[-1],n\alpha+m\beta}$ is generically well-defined. By Proposition \ref{prop:extbirpair}, the map $e^R_{E_\gamma[-1],n\alpha+m\beta}$ induces birational equivalences from every irreducible component in $\Pe{E_\gamma[-1]}{n\alpha+m\beta}$ to its target component.\\
\\
\textbf{Step 4.} Finally, by the same argument as above, when $n\geq2$ and $m\geq1$, the map $e^L_{n\alpha+m\beta,E_\gamma}$ is well-defined on general points on every irreducible component in $\Pe{n\alpha+m\beta}{E_\gamma}$. By Proposition \ref{prop:extbirpair}, the map $e^R_{E_\gamma[-1],n\alpha+m\beta}$ is a birational equivalence. 
\end{proof}

One corollary of Proposition \ref{prop:stabbir1} is the irreducibility of moduli spaces with non-primitive characters appearing in the chain. Together with the following result by Feyzbakhsh and Pertusi on the irreducibility of $\ms(m\beta)$, we may conclude the irreducibility of $\ms(v)$ for all non-zero $v\in\Kn(\Ku(Y_3))$.

\begin{Thm}[{\cite[Corollary 6.5]{FP}}]\label{thm:irred_ulrich}
    The moduli space $M_\sigma(m\beta)$ is irreducible for any $m>0$.
\end{Thm}

\begin{Cor}\label{cor:msvirreducible}
    Let $Y_3$ be a smooth cubic threefold and $\Ku(Y_3)$ be its Kuznetsov component. Then for every $0\neq v\in\Kn(\Ku(Y_3))$, the moduli space $M_\sigma(v)$ is irreducible.
\end{Cor}
\begin{proof}
By Remark \ref{rem:Sswapabc} and Theorem \ref{thm:irred_ulrich}, we only need to prove this for all $v=n\alpha+m\beta$ with $m,n>0$. We will first show the irreducibility of the stable locus. By Lemma \ref{lem:homb}.(1), the dimension of the jumping locus $\BN(n\alpha+m\beta,E_\gamma[1])$ and $\BN(E_\gamma[-2],n\alpha+m\beta)$ are both less than $\dim\ms(n\alpha+m\beta)$. So the number of irreducible components of $\Pe{E_\gamma[-1]}{n\alpha+m\beta}$ or $\Pe{n\alpha+m\beta}{E_\gamma}$ is the same as that of $\ms(n\alpha+m\beta)$. As $n\alpha+\beta$ is primitive, the moduli space $\ms(n\alpha+\beta)$ is irreducible. By Proposition \ref{prop:stabbir1},  both $\Pe{E_\gamma[-1]}{n\alpha+m\beta}$ and $\Pe{n\alpha+m\beta}{E_\gamma}$ are irreducible. Hence $\ms(v)$ is irreducible.

Now it remains to show that there is no irreducible component of $M_\sigma(v)$ consisting of only strictly semistable objects. For any $E_v\in M_\sigma(v)$, we always have $\Ext^2(E_v,E_v)=0$. Hence by deformation theory, each irreducible component of $M_\sigma(v)$ is of dimension at least $1-\chi(v,v)$. On the other hand, assume that $v=mv_0$ where $v_0$ is a primitive character, the locus of strictly semistable objects has dimension less than or equal to
\[\max \Big\{ \sum_{i=1,...,r}(1-\chi(m_iv_0,m_iv_0))\ \Big|\ (m_1,...,m_r)\text{ is a partition of }m,\ r>1 \Big\}<1-\chi(v,v).\] So the statement follows.
\end{proof}

\subsection{Stratification of moduli space in the $\Ku(Y_3)$ case} 
Given the irreducibility of all moduli spaces, we may prove the following two stratification statements Theorem \ref{thm:stratcubic3} and Proposition \ref{prop:stratofmbeta} for moduli spaces of semistable objects in $\Ku(Y_3)$. These results will provide an inductive approach to understanding the moduli spaces. In particular, in the next section, we will use them to prove the connectivity of the fibers of the \AJ map.

\begin{Lem}\label{lem:znamb}
    For every $m,n>0$, the space $\BN(n\alpha,m\beta)\neq \ms(n\alpha)\times M^s_\sigma(m\beta)$ and the space $\BN(E_\alpha,{m\beta})\neq \ms(m\beta)$.
\end{Lem} 
\begin{proof}
    By \eqref{eq4111}, there exists a semistable object $E\in M_\sigma(m\beta)$ such that $\Hom(E_\alpha,E)=0$. By the upper semi-continuity of $\Hom(E_\alpha,-)$ on $M_\sigma(m\beta)$ and Theorem \ref{thm:irred_ulrich}, the Brill--Noether locus $\BN(E_\alpha,m\beta)$ is not the whole stable locus.

    It follows that there exists $E_{m\beta}\in \ms(m\beta)$ such that $\Hom(E_\alpha^{\oplus n},E_{m\beta})=0$. Again by the upper semi-continuity of $\Hom(-,E_{m\beta})$ on $M_\sigma(n\alpha)$, the Brill--Noether locus $\BN(n\alpha,E_{m\beta})$ is not the whole stable locus. 
\end{proof}

\begin{Lem}\label{lem:openextobjKuY}
Let $w\in \Kn(\Ku(Y_3))$ satisfying $\phi_\sigma(w)-\phi_\sigma(\alpha)>0$.   Let $E_w\in\ms(w)$, $E\in M_\sigma(m\alpha)$ for some $m\geq 1$ such that $\hom(E_w,E[1])=-\chi(w,m\alpha)$. Assume that there exists $f\in \Hom(E_w,E[1])$ such that $E_f$ is $\sigma$-stable. Then there exists $E'\in\ms(m\alpha)$ and $f'\in \Hom(E_w,E'[1])$ such that $E_{f'}=\Cone(f')[-1]$ is $\sigma$-stable.
\end{Lem}
\begin{proof}
    As $M_\sigma(m\alpha)$ is irreducible by \cite[Corollary 6.5]{FP}, there exists an irreducible curve $C\to M_\sigma(m\alpha)$ such that at $p_0\in C$, the object is $E$; for a general point $p$ of $C$, the object $F_p\in \Ku(Y_3\times \{p\})$ is $\sigma$-stable. By removing the jumping locus, we may assume that   $F_p\not\in\BN(E_w,m\alpha[2])$ for all points on $C$. The relative $\Hom(E_w,m\alpha[1])$ can be defined over $C$ and is irreducible. For all $f'$ in this relative $\Ext^1$ space, by Lemma \ref{lem:hnfactorextobj}, the object $E_{f'}$ being $\sigma$-stable is an open condition. Therefore, for general $\sigma$-stable object $F_p$ and morphism $f'\in \Hom(E_w,F_p[1])$, the object $E_{f'}$ is $\sigma$-stable.
\end{proof}

\begin{Thm}\label{thm:stratcubic3}
    Let $Y_3$ be a smooth cubic threefold and $\Ku(Y_3)$ be its Kuznetsov component. Then for every $n\alpha+m\beta\in\Kn(\Ku(Y_3))$ with positive $m,n$, we have
    \begin{align}
        M^{s}_\sigma(n\alpha+m\beta)=\bigcup_{0\leq i\leq n,0\leq j\leq m,\tfrac{j}{i}<\tfrac{m}n}\ET(i\alpha+j\beta,(n-i)\alpha+(m-j)\beta). \label{eq:msnamb}
    \end{align}
    The map $e_{n\alpha,m\beta}$ is dominant and generically finite. The space $\overline{\ET(n\alpha,m\beta)}=M^s_\sigma(n\alpha+m\beta)$, and the other $\ET(v,w)$'s in \eqref{eq:msnamb} are of strictly smaller dimensions.
\end{Thm}

\begin{proof}
  We prove the theorem by induction on $n$. 

When $n=1$, by Lemma \ref{lem:extbyonealpha},  the rational map $e_{\alpha,m\beta}$ is well-defined on a non-empty subset of $\Pe{\alpha}{m\beta}$. Note that $\phi_\sigma(m\beta)-\phi_\sigma(\alpha)=\tfrac{1}{3}$. By Lemma \ref{lem:znamb}, the Brill--Noether locus $\BN(\alpha,m\beta)\neq \ms(\alpha)\times \ms(m\beta)$. By Corollary \ref{cor:msvirreducible} and  Lemma \ref{lem:extspaceof13gap}, the map $e_{\alpha,m\beta}$ is generically finite and dominant.

It follows that every general object $E_{\alpha+m\beta}\in M^s_\sigma(\alpha+m\beta)$ is in $\ET(\alpha,m\beta)$. Hence every general object $E_{\alpha+m\beta}\in\BNP_{\alpha+m\beta}\BN(\alpha,\alpha+m\beta)$. It follows that $\overline{\BNP_{\alpha+m\beta}\BN(\alpha,\alpha+m\beta)}=M^s_\sigma(\alpha+m\beta)$. By Proposition \ref{prop:imofprimpliesext}, the formula \eqref{eq:msnamb} holds for $\alpha+m\beta$.\\

Assuming that the statement holds for $n\alpha+m\beta$, we will first show that the locus $\ET((n+1)\alpha,m\beta)\neq \emptyset$. To see this, by induction, we may choose a general object $E_{n\alpha+m\beta}\in \ET(n\alpha,m\beta)\subset M^s_\sigma(n\alpha+m\beta)$ satisfying the assumption as that in Lemma \ref{lem:extbyonealpha}. In particular, the object $E_{n\alpha+m\beta}$ fits in a distinguished triangle
\begin{align}\label{eq47}
E_{n\alpha}\xrightarrow{f_1}E_{n\alpha+m\beta}\xrightarrow{f_2}E_{m\beta}\xrightarrow{+}E_{n\alpha}[1]
\end{align}
for some $E_{n\alpha}\in M^s_\sigma(n\alpha)$ and $E_{m\beta}\in M^s_\sigma(m\beta)$.

As $E_{n\alpha+m\beta}$ is general, we may further assume that $\mathsf S(E_{m\beta})\not\cong E_{n\alpha}[2]$. In particular, we have \begin{align*}
    &\hom(E_{m\beta},E_{n\alpha}[2])=\hom(E_{n\alpha}[2],\mathsf S(E_{m\beta}))=0;\\  &\hom(E_{m\beta},E_{n\alpha}[1])=-\chi(m\beta,n\alpha)=mn.
\end{align*}

Let $E'_\alpha\in M^s_\sigma(\alpha)$ be an object satisfying 
\begin{align}
    \Hom(E_{n\alpha},E'_\alpha)=0\text{ and }\hom(E_{m\beta},E'_{\alpha}[1])=-\chi(m\beta,\alpha)=m.
\end{align}
Apply $\Hom(-,E'_\alpha[1])$ to \eqref{eq47}, we get the long exact sequence
\begin{align*}
    0= & \Hom(E_{n\alpha}[1],E'_\alpha[1])\to \Hom(E_{m\beta},E'_\alpha[1])\xrightarrow{-\circ f_2}\Hom(E_{n\alpha+m\beta},E'_\alpha[1]) \\
    \xrightarrow{-\circ f_1} & \Hom(E_{n\alpha},E'_\alpha[1])\to \Hom(E_{m\beta}[-1],E'_\alpha[1])=0.
\end{align*}
As $\Hom(E_{m\beta},E'_\alpha[1])\neq 0$,   there exists $0\neq g\in \Hom(E_{n\alpha+m\beta},E'_\alpha[1])$ such that $g\circ f_1=0$.

Apply the octahedral axiom for $E_{n\alpha}\xrightarrow{f_1}E_{n\alpha+m\beta}\xrightarrow{g}E'_\alpha[1]$, we get the following commutative diagram of distinguished triangles:
    \begin{center}
	\begin{tikzcd}
     E'_\alpha \arrow[r,equal]\arrow{d}
		&   E'_\alpha\arrow{d}\\
		E_{n\alpha}\oplus E'_\alpha \arrow{r} \arrow{d}& E_g
		\arrow{d} \arrow{r}{\ev_2} & E_{m\beta} \arrow[d,equal] \arrow{r}{h} & (E_{n\alpha}\oplus E'_\alpha)[1] \\
		E_{n\alpha} \arrow{r}{f_1} \arrow{d}{0}  & E_{n\alpha+m\beta}\arrow{r}{f_2}\arrow{d}{g} & E_{m\beta} \arrow{r}{} & E_{n\alpha}[1]\\
  E'_\alpha[1] \arrow[r,equal]
		&   E'_\alpha[1].
	\end{tikzcd}
\end{center}
Applying Lemma \ref{lem:extbyonealpha} to the second vertical column, we see that $E_g$ is in $M^s_\sigma((n+1)\alpha+m\beta)$. By the choice of $E_{n\alpha}$ and $E'_\alpha$, we have $\hom(E_{m\beta},(E_{n\alpha}\oplus E'_\alpha)[1])=m(n+1)$. The pair of objects $(E_{m\beta},E_{n\alpha}\oplus E'_\alpha)$ is not on the $\Ext^1$-jumping locus of $M^s_\sigma(m\beta)\times M_\sigma((n+1)\alpha)$. By Lemma \ref{lem:openextobjKuY} and \ref{lem:openextobjstable}, for a general object $E_{(n+1)\alpha}\in M^s_\sigma((n+1)\alpha)$ and general $f\in \Hom(E_{m\beta},E_{(n+1)\alpha}[1])$, the object $\Cone(f)[-1]$ is $\sigma$-stable.

It follows that the map $e_{(n+1)\alpha,m\beta}$ is well-defined on a non-empty subset of $\Pe{(n+1)\alpha}{m\beta}$.  Note that $\phi_\sigma(m\beta)-\phi_\sigma(n\alpha)=\tfrac{1}{3}$. By Lemma \ref{lem:znamb}, the Brill--Noether locus $\BN(n\alpha,m\beta)\neq \ms(n\alpha)\times \ms(m\beta)$. By Corollary \ref{cor:msvirreducible} and  Lemma \ref{lem:extspaceof13gap}, the map $e_{n\alpha,m\beta}$ is generically finite and dominant.

It follows that every general object $E_{(n+1)\alpha+m\beta}\in M^s_\sigma((n+1)\alpha+m\beta)$ is in $\ET((n+1)\alpha,m\beta)$. Therefore, the space $\overline{\BNP_{(n+1)\alpha+m\beta}\ET((n+1)\alpha,(n+1)\alpha+m\beta)}=M^s_\sigma((n+1)\alpha+m\beta)$. By Proposition \ref{prop:imofprimpliesext},  formula \eqref{eq:msnamb} holds for $(n+1)\alpha+m\beta$.

For the last part of the statement, since for every $v\in \triangle((n+1)\alpha,(n+1)\alpha+m\beta)\setminus(\{(n+1)\alpha\}\cup \R((n+1)\alpha+m\beta))$, we have $\phi_\sigma((n+1)\alpha+m\beta-v)-\phi_\sigma(v)\in(0,\tfrac{1}{3})$. By Lemma \ref{lem:extspaceofsmallgapvw}, the space $\ET(v,(n+1)\alpha+m\beta-v)$ is with smaller dimension than that of $\ms((n+1)\alpha+m\beta)$. By formula \eqref{eq:msnamb}, we have $M^s_\sigma((n+1)\alpha+m\beta)=\overline{\ET((n+1)\alpha,m\beta)}$.
\end{proof}

By Remark \ref{rem:Sswapabc}, for every character $v\in \Kn(\Ku(Y_3))$ not proportional to $\alpha$, $\beta$ and $\gamma$, the moduli space $M^s_\sigma(v)$ admits a stratification by characters in its sextant. The following stratification theorem covers the only remaining case, and will be essential and enough for our purpose to prove the irreducibility of fibers of the \AJ map in the next section.

\begin{Prop}\label{prop:stratofmbeta}
For every $m\geq 2$, the moduli space
\begin{align}
M^s_\sigma(m\beta)=\ET(\alpha,(m-1)\beta+\gamma). \label{eq:msmb}
\end{align}
The map $e_{\alpha,(m-1)\beta+\gamma}$ is dominant.
\end{Prop}
\begin{proof}
  By Theorem \ref{thm:stratcubic3}, the rational map $e_{\alpha,m\beta}$ is dominant and generically finite.
  
  Note that $\ms(\alpha)\cong F(Y_3)$, so $b_1(\ms(\alpha))=10$. The \AJ map $\IJ_{m\beta}$ is not constant, so $b_1(M_\sigma(m\beta))>1$. By the first paragraph of the proof for Proposition \ref{prop:betti}, which does not rely on the statement in this section, the first Betti number $b_1(\ms(m\beta+\alpha))\leq 10$. It follows that $b_1(\ms(\alpha))+b_1(M_\sigma(m\beta))> 10 \geq b_1(\ms(m\beta+\alpha))$.  So the map $e_{\alpha,m\beta}$ cannot be birational.
  
  For a general object $(E_{\alpha},E_{m\beta},f)\in \Pe{\alpha}{m\beta}$ with $(E_\alpha,E_{m\beta})\not\in \BN(\alpha,m\beta)$, its image under $e_{\alpha,m\beta}$, which is by definition $E_f=\Cone(f)[-1]$, has at least another different preimage $(E_{\alpha}',E_{m\beta}',f')$ in $\Pe{\alpha}{m\beta}$. 

    Suppose $E_\alpha\cong E_\alpha'$, then we may apply $\Hom(-,E_{m\beta})$ to the distinguished triangle $E_\alpha'\to E_{f'}\to E_{m\beta}'\xrightarrow{+}$. As $\Hom(E_\alpha',E_{m\beta})=\Hom(E_\alpha,E_{m\beta})=0$, we have $\Hom(E_{m\beta}',E_{m\beta})\twoheadrightarrow \Hom(E_f',E_{m\beta})\neq0$. It follows that $E_{m\beta}\cong E_{m\beta}'$. Similar to the second part of the proof of Proposition \ref{prop:extmapfinite}, we can deduce that the morphisms $f= f'$ up to a scalar. So the two points $(E_{\alpha},E_{m\beta},f)$ and $(E_{\alpha}',E_{m\beta}',f')$ are the same, which is a contradiction.

    Hence $E_\alpha'\not\cong E_\alpha$. Apply $\Hom(E_\alpha',-)$ to the distinguished triangle $E_\alpha\to E_{f}\to E_{m\beta}\xrightarrow{+}$, we get $0\neq \Hom(E_\alpha',E_f)\hookrightarrow\Hom(E_\alpha',E_{m\beta})$. It follows that $(E'_\alpha,E_{m\beta})\in \BN(\alpha,m\beta)$. As this holds for all general $E_{m\beta}\in \ms(m\beta)$, we have $\overline{\BNP_{m\beta}(\BN(\alpha,m\beta))}=\ms(m\beta)$. By Proposition \ref{prop:imofprimpliesext}, formula \eqref{eq:msmb} holds.\\

    To show that $e_{\alpha,(m-1)\beta+\gamma}$ is dominant, we need to show that $\ET^{\geq 1}(\alpha,(m-1)\beta+\gamma)$ is of strictly smaller dimension. Continuing the above argument, by the octahedral axiom, we have the following commutative diagram of distinguished triangles:

    \begin{center}
	\begin{tikzcd}
     & E'_\alpha \arrow[r,equal]\arrow{d}
		&   E'_\alpha\arrow{d}\\
		E_\alpha \arrow{r} \arrow[d,equal]& E_f
		\arrow{d} \arrow{r} & E_{m\beta} \arrow[d] \arrow{r}{f} & E_\alpha[1]\arrow[d,equal] \\
		E_{\alpha} \arrow{r}  & E'_{m\beta}\arrow{r}{}\arrow{d}{g} & E_{(m-1)\beta+\gamma} \arrow{r}{} \arrow{d}{}& E_{\alpha}[1]\\
 & E'_\alpha[1] \arrow[r,equal]
		&   E'_\alpha[1].
	\end{tikzcd}
\end{center}
We claim that $E_{(m-1)\beta+\gamma}$ is $\sigma$-stable. Suppose that this is not the case, then by Lemma \ref{lem:hnfactorextobj}, the HN factor $\HN^-_\sigma(E_{(m-1)\beta+\alpha})$ is of character $m'\beta$ for some $m'<m$. This leads to a contradiction as $\Hom(E_{m\beta},\HN^-_\sigma(E_{(m-1)\beta+\alpha}))=0=\Hom(E'_\alpha[1],\HN^-_\sigma(E_{(m-1)\beta+\alpha}))$.

Applying $\Hom(-,E'_\alpha)$ to the distinguished triangle $E_\alpha\to E'_{m\beta}\to E_{(m-1)\beta+\gamma}\xrightarrow{+}$, we get $\hom(E_{(m-1)\beta+\gamma},E'_\alpha[2])=0$ or $1$; respectively, $\hom(E_{(m-1)\beta+\gamma},E'_\alpha[1])=m-1$ or $m$. So each general $E_{m\beta}$ is in $\ET^i(\alpha,(m-1)\beta+\gamma)$ for $i=0$ or $1$. We only need to show that $\ET^{= 1}(\alpha,(m-1)\beta+\gamma)$ is of strictly smaller dimension.

By Serre duality, $\hom(E_{(m-1)\beta+\gamma}, E'_\alpha[2])=\hom(\mathsf {S}^{-1} (E'_\alpha)[2],E_{(m-1)\beta+\gamma})$. So the Brill--Noether loci $\BN((m-1)\beta+\gamma,\alpha[2])\cong\BN(\beta,(m-1)\beta+\gamma)$. 
 
 When $m\geq 3$, by Lemma \ref{lem:homb}.(2),  the  locus $\BN(\beta,(m-1)\beta+\gamma)$ is of codimension at least $3$. So $\ET^1(\alpha,(m-1)\beta+\gamma)$ is of dimension at most
 \begin{align*}
   & \dim\BN(\beta,(m-1)\beta+\gamma) -\chi((m-1)\beta+\gamma,\alpha) -1+ 1 \\
   \leq & \dim \ms(\beta)+\dim\ms((m-1)\beta+\gamma)-3+m-1\\
   = &\ 2+(m^2-m+2)-3+m-1=m^2<\dim \ms(m\beta).
 \end{align*}
For the remaining case when $m=2$, we may consider the dimension of the \AJ map. Note that $\dim\BN(\beta,\beta+\gamma)=4$, so $\dim(\IJ_{2\beta}(\ET^1(\alpha,\beta+\gamma)))\leq 4<5=\dim(\IJ_{2\beta}(\ms(2\beta)))$. So $\ET^1(\alpha,\beta+\gamma)$ is not dense in $\ms(2\beta)$.

In all cases, we have $\ms(m\beta)=\overline{\ET^0(\alpha,(m-1)\beta+\gamma)}$. In other words, the map $e_{\alpha,(m-1)\beta+\gamma}$ is dominant.
\end{proof}
    
\subsection{Aside: derived categories of hearts and enhancements}
As an application of the results proved so far, we may show that the bounded derived category of the heart of $\sigma$ on $\Ku(Y_3)$ is equivalent to $\Ku(Y_3)$. This is one of the missing cases in our previous paper \cite{heart}.

\begin{Prop}\label{prop:dbofheart}
    Let $Y_3$ be a smooth cubic threefold, $\Ku(Y_3)$ be its Kuznetsov component, and $\sigma=(\cA,Z)$ be the stability condition constructed in \cite{BLMS:kuzcomponent}. Then $\Db(\cA)$ is equivalent to $\Ku(Y_3)$.
\end{Prop}

We first need the following lemma.
\begin{Lem}\label{lem:profbnN}
  Let $n,m$ be positive integers, and $v=n\alpha+m\beta$. Then for every $N\geq n+m-1$, we have $\pi_{v}(\BN(N\alpha+\beta,v))=\ms(v)$
\end{Lem}
\begin{proof}
    We make induction on $m$. When $m=1$, by Lemma \ref{lem:hnfactorextobj}, the locus $\ET(s\alpha,F)\neq \emptyset$ for every $s\geq 1$ and $F\in\ms(v)$. It follows that $\BN(v+s\alpha,F)\neq \emptyset$. The statement holds.

    Assume that the statement holds for all $v'=n\alpha+m'\beta$ with $m'<m$. As $\chi(v,\gamma)\geq 0$, we have $\BN(F,\gamma)\neq \emptyset$ for every $F\in\ms(v)$. By Proposition \ref{prop:imofprimpliesext}, there exists $v_0\in\triangle(v-\gamma,v)\setminus\{\R v\}$ such that we have $F\in\ET(v_0,v-v_0)$. If follows that there exists $F_{0}\in\ms(v_0)$ which is a subobject of $F$ in $\cA$.

   Note that $v_0=n_0\alpha+m_0\beta$ for some $m_0<m$ and $n_0+m_0\leq n+m$. By induction, for every $N\geq n_0+m_0-1$, there exists $F_N\in \ms(N\alpha+\beta)$ such that $\Hom(F_N,F_0)\neq 0$. As $F_0$ is a subobject of $F$ in $\cA$, we have  $\Hom(F_N,F)\neq 0$. 
   
   The statement holds for all $v$ by induction.
\end{proof}
\begin{proof}[Proof of Proposition \ref{prop:dbofheart}]
By \cite[Theorem 3.8]{heart}, we only need to check that the stability condition $\sigma$ on $\Ku(Y_3)$ satisfies \cite[Assumption 3.4]{heart}. It is clear that Assumption \emph{(a)} and \emph{(c)} hold. We only need to check Assumption \emph{(b)}. That is, 

\textbf{Claim:} For every non-zero object $E$ in $\mathcal A$ and every real number $s>0$, there exists a $\sigma$-stable object $F$ in $\mathcal A$ satisfying $\phi_\sigma(F)<s$ and  $\Hom(F,E)\neq 0$.\\

We may assume that the character $[E]$ is in the $(\alpha,\beta)$-sextant, since otherwise $\chi(N\alpha+\beta,E)>0$ when $N\gg 1$, the statement holds automatically. 

If $[E]=m\beta$, then by Theorem \ref{thm:stratcubic3}, for every $N>0$ with $\gcd(N,m)=1$, we have that $\BNP_{m\beta}(\BN(N\alpha+m\beta,m\beta))=\ms(m\beta)$. The claim follows. If $[E]=n\alpha+m\beta$ for some $n,m\geq 1$, the claim follows by Lemma \ref{lem:profbnN}. 
\end{proof}

As a direct consequence of Proposition \ref{prop:dbofheart} we deduce that $\Ku(Y_3)$ has a strongly unique dg enhancement (see \cite{CS:tour} for a survey on this topic). This implies that every equivalence $\Phi \colon \Ku(Y_3) \to \Ku(Y_3')$ among the Kuznetsov components of two cubic threefolds $Y_3$ and $Y_3'$ is of Fourier--Mukai type, in other words, the composition $i' \circ \Phi \circ i^*$ is a Fourier-Mukai functor, where $i^*$ denotes the left adjoint of the natural inclusion $\Ku(Y_3) \to \Db(Y_3)$ and $i' \colon \Ku(Y_3') \to \Db(Y_3')$ is the natural inclusion.

\begin{Cor}\label{cor:unique_dg}
The Kuznetsov component $\Ku(Y_3)$ of a smooth cubic threefold $Y_3$ has a strongly unique dg enhancement. In particular, every equivalence $\Ku(Y_3) \to \Ku(Y_3')$ for a smooth cubic threefold $Y_3'$  is of Fourier--Mukai type.   
\end{Cor}
\begin{proof}
The first statement follows from Proposition \ref{prop:dbofheart} and \cite[Theorem 3.12]{heart}. The second statement follows from the first part and \cite[Corollary 3.15]{heart}.    
\end{proof}

\subsection{Stable birationality between moduli spaces}
In this section, our main goal is to show that when the dimension of $\ms(v)$ with primitive $v$ is sufficiently large, these moduli spaces are all stably birational to each other. We achieve this by proving more birational equivalences between different $\Pe vw$'s. 

We first show that for primitive $v_i=n_i\alpha+m_i\beta$, $i=1,2$ with $n_1+m_1=n_2+m_2$, the moduli spaces $\ms(v_1)$ and $\ms(v_2)$ are stably birational to each other. Here by ``stably birational'', we mean that $\ms(v_i)\times \P^{N_i}$ are birational to each other for some $N_i\in \Z_{\geq 0}$. Note that Proposition \ref{prop:stabbir1} is not enough for this purpose, as there are non-primitive characters in the middle. For a non-primitive character $mv_0$, the spaces $\Pe{E_\gamma[1]}{mv_0}$ and $\Pe{mv_0}{E_\gamma}$ are Severi--Brauer varieties over $\ms(mv_0)$, and it is a well-known difficult question to see whether they are birational to projective bundles over the base. So we cannot use them to relate primitive characters in the simplest way. 

To deal with this issue, we prove the following general statement to relate primitive characters directly. 
\begin{Prop}\label{prop:stabbir2}
Let $Y_3$ be a smooth cubic threefold and $\Ku(Y_3)$ be its Kuznetsov component. Then for every $n>0$, $m>s>0$, and general $E_{s\gamma}\in\ms(s\gamma)$, we have a birational equivalence 
\begin{align}
e^R_{E_{s\gamma}[-1],n\alpha+m\beta}\colon\Pe{E_{s\gamma}[-1]}{n\alpha+m\beta}\dashrightarrow \Pe{(n+s)\alpha+(m-s)\beta}{E_{s\gamma}}.
    \end{align}
\end{Prop}
\begin{proof}
    We prove the statement by induction on $s$. When $s=1$, the statement follows from Proposition \ref{prop:stabbir1}. Assume that $s\geq 2$ and the statement holds for all smaller $s$.

   We first show that Assumption \ref{asspdag} holds for $(s\gamma[-1],n\alpha+m\beta)$ and $((n+1)\alpha+(m-1)\beta,s\gamma)$. It is clear that $\phi_\sigma(s\gamma)-\phi_\sigma(n\alpha+m\beta)\in(0,1)$, $\chi(n\alpha+m\beta,s\gamma)=sn>0$, and $\chi(s\gamma,(n+1)\alpha+(m-1)\beta)=s(1-m)<0$.

By Serre duality, we have that the jumping locus $\BN(n\alpha+m\beta,s\gamma[1])\cong\BN(s\alpha,n\alpha+m\beta)$. Note that $\mathsf S^{-1}(E_\gamma)[1]\in\ms(\alpha)$. By Theorem \ref{thm:stratcubic3}, a general object $F\in\ms(n\alpha+m\beta)$ is in $\ET(n\alpha,m\beta)$. Since $\BN(s\alpha,n\alpha)\neq \ms(s\alpha)\times\ms(n\alpha)$ and $\BN(s\alpha,m\beta)\neq \ms(s\alpha)\times\ms(m\beta)$ by Lemma \ref{lem:znamb}, we have $\BN(s\alpha,n\alpha+m\beta)\neq \ms(s\alpha)\times\ms(n\alpha+m\beta)$.  Similarly, by Serre duality, $\BN(s\gamma,((n+s)\alpha+(m-s)\beta)[2])\cong\BN(((n+s)\alpha+(m-s)\beta)[2],\beta[2])$. Again by Lemma \ref{lem:znamb} and Theorem \ref{thm:stratcubic3}, this is not the full space. \\

We then show that $e^R_{s\gamma[-1],n\alpha+m\beta}$ maps a general point $(E_{s\gamma}[-1],E_{n\alpha+m\beta},f)$ to $(E_f,E_{s\gamma})\in\ms((n+s)\alpha+(m-s)\beta,s\gamma)^\dag$.

By induction, we may fix a general $E_{(s-1)\gamma}\in\ms((s-1)\gamma)$ so that  $e^L_{{(n+s-1)\alpha+(m-s+1)\beta,E_{(s-1)\gamma}}}$ is a birational equivalence and in addition $\BN(E_{(s-1)\gamma}[-2],(n+s)\alpha+(n-s)\beta)\neq \ms((n+s)\alpha+(n-s)\beta)$. 

Let $E_\gamma\in\ms(\gamma)$, then by Proposition \ref{prop:stabbir1}, the space  $\BN(E_{\gamma}[-2],(n+s)\alpha+(n-s)\beta)\neq \ms((n+s)\alpha+(n-s)\beta)$, $\BN(n\alpha+m\beta,E_\gamma[1])\neq \ms(n\alpha+m\beta)$, and the map $e^R_{E_\gamma[-1],(n+s-1)\alpha+(m-s+1)\beta}$ is a birational equivalence.

So for every general $E\in\ms((n+s-1)\alpha+(n-s+1)\beta)$, $g\in\Hom(E_{(s-1)\gamma}[-1],E)$ and $h\in\Hom(E,E_\gamma)$, we have 
\begin{align}
 \label{eq:eg}   E_g:=\Cone(g) & \in\ms(n\alpha+m\beta)\setminus \BN(n\alpha+m\beta,E_{(s-1)\gamma}[1])\cup\BN(n\alpha+m\beta,E_{\gamma}[1])\\
 \label{eq:eh}   E_h:=\Cone(h)[-1] & \in \ms((n+s)\alpha+(n-s)\beta)\setminus \cup_{i=1,(s-1)}\BN(E_{i\gamma}[-2],(n+s)\alpha+(n-s)\beta).
\end{align}
   By the octahedral axiom, there is a commutative diagram of distinguished triangles:

    \begin{center}
	\begin{tikzcd}
     & E_h \arrow[r,equal]\arrow{d}
		&   E_h\arrow{d}\\
		E_{(s-1)\gamma}[-1] \arrow{r}{g} \arrow[d,equal]& E
		\arrow{d}{h} \arrow{r} & E_g \arrow[d] \arrow{r}{} & E_{(s-1)\gamma}\arrow[d,equal] \\
		E_{(s-1)\gamma}[-1] \arrow{r}  & E_\gamma\arrow{r}{}\arrow{d}{} & F_{s\gamma} \arrow{r}{} \arrow{d}{}& E_{(s-1)\gamma}\\
 & E_h[1] \arrow[r,equal]
		&   E_h[1].
	\end{tikzcd}
\end{center}
As $F_{s\gamma}$ is extended by $E_\gamma$ and $E_{(s-1)\gamma}$, it is $\sigma$-semistable.  In particular, we have $F_{s\gamma}\in M_\sigma(s\gamma)$. 

By \eqref{eq:eg}, we have  $\Hom(E_g,E_{(s-1)\gamma}[1])=\Hom(E_g,E_\gamma[1])=0$. Hence $\Hom(E_g,F_{s\gamma}[1])=0$. By the semicontinuity, 
 we have $\Hom(E_g,E_{s\gamma}[1])=0$ for general $E_{s\gamma}\in\ms(s\gamma)$. 

Note that $E_h\cong \Cone(E_g\to F_{s\gamma})$ is $\sigma$-stable. By Lemmas \ref{lem:openextobjKuY} and \ref{lem:openextobjstable}, for each general $f\in\Hom(E_g,E_{s\gamma})$, the object $\Cone(f)[-1]\in \ms((n+s)\alpha+(n-s)\beta)$. 

By \eqref{eq:eh}, we have $\Hom(E_\gamma,E_h[2])=\Hom(E_{(s-1)\gamma},E_h[2])=0$. Hence $\Hom(F_{s\gamma},E_h[2])=0$. By the upper semicontinuity of $\hom(-,E_h[2])$ on $M_\sigma(s\gamma)$ and the general assumption on $E_{s\gamma}$ and $f$, we have $\Hom(E_{s\gamma},E_h[2])=0$.

In summary, the map $e^R_{s\gamma[-1],n\alpha+m\beta}$ maps the above point $(E_{s\gamma}[-1],E_g,f)$ to a point in the locus $\ms((n+s)\alpha+(m-s)\beta,s\gamma)^\dag$. By Proposition \ref{prop:extbirpair} and Corollary \ref{cor:msvirreducible}, the map $e^R_{s\gamma[-1],n\alpha+m\beta}$ is a birational equivalence.
\end{proof}

Next, we will prove the stable birationality between characters $v=n\alpha+m\beta$ with different $n+m$. The strategy is to replace the character $\gamma$ in Proposition \ref{prop:stabbir1} by $\beta+\gamma$. This will allow one to relate the characters with $n+m$ differed by one. The strategy of the proof is the same as that for Proposition \ref{prop:stabbir1}, but requires a more delicate estimate on the dimension of several Brill--Noether loci. Similar to Lemma \ref{lem:homb}, we need the following lemma for preparation.

\begin{Lem}\label{lem:homb2}
    Let $m,n$ be positive integers. Let $E_{\alpha-\gamma}\in\ms(\alpha-\gamma)$ and $E_{\alpha+\beta}\in\ms(\alpha+\beta)$ be general elements. Then
   \begin{enumerate}[(1)]
       \item when $n>m$, we have $\dim \BN(n\alpha+m\beta,E_{\alpha+\beta}) \leq \dim\ms(n\alpha+m\beta)-4;$\\
         $\dim\BN^2(n\alpha+m\beta,E_{\alpha+\beta})\leq \dim\ms(n\alpha+m\beta)-5$.
         \item $\dim\BN^i(n\alpha+(n-s)\beta,E_{\alpha+\beta},)$ and $\dim\BN^i(E_{\alpha+\beta},(n-s)\alpha+n\beta)\leq \dim\ms((n-s)\alpha+n\beta)-i$ 
         for $s=1,2$, $n\geq s$, and $i\geq 1$.
        \item when $n>m$, we have $\dim \BN(E_{\alpha-\gamma}, n\alpha+m\beta) \leq \dim\ms(n\alpha+m\beta)-1.$
      \end{enumerate}
        
\end{Lem}
\begin{proof}
\emph{(1)} By Proposition \ref{prop:imofprimpliesext}, the space $\BN(n\alpha+m\beta,E_{\alpha+\beta})$ is contained in
\begin{align}\label{eq440}
    \ET((n-1)\alpha+(m-1)\beta,E_{\alpha+\beta})\cup(\bigcup_{
        v\in\triangle'_{n,m}
    }\ET(n\alpha+m\beta-v,v)),
\end{align}
    where $\triangle'_{m,n}:=\triangle(n\alpha+m\beta,\alpha+\beta)\setminus  \{\R(n\alpha+m\beta),\alpha+\beta\}$.  By Lemma \ref{lem:extspaceofsmallgapvw}, the codimension of the first term is at least $4-\chi(\alpha,\alpha+\beta)=5$.

    When $n-m=1$, the set $\triangle'_{m,n}$ is empty, so the statement holds for this case. We induct on $n-m$ and assume that the statement holds for all smaller $n-m$.

      For $v\in \triangle'_{m,n}$ in \eqref{eq440}, each object $E$ in $\ET(n\alpha+m\beta-v,v))$ fits in a distinguished triangle $E_{n\alpha+m\beta-v}\to E\to E_{v}\xrightarrow{+} $. It is clear that $v=a\alpha+b\beta$ for some $a-b<n-m$. By induction, $\BN(v,E_{\alpha+\beta})$ (resp. $\BN^2$) has codimension at least $4$ (resp. $5$). When $n\alpha+m\beta-v\not\in \Z\alpha$, by induction, $\BN(n\alpha+m\beta-v,E_{\alpha+\beta})$ has codimension at least $4$ as well. By Lemma \ref{lem:extspaceofsmallgapvw}, the space $\ET(n\alpha+m\beta-v,v)\cap\BN(n\alpha+m\beta,E_{\alpha+\beta})$ has codimension at least $5$ for all these terms.

      The remaining terms are those with $v$'s satisfying $n\alpha+m\beta-v=t\alpha$ for some $t\geq 1$. This can only happen when $m=1$ and $v=(n-t)\alpha+\beta$ with $n-t\geq 2$. 
      
      By Theorem \ref{thm:stratcubic3}, the object $E_{\alpha+\beta}$ always fits in the distinguished triangle $E_\alpha\to E_{\alpha+\beta}\to E_\beta\xrightarrow{+}$ for some $E_\alpha\in \ms(\alpha)$ and $E_\beta\in\ms(\beta)$. Together with Lemma \ref{lem:znamb}, it follows that the Brill--Noether locus $\BN(t\alpha,E_{\alpha+\beta})$ has codimension at least $1$. 
      
      When $t\geq 2$, we have $-\chi(v,n\alpha+\beta-v)=-\chi((n-t)\alpha+\beta,t\alpha)=t(n-t+1)\geq 4$. By Lemma \ref{lem:extspaceofsmallgapvw},
      \begin{align*}
           & \codim(\ET((n-t)\alpha+\beta,t\alpha)\cap\BN(n\alpha+\beta,E_{\alpha+\beta})) \\
           \geq & \min\{\codim(\BN((n-t)\alpha+\beta,E_{\alpha+\beta})),\codim(\BN(t\alpha,E_{\alpha+\beta})\}-\chi((n-t)\alpha+\beta,t\alpha)\geq 5.
      \end{align*}
      The statement holds in this case.
      
      When $t=1$, as $E_{\alpha+\beta}$ is general, we may assume that the preimage $e_{\alpha,\beta}^{-1}(E_{\alpha+\beta})$ is a finite set and not in $\BN(\alpha,\beta)$. It follows that $\BN(\alpha,E_{\alpha+\beta})$ (resp. $\BN^2$) has codimension $2$ (resp. empty). As $-\chi((n-1)\alpha+\beta,\alpha)\geq 2$, by Lemma \ref{lem:extspaceofsmallgapvw} and the same computation as above, the locus  $\ET((n-1)\alpha+\beta,\alpha)\cap\BN$ (resp. $\BN^2$) has codimension at least $4$ (resp. $5$).\\

      \emph{(2)} We only  prove the statement for $\BN^i(n\alpha+(n-s)\alpha,E_{\alpha+\beta})$, and the other part can be proved in exactly the same way. We prove the statement by induction on $n-s$. 

      When $n=s=1$,  as $E_{\alpha+\beta}$ is general, we may assume that the preimage $e_{\alpha,\beta}^{-1}(E_{\alpha+\beta})$ is a finite set and not in $\BN(\alpha,\beta)$. It follows that $\BN^2(\alpha,E_{\alpha+\beta})=\emptyset$.
      
      When $n=s=2$, by Proposition \ref{prop:imofprimpliesext}, the Brill--Noether locus $\BN(2\alpha,E_{\alpha+\beta})=\ET(\gamma[-1],E_{\alpha+\beta})$. For $i\geq 2$, note that $\chi(\alpha+\beta,\gamma)=1$, we have 
      \begin{align*}
          \BN^i(E_{\alpha+\beta},\gamma)=\BN^{i-1}(E_{\alpha+\beta}[-1],\gamma)=\BN^{i-1}(\gamma,\mathsf S(E_{\alpha+\beta})[-1])\cong\BN^{i-1}(\alpha,E_{\alpha+\beta}).
      \end{align*} By the result for $n=s=1$ case, the Brill--Noether locus $\BN^2(E_{\alpha+\beta},\gamma)$ is of dimension $0$ and $\BN^3(E_{\alpha+\beta},\gamma)=\emptyset$. It follows that
      \begin{align*}
          \dim \ET^t(\gamma[-1],E_{\alpha+\beta})\leq \dim\BN^{t+1}(E_{\alpha+\beta},\gamma)+t\leq 2
      \end{align*}
      for every $t\geq 0$. So $\dim\BN(2\alpha,E_{\alpha+\beta})=\dim\ET(\gamma[-1],E_{\alpha+\beta})\leq 2$. The statement holds for $i\leq 3$. By Proposition \ref{prop:stratofmbeta}, it is clear that $\BN^4(2\alpha,E_{\alpha+\beta})=\emptyset$. So the statement holds in all cases when $n-s=0$.\\

     Now assume that statement holds for smaller $n-s$. When $s=1$, the formula \eqref{eq440} only has the first term. As the Brill--Noether locus $\BN(E_{\alpha+\beta},((n-1)\alpha+(n-s-1)\beta)[2])=\emptyset$, by Lemma \ref{lem:extspaceofsmallgapvw} and induction, we have
      \begin{align*}
          & \codim \BN^i(n\alpha+(n-s)\beta,E_{\alpha+\beta})\\
          \geq & \codim \ET(\BN^{i-1}((n-1)\alpha+(n-s-1)\beta,E_{\alpha+\beta}),E_{\alpha+\beta}) \\
          \geq & \codim(\BN^{i-1}((n-1)\alpha+(n-s-1)\beta,E_{\alpha+\beta}))-\chi((n-1)\alpha+(n-s-1)\beta,\alpha+\beta)\\ \geq &\ i-1+(n-1)+2(n-s-1)\geq i,
      \end{align*}
for every $i\geq 1$.

When $s=2$, the formula \eqref{eq440} has at most two terms. The first term follows the same argument as that in the $s=1$ case. The second term is $\ET(l\alpha+(l-1)\beta,(l+1)\alpha+l\beta)$ for some $1\leq l<n-s$. By Lemma \ref{lem:extspaceofsmallgapvw} and induction, we have
\begin{align*}
    & \codim (\BN^i(n\alpha+(n-s)\beta,E_{\alpha+\beta})\cap \ET(l\alpha+(l-1)\beta,(l+1)\alpha+l\beta)) \\
    \geq &\ i-\chi(l\alpha+(l-1)\beta,(l+1)\alpha+l\beta)\geq i-1+3l^2\geq i+2
\end{align*}
for every $i\geq 1$.

By induction, the statement holds in all cases.\\

      \emph{(3)} By Proposition \ref{prop:imofprimpliesext}, the space $\BN(E_{\alpha-\gamma},n\alpha+m\beta)$ is contained in
\begin{align}\label{eq441}
    \ET(E_{\alpha-\gamma},(n-2)\alpha+(m+1)\beta)\cup(\bigcup_{
        v\in\triangle^*_{n,m}
    }\ET(v,n\alpha+m\beta-v)),
\end{align}
    where $\triangle^*_{m,n}:=\triangle(\alpha-\gamma,n\alpha+m\beta)\setminus  \{\R(n\alpha+m\beta),\alpha-\gamma\}$. Every $v\in\triangle^*_{m,n}$ is in the $(\alpha,\beta)$-sextant. So $\phi_\sigma(n\alpha+m\beta-v)-\phi_\sigma(v)<\tfrac{1}{3}.$  By Lemma \ref{lem:extspaceofsmallgapvw}, the codimension of the rest terms is at least one.
    
    For the first term, when $n-2>m+1$, $\phi_\sigma((n-2)\alpha+(m+1)\beta)-\phi_\sigma(\alpha-\gamma)<\tfrac{1}{3}.$ When $n-2=m+1$, the Brill--Noether locus $\BN((n-2)\alpha+(m+1)\beta,E_{\alpha-\gamma}[2])$ is either empty or only consists of the point $\mathsf S^{-1}(E_{\alpha-\gamma})[2]$. By Lemma \ref{lem:extspaceofsmallgapvw}, the codimension of the first term is at least $4$.

    The only remaining case is when $n-2<m+1$,  the Brill--Noether locus $\BN^i((n-2)\alpha+(m+1)\beta,E_{\alpha-\gamma}[2])=\BN^i(\mathsf S^{-1}(E_{\alpha-\gamma})[2],(n-2)\alpha+(m+1)\beta)$. By \emph{(2)},
    \begin{align*}
        &\dim\ET^i(E_{\alpha-\gamma},(n-2)\alpha+(m+1)\beta)\\
        \leq & \dim((n-2)\alpha+(m+1)\beta)-i-\chi((n-2)\alpha+(m+1)\beta,\alpha-\gamma)+i\\
        = &\ n^2+mn+m^2+1-n+m=\dim\ms(n\alpha+m\beta)-n+m.
    \end{align*}
  The statement follows by the assumption that $n>m$.
\end{proof}

\begin{Prop}\label{prop:stabbir3}
    Let $Y_3$ be a smooth cubic threefold and $\Ku(Y_3)$ be its Kuznetsov component. Then for every $n\alpha+m\beta\in\Kn(\Ku(Y_3))$ with $n>m\geq 3$ and  every general $E_{\beta+\gamma}\in\ms(\beta+\gamma)$, we have a birational equivalence \begin{align}
e^R_{E_{\beta+\gamma}[-1],n\alpha+m\beta}:\Pe{E_{\beta+\gamma}[-1]}{n\alpha+m\beta}\dashrightarrow \Pe{(n+1)\alpha+(m-2)\beta}{E_{\beta+\gamma}}.
    \end{align}
\end{Prop}
\begin{proof}
The proof follows the same strategy as that for Proposition \ref{prop:stabbir1}. 

First, we show that Assumption \ref{asspdag} holds for $((\beta+\gamma)[-1],n\alpha+m\beta)$ and $((n+1)\alpha+(m-2)\beta,\beta+\gamma)$. It is clear that $\phi_\sigma(n\alpha+m\beta)-\phi_\sigma(\beta+\gamma)\in(0,1)$, $\chi(n\alpha+m\beta,\beta+\gamma)=n-m>0$, and $\chi(\beta+\gamma,(n+1)\alpha+(m-2)\beta)=-n-2m+3<0$.

By Serre duality, the Brill--Noether locus $\BN(n\alpha+m\beta,E_{\beta+\gamma}[1])=\BN(\mathsf S^{-1}(E_{\beta+\gamma})[1],n\alpha+m\beta)$. Note that $\mathsf S^{-1}(E_{\beta+\gamma})[1]\in \ms(\alpha-\gamma)$, by Lemma \ref{lem:homb2}.(3), the space $\BN(n\alpha+m\beta,E_{\beta+\gamma}[1])\neq \ms(n\alpha+m\beta)$ for general $E_{\beta+\gamma}$.

We have $\BN(E_{\beta+\gamma},((n+1)\alpha+(m-2)\beta)[2])=\BN(((n+1)\alpha+(m-2)\beta)[2],\mathsf S(E_{\beta+\gamma}))$. Note that $\mathsf S(E_{\beta+\gamma})[-2]\in \ms(\alpha+\beta)$, by Lemma \ref{lem:homb2}.(1), the space $\BN(E_{\beta+\gamma},((n+1)\alpha+(m-2)\beta)[2])\neq \ms((n+1)\alpha+(m-2)\beta)[2]$.\\

    Secondly, we show that the map $e_{E_{\beta+\gamma}[-1],n\alpha+m\beta}$ is well-defined on any general points in $\Pe{E_{\beta+\gamma}[-1]}{n\alpha+m\beta}$. 
    
    Note that every lattice point \[v\in \triangle'_{n,m}:=\triangle((n+1)\alpha+(m-2)\beta,n\alpha+m\beta)\setminus\{(n+1)\alpha+(m-2)\beta,n\alpha+m\beta\},\]
    can be written of the form $a\alpha+b\beta$ for some $a\leq n$ and $b<m$. 
    
    By Proposition \ref{prop:imofprimpliesext} and Lemma \ref{lem:extspaceofsmallgapvw}, we have\
    \[\dim (\bigcup_{v\in\triangle'_{n,m}}\BNP_{n\alpha+m\beta}(v,n\alpha+m\beta))<\dim \ms(n\alpha+m\beta).\]
    So for every general point $E_{n\alpha+m\beta}\in\ms(n\alpha+m\beta)$, the vanishing $\Hom(E_v,E_{n\alpha+m\beta})=0$ holds for every $E_v\in\ms(v)$ with $v\in \triangle'_{n,m}$.

    For every $0\neq f\in\Hom(E_{n\alpha+m\beta},E_\gamma)$, by Lemma \ref{lem:hnfactorextobj}, the object $E_f:=\Cone(f)[-1]$ is $\sigma$-stable. \\
    
    Finally, we show that for every general $f$, the object $E_f$ is not in the Brill--Noether locus $\BN(E_{\beta+\gamma}[-2],(n+1)\alpha+(m-2)\beta)$.  For this purpose, we further assume that $E_{n\alpha+m\beta}\not\in \BN(n\alpha+m\beta,\mathsf{S}(E_{\beta+\gamma})[-2])$, which has a strictly smaller dimension than that of $\ms(n\alpha+m\beta)$. 

    Apply $\Hom(-,\mathsf{S}(E_{\beta+\gamma}))$ to the distinguished triangle $E_{\beta+\gamma}[-1]\to E_f\to E_{n\alpha+m\beta}\xrightarrow{+}$, we get the long exact sequence:
    \begin{align*}
         0&\to \Hom(E_f,\mathsf{S}(E_{\beta+\gamma})[-2])\to \Hom(E_{\beta+\gamma},\mathsf{S}(E_{\beta+\gamma})[-1])\\ &\to\Hom(E_{n\alpha+m\beta},\mathsf{S}(E_{\beta+\gamma})[-1])\to...
    \end{align*}
It follows that $\hom(E_{\beta+\gamma},E_f[2])=\hom(E_f,\mathsf S(E_{\beta+\gamma})[-2])\leq  \hom(E_{\beta+\gamma},\mathsf{S}(E_{\beta+\gamma})[-1])=4$. So  $E_f\not\in \BN^5(E_{\beta+\gamma},((n+1)\alpha+(m-2)\beta)[2])$.

By Lemma \ref{lem:homb2}.(1), for all $i=1,2,3,4$, the dimension of $\Pe{\BN^{=i}(E_{\beta+\gamma}[-2],(n+1)\alpha+(m-2)\beta)}{E_{\beta+\gamma}}$ is strictly less than that of $\P{E_{\beta+\gamma}[-1]}{n\alpha+m\beta}$. So the map $e^R_{E_{\beta+\gamma}[-1],n\alpha+m\beta}$ is well-defined generically.

As every $\ms(v)$ is irreducible by Corollary \ref{cor:msvirreducible}, by Proposition \ref{prop:extbirpair}, the map $e^R_{E_{\beta+\gamma}[-1],n\alpha+m\beta}$ induces a birational equivalence between $\Pe{E_{\beta+\gamma}[-1]}{n\alpha+m\beta}$ and $\Pe{(n+1)\alpha+(m-2)\beta}{E_{\beta+\gamma}}$.
\end{proof}
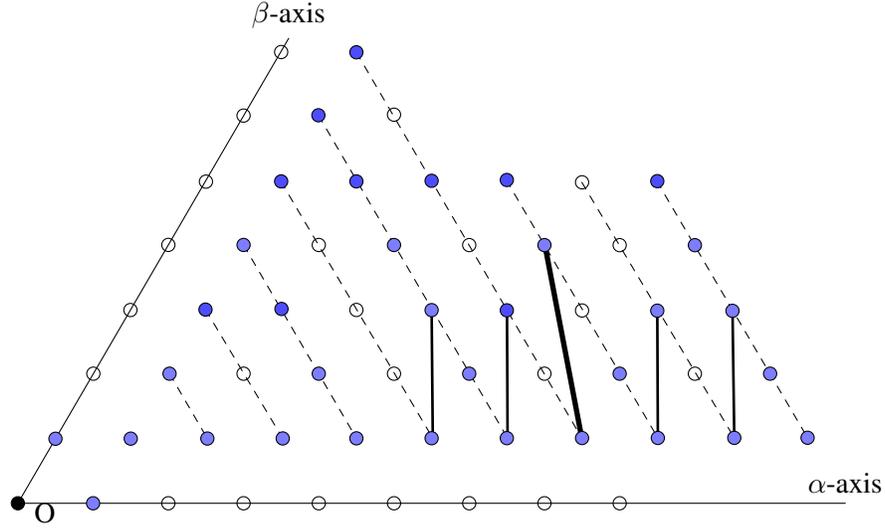
\begin{figure}
    \centering
\begin{tikzpicture}[line cap=round,line join=round,>=triangle 45,x=1cm,y=1cm]
\clip(-2,-0.4) rectangle (16,7);
\draw (0.08455094931971,0.14229617794195135) node[anchor=north west] {O};
\draw [domain=0:3.6] plot(\x,{(-0--1.7240222753698782*\x)/1.0032964027047553}) node[above]{$\beta$-axis};

\draw (0,0)--(11,0) node[above]{$\alpha$-axis};

\draw [dashed] (4.5,6)-- (7.496348019888999,0.8676972535331512);
\draw [dashed] (3.994923746417918,5.162108366315537)-- (6.509893925177644,0.8663962326463687);
\draw [dashed] (3.4956010788016636,4.28234938051547)-- (5.51557967100966,0.8650848451022006);
\draw [dashed] (2.999964267983041,3.4370165371417243)-- (4.500562937784039,0.8637461533400193);
\draw [dashed] (2.4939200776525094,2.5769902671455376)-- (3.5207945856832126,0.8624539501882197);
\draw [dashed] (2.015893980295695,1.7240222753698784)-- (2.5187550474035376,0.8611323738719238);
\draw [line width=1pt] (5.500169217346692,2.5696907391265946)-- (5.51557967100966,0.8650848451022006);
\draw [line width=1pt] (6.502007381668468,2.5672581577232925)-- (6.509893925177644,0.8663962326463687);
\draw [dashed] (6.4920467031418445,4.301255110780406)-- (8.507688759723262,0.8690310970776531);
\draw [dashed] (7.496927890294832,4.270335689637236)-- (9.5237275065176,0.870371136759628);
\draw [line width=1pt] (8.507688759723262,0.8690310970776531)-- (8.500151530178112,2.562406427703532);
\draw [line width=1pt] (9.5237275065176,0.870371136759628)-- (9.50014563445525,2.5599783238958445);
\draw [dashed] (8.494079222162027,4.285795400208821)-- (10.497988300094532,0.8716560760674096);
\draw [line width=2pt] (7.00623139622215,3.434743686484229)-- (7.496348019888999,0.8676972535331512);
\begin{scriptsize}
\draw [fill=] (0,0) circle (2.5pt);
\draw [] (2,0) circle (2.5pt);
\draw [] (3,0) circle (2.5pt);
\draw [] (4,0) circle (2.5pt);
\draw [] (5,0) circle (2.5pt);
\draw [] (1.0032964027047553,1.7240222753698782) circle (2.5pt);
\draw [fill=xdxdff] (0.4995870375608891,0.8584693206504792) circle (2.5pt);
\draw [fill=xdxdff] (1.5001639746556859,0.8597889679619681) circle (2.5pt);
\draw [fill=xdxdff] (2.5187550474035376,0.8611323738719238) circle (2.5pt);
\draw [fill=xdxdff] (3.5207945856832126,0.8624539501882197) circle (2.5pt);
\draw [] (1.4962055971819346,2.5710166717639726) circle (2.5pt);
\draw [fill=ududff] (2.4939200776525094,2.5769902671455376) circle (2.5pt);
\draw [fill=ududff] (3.5042316795615425,2.5852715097841363) circle (2.5pt);
\draw [] (2,3.4375835541527273) circle (2.5pt);
\draw [] (3,1.7240222753698788) circle (2.5pt);
\draw [fill=xdxdff] (4,1.724022275369879) circle (2.5pt);
\draw [fill=xdxdff] (4.500,0.8637461533400193) circle (2.5pt);
\draw [fill=xdxdff] (2.015893980295695,1.7240222753698784) circle (2.5pt);
\draw [fill=xdxdff] (1,0) circle (2.5pt);
\draw [fill=xdxdff] (5.5,0.8650848451022006) circle (2.5pt);
\draw [] (6,0) circle (2.5pt);
\draw [] (7,0) circle (2.5pt);
\draw [] (8,0) circle (2.5pt);
\draw [fill=xdxdff] (7.5,0.8676972535331512) circle (2.5pt);
\draw [fill=xdxdff] (6.5,0.8663962326463687) circle (2.5pt);
\draw [fill=xdxdff] (6.,1.7240222753698797) circle (2.5pt);
\draw [] (5,1.7240222753698793) circle (2.5pt);
\draw  (7,1.72402227536988) circle (2.5pt);
\draw [fill=ududff] (6.5,2.5672581577232925) circle (2.5pt);
\draw [fill=xdxdff] (5.5,2.5696907391265946) circle (2.5pt);
\draw  (4.5,2.5720867625756156) circle (2.5pt);
\draw  (6,3.435311794320988) circle (2.5pt);
\draw [fill=xdxdff] (5,3.4358803676891667) circle (2.5pt);
\draw  (4,3.436439992185827) circle (2.5pt);
\draw [fill=xdxdff] (3,3.4370165371417243) circle (2.5pt);
\draw  (2.5,4.279248389327641) circle (2.5pt);
\draw  (3.,5.157928200132907) circle (2.5pt);
\draw  (3.5,6.003605730619815) circle (2.5pt);
\draw [fill=ududff] (4.5,6) circle (2.5pt);
\draw [fill=ududff] (3.994923746417918,5.162108366315537) circle (2.5pt);
\draw (5,5.170034122944366) circle (2.5pt);
\draw [fill=ududff] (3.5,4.28234938051547) circle (2.5pt);
\draw [fill=ududff] (4.5,4.28234938051547) circle (2.5pt);
\draw [fill=ududff] (5.5,4.290275137144299) circle (2.5pt);
\draw [fill=xdxdff] (7,3.434743686484229) circle (2.5pt);
\draw  (8,3.4341762363987316) circle (2.5pt);
\draw [fill=xdxdff] (9,3.433608770480418) circle (2.5pt);
\draw [fill=ududff] (6.5,4.301255110780406) circle (2.5pt);
\draw  (7.5,4.270335689637236) circle (2.5pt);
\draw [fill=ududff] (8.5,4.285795400208821) circle (2.5pt);
\draw  (7.5,2.5648236333514425) circle (2.5pt);
\draw [fill=xdxdff] (8.5,2.562406427703532) circle (2.5pt);
\draw [fill=xdxdff] (9.5,2.5599783238958445) circle (2.5pt);
\draw [fill=xdxdff] (8,1.7240222753698802) circle (2.5pt);
\draw  (9,1.7240222753698806) circle (2.5pt);
\draw [fill=xdxdff] (10,1.7240222753698808) circle (2.5pt);
\draw [fill=xdxdff] (8.507688759723262,0.8690310970776531) circle (2.5pt);

\draw [fill=xdxdff] (9.5237275065176,0.870371136759628) circle (2.5pt);

\draw [fill=xdxdff] (10.497988300094532,0.8716560760674096) circle (2.5pt);
\end{scriptsize}
\end{tikzpicture}
    \caption{Stable birationality between $\ms(v)$'s: Blue dots stand for primitive characters. Proposition \ref{prop:stabbir2} proves the stable birationality between primitive characters on the dashed line segments. Proposition \ref{prop:stabbir3} is for the usual line segments. Proposition \ref{prop:stabbir4} is for the thickened line segment.}
    \label{fig:bir}
\end{figure}
The following proposition is to connect two more primitive characters that are not covered by Proposition \ref{prop:stabbir2} and \ref{prop:stabbir3}. It will be useful in the proof of Theorem \ref{thm:connectedfiber}.

\begin{Prop}\label{prop:stabbir4}
    Let $Y_3$ be a smooth cubic threefold and $\Ku(Y_3)$ be its Kuznetsov component. Then for every general $E\in\ms(\beta+2\gamma)$, we have a birational equivalence \begin{align}
e^R_{E[-1],5\alpha+4\beta}\colon\Pe{E[-1]}{5\alpha+4\beta}\dashrightarrow \Pe{7\alpha+\beta}{E}.
    \end{align}
\end{Prop}
\begin{proof}
The proof follows the same strategy as that for Proposition \ref{prop:stabbir2}.
   We first show that Assumption \ref{asspdag} holds with respect to $((\beta+2\gamma)[-1],5\alpha+4\beta)$ and $(7\alpha+\beta,\beta+2\gamma)$. It is clear that $\phi_\sigma(\beta+2\gamma)-\phi_\sigma(5\alpha+4\beta)\in(0,1)$; $\chi(5\alpha+4\beta,\beta+2\gamma)=6>0$; and $\chi(\beta+2\gamma,7\alpha+\beta)=-10<0$.

By Serre duality, the Brill--Noether locus $\BN(5\alpha+4\beta,(\beta+2\gamma)[1])\cong\BN(2\alpha-\gamma,5\alpha+4\beta)$. Let $F$ be a general object in $\ET(\alpha-\gamma,\alpha)$, then by Lemma \ref{lem:homb}.(2) and Lemma \ref{lem:homb2}.(3),  the space  $\BN(F,5\alpha+4\beta)\neq \ms(5\alpha+4\beta)$.  Similarly, by Serre duality, $\BN(\beta+2\gamma,(7\alpha+\beta)[2])\cong\BN((7\alpha+\beta)[2],(\alpha+2\beta)[2])$. Note that $\phi_\sigma(\alpha+2\beta)-\phi_\sigma(6\alpha-\beta)<\tfrac{1}{3}$, by Proposition \ref{lem:extspaceofsmallgapvw}, this is not the full space. \\

Let $E_\gamma\in\ms(\gamma)$, then by Proposition \ref{prop:stabbir1}, the  map
    $e^L_{E_{\gamma}[-1],{5\alpha+4\beta}}$ is a birational equivalence and in addition $\BN(E_{\gamma}[-2],7\alpha+\beta)\neq \ms(7\alpha+\beta)$.

     By Proposition \ref{prop:stabbir3}, we may fix a general $E_{\beta+\gamma}\in\ms(\beta+\gamma)$ such that the spaces
     \[\BN(E_{\beta+\gamma}[-2], 7\alpha+\beta)\neq \ms(7\alpha+\beta) \;\text{and}\;\BN(5\alpha+4\beta,E_\gamma[1])\neq \ms(5\alpha+4\beta),\]
     and the map $e^R_{E_{\beta+\gamma}[-1],6\alpha+3\beta}$ is a birational equivalence. So for general $E_{6,3}\in\ms(6\alpha+3\beta)$ and $h\in\Hom(E_{6,3},E_{\beta+\gamma})$, we have
     \[E_{7,1}:=\Cone(h)[-1]\in\ms(7\alpha+\beta)\setminus (\BN(E_\gamma[-2],7\alpha+\beta)\cup \BN(E_{\beta+\gamma}[-2],7\alpha+\beta)).\]

As $\Hom(E_\gamma[-2],E_{7,1})=0$, we have $\hom(E_\gamma[-1],E_{7,1})=\chi(-\gamma,7\alpha+\beta)=1$.

As $\hom(E_{\gamma}[-1],E_{6,3})=3$, for a general $g\in\Hom(E_{\gamma}[-1],E_{6,3})$  we have 
\begin{align*}    E_{5,4}:=\Cone(g)\in\ms(5\alpha+4\beta)\setminus \BN(5\alpha+4\beta,E_{\gamma}[1])\cup\BN(5\alpha+4\beta,E_{\beta+\gamma}[1])
\end{align*}
and $h\circ g\neq 0$.

   By the octahedral axiom,  there is a commutative diagram of distinguished triangles:

    \begin{center}
	\begin{tikzcd}
     & E_{7,1} \arrow[r,equal]\arrow{d}
		&   E_{7,1}\arrow{d}\\
		E_{\gamma}[-1] \arrow{r}{g} \arrow[d,equal]& E_{6,3}
		\arrow{d}{h} \arrow{r} & E_{5,4} \arrow[d] \arrow{r}{} & E_{\gamma}\arrow[d,equal] \\
		E_{\gamma}[-1] \arrow{r}  & E_{\beta+\gamma}\arrow{r}{}\arrow{d}{} & F_{\beta+2\gamma} \arrow{r}{} \arrow{d}{}& E_{\gamma}\\
 & E_{7,1}[1] \arrow[r,equal]
		&   E_{7,1}[1].
	\end{tikzcd}
\end{center}

As $h\circ g\neq 0$, by Lemma \ref{lem:extendstabobj}, the object $F_{\beta+2\gamma}\cong\Cone(h\circ g)$ is $\sigma$-stable.

As $E_{5,4}\not\in \BN(5\alpha+4\beta,E_{\gamma}[1])\cup\BN(5\alpha+4\beta,E_{\beta+\gamma}[1]) $, we have $E_{5,4}\not\in\BN(5\alpha+4\beta,F_{\beta+2\gamma}[1])$. 

As $E_{7,1}\not\in\BN(E_{\gamma}[-2],7\alpha+\beta)\cup\BN(E_{\beta+\gamma}[-2],7\alpha+\beta)$, we have $E_{7,1}\not\in\BN(F_{\beta+2\gamma}[-2],7\alpha+\beta)$.

In other words, the point $(F_{\beta+2\gamma}[-1],E_{5,4})\in\ms((\beta+2\gamma)[-1],5\alpha+4\beta)^\dag$; and the point $(E_{7,1},F_{\beta+2\gamma})\in\ms(7\alpha+\beta,\beta+2\gamma)^\dag$. By Proposition \ref{prop:extbirpair}, the map $e^R_{F_{\beta+2\gamma}[-1],5\alpha+4\beta}$ is a birational equivalence.
\end{proof}

The following lemma is a precise statement for the equivalences of primitive lattice points in Figure \ref{fig:bir}.
\begin{Lem}[See Figure \ref{fig:bir}]\label{lem:tree}
    Let $S=\{(a,b)\in\Z_{\geq 1}\times \Z_{\geq 1}\;|\;a+b\geq 6,a>b,\gcd(a,b)=1\}$. We say two points $(a,b)$ and $(c,d)$ in $S$ are equivalent to each other if $a+b=c+d$, or $(c,d)=(a+1,b-2)$, or $(a,b,c,d)=(5,4,7,1)$. Let this span an equivalence relation,  then all points in $S$ are equivalent.
\end{Lem}

\begin{proof}
    It is enough to show that for every $m\geq 6$ and $m\neq 8$, there exists a point $(a,b)\in S$ with $a+b=m$ and $(a-1,b+2)\in S$. 
    
    Let $p$ be the smallest odd prime number that does not divide $m+1$. By assumption, $m+1\geq 7$ and $\neq 9$, so we have $3\leq p<\tfrac{m+1}{2}$. It follows that $\gcd(m+1-p,p)=1$. By definition, $(m+1-p,p)\in S$. As $\gcd(m+1,m)=1$, $\gcd(m,q)=1$ for every odd prime $q<p$. It follows that $\gcd(m-p+2,p-2)=1$. Therefore, $(m-p+2,p-2)\in S$. The statement holds.
\end{proof}

\begin{Cor}\label{cor:stabbir}
    Let $Y_3$ be a smooth cubic threefold and $\Ku(Y_3)$ be its Kuznetsov component. Let  $v,w\in\Kn(\Ku(Y_3))$ be two primitive characters. If both $\chi(v,v),\chi(w,w)< -22$, or both $\chi(v,v),\chi(w,w)\in[-21,-16]$, or $\chi(v,v)=\chi(w,w)=-13$ (or resp. $\chi(v,v)=\chi(w,w)=-7$), then the spaces $\ms(v)$ and $\ms(w)$ are stably birational (resp. birational) to each other.
\end{Cor}
\begin{proof}
By Remark \ref{rem:Sswapabc}, we may assume that both $v$ and $w$ are  in the $(\alpha,\beta)$-sextant, that is, of the form $n\alpha+m\beta$ with $m,n>0$ and $\gcd(m,n)=1$. Note that $\chi(n\alpha+m\beta,n\alpha+m\beta)=-n^2-mn-m^2<-22$ (resp. $\in[-21,-16]$, $=-13$, $=-8$) if and only if $m+n\geq 6$ (resp. $=5$, $=4$, $=3$). 
    
    By Lemma \ref{lem:tree},  there exists a finite sequence of primitive characters $v_1,\dots,v_t$, with $v_i=n_i\alpha+m_i\beta$ in the $(\alpha,\beta)$-sextant satisfying $v_1=v$, $v_t=w$, and for every $i$, either $n_i+m_i=n_{i+1}+m_{i+1}$, or $(n_{i+1},m_{i+1})=(n_i+1,m_i-2)$, or $(n_i,m_i,n_{i+1},m_{i+1})=(5,4,7,1)$. In either case, by Proposition \ref{prop:stabbir2}, \ref{prop:stabbir3}, and \ref{prop:stabbir4} respectively, there exists a $\sigma$-stable object $E$ such that we have the birational map
\[e^R_{E[-1],v_{i+1}}\colon\Pe{E[-1]}{v_{i+1}}\dashrightarrow \Pe{v_{i}}E\]
    or $e^R_{E[-1],v_{i}}$ when $n_i>n_{i+1}$. For each primitive $v_i$, the space $\ms(v_i)\times Y_3$ admits a universal object. It follows that $\Pe{v_i}E$ (or $\Pe{E[-1]}{v_i}$) is the projectivization of a locally free sheaf over $\ms(v_i)$, in particular, birational to $\ms(v_i)\times \mathbf P^r$ for some $r\geq 0$.  So  $\ms(v_1=v)$ is stably birational to $\ms(v_t=w)$.\\
    
    In the case that $\chi(v,v)=\chi(w,w)=-8$, we may assume that $v=\alpha+2\beta$ and $w=2\alpha+\beta$. Let $E\in\ms(\gamma)$, by Proposition \ref{prop:stabbir1}, we have the birational map $e^R_{E[-1],v}:\Pe{E[-1]}{\alpha+2\beta}\dashrightarrow\Pe{2\alpha+\beta}E$. Note that $\Pe{E[-1]}{\alpha+2\beta}$ is birational to $\ms(\alpha+2\beta)$ and $\Pe{2\alpha+\beta}E$ is birational to $\ms(2\alpha+\beta)$, the statement follows.
\end{proof}

\section{\AJ map and properties of the fibers}\label{sec7}
This section studies the properties of the fibers of the \AJ map from the moduli space to the intermediate Jacobian. In Section \ref{sec:irr_fiber}, we show that the \AJ map has irreducible general fiber. Then we show that these fibers are Fano varieties with primitive canonical class. Our approach is based on two steps: in Section \ref{sec:rel_NS}, we show that the morphism $\IJ_v$ has relative Picard rank one. In particular, the canonical bundle of a fiber is proportional to the restriction of an ample line bundle. Then in Section \ref{sec:deg_K}, we directly compute the intersection number of the canonical divisor with a curve consisting of objects that are extended by two fixed objects. This will show that the general fiber of $\IJ_v$ is Fano. In particular, $\IJ_v$ gives the maximally rationally connected (MRC) quotient of the moduli space.

\subsection{Irreducibility of the fibers}\label{sec:irr_fiber}

We start this section with the following observation.

\begin{Rem}\label{rem:compatibleextandajmap}
We have a commutative diagram as in \eqref{eq524} whenever $e_{v,w}$ is generically well-defined. Moreover, the birational maps $e^R_{E_v,w}$ and $e^L_{v,E_w}$ are compatible with the $\AJ$ map in the sense that they map a fiber to another whenever generically well-defined on the fiber:
 \begin{equation}\label{eq524}
	\begin{tikzcd}
 \ms(v+w)\arrow{rr}{\IJ_{v+w}} && J_{v+w}(Y_3)\\
 \Pe vw \arrow[u,dashed]{}{e_{v,w}} \arrow{r}{\pi_{v,w}} \arrow[d,shift left, dashed]{}{e^R_{v,w}} & \ms(v)\times \ms(w) \arrow{r}{\Phi_v\times \Phi_w} & J_v(Y_3)\times J_w(Y_3) \arrow{u}{+}\arrow{d}{\mu}\\
 \Pe{v+w}{v[1]} \arrow[u,shift left, dashed]{}{e^L_{v+w,v[1]}} \arrow{r}{\pi_{v+w,v[1]}} & \ms(v+w)\times \ms(v[1])\arrow{r}{\IJ_{v+w}\times \IJ_{v[1]}}& J_{v+w}(Y_3)\times J_{v[1]}(Y_3).
	\end{tikzcd}
\end{equation}
For every $c_v\in J_v(Y_3)$ and $c_w\in J_w(Y_3)$,  the morphism $\mu$ maps $(c_v,c_w)$ to $(c_v+c_w,-c_v)$. For every $E_v\in \ms(v,c_v)$, if $e^R_{E_v,w}$ is well-defined on a general point in $\ms(w,c_w)$, then by Proposition \ref{prop:extbirpair}, the birational map $e^R_{E_v,w}$ restricts to a birational map between the fibers: 
\begin{align*}
    e^R_{E_v,c_w}\colon\Pe{E_v}{c_w}\dashrightarrow \Pe{c_w+c_v}{E_v[1]}.
\end{align*}
In particular, the fiberwise version of Propositions \ref{prop:stabbir2}, \ref{prop:stabbir3}, and \ref{prop:stabbir4} holds.
\end{Rem}
\begin{Thm}\label{thm:connectedfiber}
     Let $Y_3$ be a smooth cubic threefold and $\Ku(Y_3)$ be its Kuznetsov component. Then for every $v\in\Kn(\Ku(Y_3))$ with $\chi(v,v)\leq -4$,  the map $\IJ_v\colon M_\sigma(v)\to J_v(Y_3)$ is surjective with connected fibers.
     
     For primitive $v$, $w$ with $\chi(v,v),\chi(w,w)<-22$, general $c_v\in J_v(Y_3)$ and $c_w\in J_w(Y_3)$, the general fibers $\ms(v,c_v)$ and $\ms(w,c_w)$ are stably birational to each other.
\end{Thm}
\begin{proof}
    We first prove the first statement for all $v$ of the form $n\beta$, $n\beta+\alpha$, and $n\beta+\gamma$  by induction on $n$. By Remark \ref{rem:Sswapabc}, at each step of the induction, the statement will also hold for $n\alpha$, $n\gamma$,  $n\alpha+\beta$, $n\alpha-\gamma$, $n\gamma+\beta$, and $n\gamma-\alpha$.

    When $n=1$, the space $\ms(\beta)$ is the Fano variety of lines in $Y_3$ and $\Phi_{\beta}$ is a closed embedding. The space $\ms(\beta+\alpha)\cong\ms(\beta+\gamma)$, and as that in Example \ref{eg:ab}, the map $\Phi_{\beta+\gamma}$ is the resolution of the Theta divisor in $J_{\beta+\gamma}(Y_3)$. In particular, its image is connected.

    When $n=2$, as $\chi(\beta+\gamma,\alpha)=-1$, the space $\Pe{\alpha}{\beta+\gamma}\cong \ms(\alpha,\beta+\gamma)^\dag$, which is an open subset of $\ms(\alpha)\times \ms(\beta+\gamma)$. By Proposition \ref{prop:stratofmbeta}, there is a dominant rational map $e_{\alpha,\beta+\gamma}\colon\ms(\alpha)\times \ms(\beta+\gamma)\dashrightarrow \ms(2\beta)$. By \eqref{eq524}, for a general point $p\in J_{2\beta}(Y_3)$, the fiber \[(\Phi_{2\beta}\circ e_{\alpha,\beta+\gamma})^{-1}(p)\sim(+\circ (\Phi_\alpha\times\Phi_{\beta+\gamma}))^{-1}(p)\sim F(Y_3)\cap(\tau_p(\Theta)) \text{ in } J(Y_3).\]
    Here we write $\sim$ for birational equivalence, and $\tau_p$ for the map $J(Y_3)\to J(Y_3)\colon x\mapsto x+p$. By Bertini Theorem, the variety $F(Y_3)\cap(\tau_p(\Theta))$ is non-empty and irreducible for general $p$. Note that $\dim(\ms(2\beta))=5=\dim(J_{2\beta}(Y_3))$ and the general fiber of $\Phi_{2\beta}\circ e_{\alpha,\beta+\gamma}$ is irreducible, so the map $\Phi_{2\beta}$ must be dominant with degree one, in other words, a birational equivalence. In particular, the first statement holds for $2\beta$. (In fact, one can explicitly show that $M_\sigma(2\beta)\cong \mathrm{Bl}_{F(Y_3)}J(Y_3)$.)

    Assuming the first statement holds for $n\beta$, $n\geq 2$, we claim that it holds for $n\beta+\alpha$ and $n\beta+\gamma$. By Theorem \ref{thm:stratcubic3}, the rational map $e_{\alpha,n\beta}$ is dominant. By induction, the map $\Phi_{n\beta}$ is surjective with connected fiber. Note that  $\Phi_\alpha$ is a closed immersion of $\ms(\alpha)\cong F(Y_3)$. By Lemma \ref{lem:connectdfiber}, in the diagram \eqref{eq524}, every fiber of $+\circ(\Phi_\alpha\times \Phi_{n\beta})\circ \pi_{\alpha,n\beta}$    is connected. As $\Pe{\alpha}{n\beta}$ is irreducible, a general fiber is irreducible. It follows that $\Phi_{n\beta+\alpha}$ is surjective, and  a general fiber of $\Phi_{n\beta+\alpha}$ is irreducible. By the same argument, the first statement holds for both $n\beta+\alpha$ and $n\beta+\gamma$.

    Assuming the first statement holds for $n\beta+\gamma$ for some $n\geq 2$, then by Proposition \ref{prop:stratofmbeta}, the rational map $e_{\alpha,n\beta+\gamma}$ is dominant. By  the same argument as above, the first statement holds for $(n+1)\beta$.

    Now by induction, the first statement holds for $n\alpha$, $n\beta$, and $n\gamma$ for all $n \geq 2$. For every $n\alpha+m\beta$ with $n,m\geq 2$, by Theorem \ref{thm:stratcubic3}, the rational map $e_{n\alpha,m\beta}$ is dominant. By the same argument as above, the map $\Phi_{n\alpha+m\beta}$ is surjective with connected fiber. By Remark \ref{rem:Sswapabc}, the first statement holds for all $v$ with $\chi(v,v)\leq -4$.\\

    The second statement follows the same argument as that for Corollary \ref{cor:stabbir}, we expand a few more details for solidarity.
    
    By Proposition \ref{prop:stabbir2} and Remark \ref{rem:compatibleextandajmap}, there exists $c\in J_{\alpha+5\beta}(Y_3)$ such that for a general object $E\in \ms(4\gamma)$, the birational map 
    \begin{align}\label{eq534}
        e^R_{E[-1],c}\colon\Pe{E[-1]}c\dashrightarrow\Pe{c-c_2(E)}E
    \end{align} is well-defined. As $5\alpha+\beta$ is primitive, the space $\Pe{5\alpha+\beta}E$ is the projectivization of a locally free sheaf over $\ms(5\alpha+\beta)$.  For general $c'\in J_{5\alpha+\beta}(Y_3)$, the space $\Pe{c'}E$ is the projectivization of a locally free sheaf over $\ms(5\alpha+\beta,c')$. It is therefore birational to $\ms(5\alpha+\beta,c')\times \mathbf P^3$.

    By \eqref{eq534}, for all general $E\in \ms(4\gamma)$, the spaces $\ms(5\alpha+\beta,c-c_2(E))\times \mathbf P^3$ are birational to each other.     As the map $\IJ_{4\gamma}$ is surjective, the spaces $\ms(5\alpha+\beta,c')$ are stably birational to each other for general $c'\in J_{5\alpha+\beta}(Y_3)$.

    As in Corollary \ref{cor:stabbir}, we may assume $v=n\alpha+m\beta$ form some $m,n>0$, $\gcd(m,n)=1$, and  $m+n\geq 6$.  By Lemma \ref{lem:tree},  there exists a finite sequence of primitive characters $v_1,\dots,v_n$, with $v_i=n_i\alpha+m_i\beta$ in the $(\alpha,\beta)$-sextant satisfying $v_1=v$, $v_n=5\alpha+\beta$. We have birational maps
\[e^R_{E[-1],v_{i+1}}\colon\Pe{E[-1]}{v_{i+1}}\dashrightarrow \Pe{v_{i}}E\]
    or $e^R_{E[-1],v_{i}}$ when $n_i>n_{i+1}$. For a general $c_i\in J_{v_i}(Y_3)$, by Remark \ref{rem:compatibleextandajmap}, the maps $e^R$'s restrict to $\Pe{E[-1]}{c_i}$ and $\Pe{c_i}E$. As all $v_i$'s are primitive,  the spaces $\Pe{E[-1]}{c_i}$ and $\Pe{c_i}E$ are both birational to $\ms(v_i,c_i)\times \mathbf P^r$ for some $r\geq 0$. So  the space $\ms(v_1=v,c_1)$ is stably birational to $\ms(v_n=5\alpha+\beta,c_n)$ for a general $c_n\in J_{5\alpha+\beta}(Y_3)$. Therefore for all general $c\in J_v(Y_3)$, the spaces $\ms(v,c)$ are stably birational to each other.
\end{proof}

\begin{Lem}\label{lem:connectdfiber}
    Let $A$ be an abelian variety, $\Phi_i\colon M_i\to A$ be two proper morphisms such that $\Phi_1$ is surjective, the image of $\Phi_2$ is connected in $A$, the non-empty fibers of $\Phi_i$'s are connected. Denote by $\pi$ the composition $M_1\times M_2\xrightarrow{(\Phi_1,\Phi_2)} A\times A\xrightarrow{+}A$, then the fibers of $\pi$ are connected.
\end{Lem}
\begin{proof}
    Note that the map $M_1\times M_2\to A\times \im\Phi_2\xrightarrow{+}A$ is a composition of proper and surjective maps with connected fibers. So the fibers of $\pi$ are connected.
\end{proof}

\begin{Rem}
    By \cite[Theorem 3.2]{Shinder_2022}, the stably birational equivalence holds between any smooth fibers of the \AJ map of the expected dimensions.
\end{Rem}

\subsection{The relative Neron--Severi group}\label{sec:rel_NS}

We start with an analogue of \cite{beauville:cohomology}. Let $v$ be a primitive character in $\Kn(\Ku(Y_3))$, recall that $\ms(v)$ stands for the moduli space of $\sigma$-stable objects in $\Ku(Y_3)$ with character $v$. Note that we have the following properties:
\begin{enumerate}
    \item There is a universal family $\cU\to \ms(v)\times Y_3$; 
    \item For any $E,F\in \ms(v)$, $\Ext^i(E,F)=0$ except for $i=0,1$. 
\end{enumerate}
\begin{Lem}\label{lem:beauville}
    The cohomology algebra $H^*(\ms(v);\mathbb{Q})$ is generated by the K\"unneth components over $\ms(v)$ of Chern classes $c_i(\cU)$. 
\end{Lem}

\begin{proof} For simplicity we write $M=\ms(v)$ and $m= \dim M$ in the proof. Let $\delta$ denote the class of the diagonal in $M\times M$. 
Note that if $\delta$ can be written in the form 
\begin{align}\label{eq:diagonal}\delta=\sum_i \operatorname{pr}_1^*\eta_i\cdot \operatorname{pr}_2^*\xi_i,\end{align}
where $\operatorname{pr}_1$ and $\operatorname{pr}_2$ denote the projections, then the cohomology algebra is generated by $\{\,\eta_i\,\}_i$. 
    
Let $\pi_{ij}$ denote the projection from $M\times M\times X$ onto the $i$-th and $j$-th factors. Let 
    $\mathcal{H}=\RlHom_{\pi_{12}}(\pi_{13}^*\cU,\pi_{23}^*\cU)[1]$.
By \cite[Proposition 5.4]{Bridgeland-Maciocia:K3Fibrations}, the object $\mathcal{H}$ can be represented by a complex 
    $K^\bullet=\{\,K^{-1}\xrightarrow{u}K^0\,\}$
of locally free sheaves in degrees $-1$ and $0$.
For a point $x=(E,F)\in M\times M$, we have an exact sequence
\begin{align*}
    0\to \Hom(E,F)\to K^{-1}(x)\xrightarrow{u(x)} K^0(x)\to \Ext^1(E,F)\to 0
\end{align*}
by Cohomology and Base Change. Note that $\Hom(E,F)\not=0$ if and only if $x\in \Delta\subset M\times M$, and when it is nonzero, it has dimension $1$. Thus by Porteous' formula, the cohomology class of the diagonal is a multiple of $c_m(K^0-K^{-1})=c_m(\mathcal{H})$. 
By Grothendieck--Riemann--Roch, the diagonal indeed can be written in the form \eqref{eq:diagonal}, with $\{\,\eta_i\,\}_i$ being the K\"unneth components of the Chern classes. 
\end{proof}
The statement still holds if we replace ``Chern classes'' by ``Chern characters'', which is more convenient for our next application.

Now we are ready to prove the following result on the first two Betti numbers of $\ms(v)$.

\begin{Prop}\label{prop:betti}
    For every primitive $v\in\Kn(\Ku(Y_3))$ with $\chi(v,v)<-4$, we have Betti numbers of the moduli space $b_1(\ms(v))=10$ and $b_2(\ms(v))=46$.
\end{Prop}

\begin{proof}

Define the K\"unneth components of $\ch_i(\cU)\in H^{2i}(\ms(v)\times Y_3; \Q)$ by the following equality:
\begin{align*}
    \ch_i(\cU)=a_i\otimes 1+e_i\otimes H + \sum_{j=1}^{10} f_{i,j}\otimes \rho_j + g_i\otimes H^2 + h_i\otimes H^3,
\end{align*}
where $H$ is the hyperplane class of $Y_3$, and $\{\rho_j\}_{j=1,...,10}$ is a basis of $H^3(Y_3;\mathbb Q)$.  We can choose a universal family $\cU$ such that $a_1=0$. 
By Lemma \ref{lem:beauville}, the space $H^1(\ms(v);\mathbb{Q})$ is generated by $\{\,f_{2,j}\,\}_j$ and $H^2(\ms(v);\mathbb{Q})$ is generated by 
    $\{\,e_2,g_3,h_4, f_{2,j_1}\cup f_{2,j_2}\,\}_{j_1,j_2}$.
In particular, $b_1(\ms(v))\leq 10$.

Let $p\colon \ms(v)\times Y_3\to \ms(v)$ denote the projection. As the objects are in $\Ku(Y_3)$, we have 
\begin{align*}
    Rp_*\cH om(\cO, \cU)=0=Rp_*\cH om(\cO(H), \cU).
\end{align*}
Applying the Grothendieck--Riemann--Roch Theorem, we obtain 
\begin{align*}
    p_*(\ch(\cU)\td Y_3)=p_*(\ch(\cU)\ch(\cO(-H))\td Y_3)=0. 
\end{align*}
Recall that the Todd class of a cubic threefold is $\td Y_3=1+H+\frac{2}{3}H^2+\frac{1}{3}H^3$. Considering the degree 2 part of the equalities, we obtain two relations among $e_2, g_3,$ and $h_4$:
\begin{align*}
    \frac{1}{3}h_4+ g_3+\frac{2}{3} e_2&=0,\\
    \frac{1}{3}h_4+\frac{1}{6}e_2&=0.    
\end{align*}
Hence $H^2(\ms(v);\mathbb{Q})$ is generated by 
    $\{\,e_2, f_{2,j_1}\cup f_{2,j_2}\,\}_{j_1,j_2}$.
So $b_2(\ms(v))\leq 46$.

On the other hand, by Theorem \ref{thm:connectedfiber}, when $\chi(v,v)< -4$, the Abel--Jacobi map $\IJ_v\colon \ms(v)\to J_v(Y_3)$ is surjective. Hence the pullback of cohomology $\IJ_v^*\colon H^i(J_v(Y_3)) \to H^i(\ms(v))$ is injective (cf. \cite[Page 177]{Voisin:hogedbook}). Moreover, since the morphism $\IJ_v$ now has positive dimensional fibers, the relative Neron--Severi group has to be of positive rank, we must have $b_2(\ms(v))>b_2(J_v(Y_3))$. Note that $b_1(J_v(Y_3))=10$ and $b_2(J_v(Y_3))=45$, now the statement follows.
\end{proof}

In particular, the proof gives the following corollary, which we state separately for later reference.

\begin{Cor}\label{cor:relpic1}
     For every primitive $v\in\Kn(\Ku(Y_3))$ with $\chi(v,v)< -4$, the relative Neron--Severi group of the morphism $\IJ_v\colon \ms(v)\to J_v(Y_3)$ has rank $1$.
\end{Cor}

\subsection{Degree of the tangent bundle}\label{sec:deg_K}
\subsubsection{Quiver model for the formal neighborhood of extension locus}
Let $Q$ be the quiver consisting of two vertices $v_1$ and $v_2$ with $a_{ij}$ arrows/loops from $v_i$ to $v_j$. For $i=1,2$, we denote by $S_i$ the simple representation with $\C$ at $v_i$, and $0$ at the other vertex and all arrows/loops. 

The category $\Rep(Q)$ is of homological dimension one.  The $\Ext^\bullet(S_i,S_j)$-algebra is the same as the path algebra associated to $Q$. In other words, we have $\Hom_{\Db(\Rep(Q))}(S_i,S_j[k])=0$ when $k\neq 0$ or $1$; $\Hom(S_i,S_j)=\C^{\delta_{ij}}$; $\Hom(S_i,S_j[1])=\C^{a_{ij}}$; and the composition of any two elements in $\Hom(S_\bullet,S_\bullet[1])$ is $0$.

Denote by $\{\be_1,\dots,\be_r\}$, where $r=a_{21}$, the basis for $\Hom_{\Rep(Q)}(S_2,S_1[1])$ corresponding to the arrows from $v_2$ to $v_1$. Let $\mathbf 1=(1,1)$ be a dimension vector on the vertices of $Q$ and $\theta=(-1,1)$ be the weight in the sense of King \cite{King:QuiverStability}.  For every $f=a_1\be_1+\dots a_r\be_r\in \Hom(S_2,S_1[1])$, the object  $\Cone(f)[-1]$ is the representation of $Q$ of dimension vector $\mathbf 1$, timing $a_i$ on the $i$-th arrow from $v_2$ to $v_1$, $0$ on all arrows from $v_1$ to $v_2$ and all loops. The object $\Cone(f)[-1]$ is $\theta$-stable if and only if $f\neq 0$.

 Denote by $\cR^s_\theta(Q,\mathbf 1)$ the variety parametrizing isomorphic classes of $\theta$-stable representations of dimension vector $\mathbf 1$. Let $Z$ be the subvariety in $\cR^s_\theta(Q,\mathbf 1)$ parametrizing objects of the form $\Cone(f)[-1]$ for all $0\neq f\in\Hom(S_2,S_1[1])$.  It is clear that $Z\cong \mathbf P^{r-1}$.

\begin{Lem}\label{lem:degofcanquivermodel}
    Adopting the notation as above, assume that $a_{21}\geq 2$, then $\omega_{\cR^s_\theta(Q,\mathbf 1)}|_Z\cong \cO_Z(a_{12}-a_{21})$.
\end{Lem}
\begin{proof}
By the theory of stability of quiver representations introduced by King \cite{King:QuiverStability}, we have $\cR^s_\theta(Q,\mathbf 1):=\Rep^s_{\theta}(Q,\mathbf 1)/(G_{\mathbf 1}/\C^\times)$. Here the space of $\theta$-stable representations is identified as  \[\Rep^s_{\theta}(Q,\mathbf 1)=\C^{\oplus a_{11}}\times \C^{\oplus a_{22}}\times (\C^{\oplus a_{21}}\setminus\{0\})\times \C^{\oplus a_{12}}.\]
The group $G_{\mathbf 1}/\C^\times\cong \C^\times $ acts  on each factor $\C^{\oplus a_{ij}}$ with weight $i-j$. The objects $\Cone(f)[-1]$'s form the subvariety $\{0\}\times\{0\}\times (\C^{\oplus a_{21}}\setminus\{0\})\times \{0\}$.

As the group action on the factors $\C^{\oplus a_{ii}}$ is trivial, this part does not affect the degree of $\omega_{\cR^s_\theta(Q,\mathbf 1)}$ on $Z$. We may reduce the computation to $X:=\left((\C^{\oplus r}\setminus\{0\})\times \C^{\oplus t}\right)/\C^\times$ with $Z$ being locus of $(\C^{\oplus r}\setminus\{0\})\times\{0\}$. We write $\C[x_1,\dots,x_r,y_1,\dots,y_t]$ for the weighted coordinate ring of $X$. The $\C^\times$ action is given as $k\cdot(x_1,\dots,x_r,y_1\dots,y_t):=(kx_1,\dots,kx_r,k^{-1}y_1,\dots,k^{-1}y_t)$. We denote  $U_i:=X\setminus \{x_i=0\}\cong \mathrm{Spec} (\C[x_1/x_i,\dots,x_r/x_i,x_iy_1,\dots,x_iy_t])$ as the affine open subset on $X$. Then on $U_1\cap U_2$, we have the following section  $s_1\in H^0(U_1\cap U_2,\omega_X|_{U_1\cap U_2})$: 
    \begin{align*}
     s_1:= &  d\left(\frac{x_2}{x_1}\right)\wedge\dots \wedge  d\left(\frac{x_r}{x_1}\right)\wedge d(x_1y_1)\wedge\dots\wedge d(x_1y_t) \\
        =& \left(\frac{1}{x_1}dx_2-\frac{x_2}{x^2_1}dx_1\right)\wedge\dots\wedge(x_1dy_1+y_1dx_1)\wedge\dots\wedge(x_1dy_t+y_tdx_1)\\
        = & x^{t+1-r}_1 dx_2\wedge\dots \wedge dx_r\dots \wedge dy_t - x_2 x_1^{t-r} dx_1\wedge dx_3\wedge \dots dx_r\wedge\dots dy_t+\dots\\
        = & -\left(\frac{x_1}{x_2}\right)^{t-r} d\left(\frac{x_1}{x_2}\right)\wedge\dots \wedge  d\left(\frac{x_r}{x_2}\right)\wedge d(x_2y_1)\wedge\dots\wedge d(x_2y_t).
    \end{align*}
    The section  $s_1$  extends holomorphically on $U_1$ without any zeros. It has a pole along $x_2=0$ of degree $t-r$ (or a zero of degree $r-t$). In particular, $\omega_X$ is with degree $r-t$ while restricting to $Z$, which is given by the equations $y_1=\dots=y_t=0$. The statement follows.
\end{proof}

\begin{Prop}\label{prop:localneighbourofcurve}
     Let $v_i\in\Kn(\Ku(Y_3))$, $i=1,2$, and  $E_i\in \ms(v_i)$, satisfying 
     \begin{itemize}
         \item $\Hom(E_i,E_j[k])= 0$  when $k\neq 0,1$, or $k=0$ and $i\neq j$;
         \item $\hom(E_2,E_1[1])=r\geq 2$ and for every $0\neq f\in \Hom(E_2[-1],E_1)$, the object $\Cone(f)$ is $\sigma$-stable.
     \end{itemize}
    Denote by $\omega$ the canonical divisor on $\ms(v_1+v_2)$. Then $\ET(E_1,E_2)\cong \mathbf P^{r-1}$ and $\omega|_{\ET(E_1,E_2)}\cong \cO_{\mathbf P^{r-1}}(d)$, where $d=\hom(E_1,E_2[1])-r$.
\end{Prop}
\begin{proof}
We first show that $\ET(E_1,E_2)\cong \mathbf P^{r-1}$. Consider the map \[e\colon\P(\Hom(E_2,E_1[1]))\to \ms(v_1+v_2)\colon[f]\mapsto \Cone(f)[-1]=:E_f.\] 
By the second assumption, the map is well-defined.  Assume that $\Cone(f)\cong \Cone(f')$, then as $\Hom(E_1,E_2)=0$, by the second part of the argument for Proposition \ref{prop:extmapfinite}, the morphism $f=f'$ up to a scalar. Therefore the map $e$ is injective.

We then show that $e$ separates tangent vectors, and hence is a closed embedding. For a point $[f]\in \P(\Hom(E_2,E_1[1]))$, its tangent space $T_f(\P(\Hom(E_2,E_1[1])))$ is canonically identified as the quotient subspace $\Hom(E_2,E_1[1])/\C f$. The tangent space at $E_f$ on $\ms(v_1+v_2)$ is identified as $\Hom(E_f,E_f[1])$.  For an element $g\in\Hom(E_2,E_1[1])/\C f$, the induced pushforward of $e$ maps $g$ to the composition \[de_f(g)\colon E_f\xrightarrow{f^L} E_2\xrightarrow{g}E_1[1]\xrightarrow{f^R} E_f[1]\] up to a scalar. (It follows from the first assumption that $\hom(E_f,E_2)=\hom(E_1,E_f)=1$. So the morphisms $f^L$ and $f^R$ are uniquely determined up to scalar.)
   \begin{equation}\label{eqdiagram6}
	\begin{tikzcd}
		 E_f \arrow{r}{f^L} \arrow[d,dashed]{}{h} & E_2
		\arrow{d}{g} \arrow{r}{f} & E_1[1] \arrow{r}{+}& \;\\
		 E_2 \arrow{r}{f} & E_1[1]
		 \arrow{r}{f^R} & E_f[1] \arrow{r}{+}& .
	\end{tikzcd}
\end{equation}
Assume that $de_f(g)=0$, then there exists $h\in \Hom(E_f,E_2)$ that commutes the diagram. As $h\sim f^L$, we have $g\circ f^L= f\circ h=0$. By the first assumption $g\circ f^L=0$ if and only if $g\sim f$, in other words, the element $g= 0$ in $\Hom(E_2,E_1[1])/\C f$. So the pushforward map $de_f$ is injective. In particular, the space $\ET(E_1,E_2)=e(\P(\Hom(E_2,E_1[1]))\cong \P(\Hom(E_2,E_1[1]))\cong \mathbf P^{r-1}$.\\

To compute $\omega|_{\ET(E_1,E_2)}$, we claim that the tangent space at $E_f$, which is $T_{E_f}(\ms(v_1+v_2))\cong\Hom(E_f,E_f[1])$, admits the following filtration up to the choice of $f^R$ and $f^L$:
\begin{align}\label{eq:filofefef1}
    0\subset V_1\subset V_2\subset V_3=\Hom(E_f,E_f[1]).
\end{align} Here $V_1\cong \Hom(E_2,E_1[1])/\C f$, $V_2/V_1\cong \Hom(E_1,E_1[1])\oplus \Hom(E_2,E_2[1])$, and $V_3/V_2\cong \Hom(E_1,E_2[1])$. 

To get \eqref{eq:filofefef1}, we consider the map $\pi_{f}\colon\Hom(E_f,E_f[1])\to \Hom(E_1,E_2[1])\colon k\mapsto f^L[1]\circ k\circ f^R[-1]$. As $\Hom(E_i,E_j[2])=0$ for all $i,j\in\{1,2\}$, the map $\pi_f$ is surjective. 

For every $k\in \ker \pi_f$, there exists $k_1\in \Hom(E_1,E_1[1])$ that commutes the diagram \eqref{eqdiagram62} below. As $\Hom(E_1,E_2)=0$, such a morphism $k_1$ is unique. Similarly, there exists unique $k_2\in\Hom(E_2,E_1[1])$ that commutes the diagram. This gives the map \[\pi_{f,1,2}\colon\ker \pi_f\to \Hom(E_1,E_1[1])\oplus \Hom(E_2,E_2[1])\colon k\mapsto (k_1,k_2).\]
\begin{equation}\label{eqdiagram62}
	\begin{tikzcd}
		 E_1 \arrow{r}{f^R[-1]} \arrow[d,dashed]{}{k_1} & E_f
		\arrow{d}{k} \arrow{r}{f^L} \arrow[dl,dashed]{}{k'} & E_2 \arrow{r}{+}\arrow[d,dashed]{}{k_2}& \;\\
		 E_1[1] \arrow{r}{f^R} & E_f[1]
		 \arrow{r}{f^L[1]} & E_2[1] \arrow{r}{+}& .
	\end{tikzcd}
\end{equation}
The map $\pi_{f,1,2}$ is surjective:  for every  $k_1\in \Hom(E_1,E_1[1])$, as $\Hom(E_2,E_1[2])=0$, it lifts to $k'\in \Hom(E_f,E_1[1])$ such that $k_1=k'\circ f^R[-1]$. Let $k=f^R\circ k'$, then $k'\in \ker \pi_f$ and $\pi_{f,1,2}(k)=(k_1,0)$. Similarly, $\pi_{f,1,2}$ is surjective onto the second factor. So $\pi_{f,1,2}$ is surjective.

The space $\ker\pi_{f,1,2}$ has dimension equal to $\hom(E_2,E_1[1])-1$ and it consists of morphisms satisfying $k\circ f^R[-1]=f^L[1]\circ k=0$. By the first part of the proof, this is just the sub tangent space $T_{E_f}(\ET(E_1,E_2))=de_f(\Hom(E_2,E_1[1])/\C f)$. To sum up, we get the filtration \eqref{eq:filofefef1} for $\Hom(E_f,E_f[1])$. 

Fix a family of $f^R$ and $f^L$ on $\Hom(E_2,E_1[1])\setminus\{0\}$ that is compatible with the action of $\End(E_1,E_1)\times \End(E_2,E_2)/\C^\times$, in other words, for every $c_i\in \End(E_i,E_i)$, $((c_1,c_2)f)^R=c_1^{-1}c_2f^R$ and $((c_1,c_2)f)^L=c_2^{-1}c_1f^L$. We get a family version of \eqref{eq:filofefef1} over $\Hom(E_2,E_1[1])\setminus\{0\}$ as $0\subset \tilde \cV_1\subset\tilde \cV_2\subset\tilde\cV_3\cong \cHom(E_f,E_f[1]) |_{\Hom(E_2,E_1[1])\setminus\{0\}}$. As $f^R$ and $f^L$ are compatible with  the $\End(E_1,E_1)\times \End(E_2,E_2)/\C^\times$-action, this filtration descends to the filtration   $0\subset \cV_1\subset\cV_2\subset\cV_3\cong T\ms(v_1+v_2)|_{\ET(E_1,E_2)}$ after taking the quotient.  Taking the top exterior product on the filtration, we get \begin{align}
    \omega^{-1}|_{\ET(E_1,E_2)} & \cong \bigwedge \!^{\text{top}}\cV_1\otimes \bigwedge \!^{\text{top}}(\cV_2/\cV_1)\otimes \bigwedge \!^{\text{top}}(\cV_3/\cV_2).
\end{align}
Note that the divisor $\omega^{-1}|_{\ET(E_1,E_2)}$ does not depend on the extensions between these factors. The bundles $\cV_{i+1}/\cV_i$ only depend on the quiver $\Ext^\bullet(E_i,E_j)$-algebra as they are descended from the space $\Hom(E_2,E_1[1])\setminus\{0\}$. So the degree of $\omega^{-1}|_{\ET(E_1,E_2)}$ is the same as that of the quiver model in Lemma \ref{lem:degofcanquivermodel}. The statement follows.
\end{proof}

Now we are ready to prove the main theorem of this section: the general fibers of the Abel--Jacobi map are Fano varieties with prime canonical class. Classically, for moduli of vector bundles on curves, this can be deduced from the fact that every fiber is unirational of Picard rank one. In the $\Ku(Y_3)$ case, it is still unknown to us if the general fiber of $\IJ_v$ is unirational or of Picard rank one. Also note that moduli of vector bundles with a fixed determinant on curves is of index $2$, while for cubic threefolds, we will show that the canonical class is indeed primitive.

\begin{Thm}\label{thm:fanofiber}
    Let $Y_3$ be smooth cubic threefold and $\Ku(Y_3)$ be its Kuznetsov component. Then for every primitive $v\in\Kn(\Ku(Y_3))$ with $\chi(v,v)< -4$ and general $c_v\in J_{v}(Y_3)$, the space $\ms(v,c_v)$ is a smooth Fano variety with primitive canonical class.
\end{Thm}
\begin{proof}
Let $c_v\in J_v(Y_3)$ be general such that $\ms(v,c_v)$ is a smooth fiber of $\IJ_v$ of expected dimension. Then the line bundle $\omega_{\ms(v,c_v)}$ is isomorphic to the restriction of $\omega_{\ms(v)}$. Denote by $\iota\colon \ms(v,c_v)\hookrightarrow\ms(v)$ the inclusion map, by Corollary \ref{cor:relpic1}, the image of $\iota^*(\Pic(\ms(v)))$ is rank $1$. In particular, as $\ms(v)$ is projective, the image of the restriction is generated by an ample divisor on $\ms(v,c_v)$. To show that $\ms(v,c_v)$ is Fano, we only need to check the intersection number of $\omega_{\ms(v,c_v)}$ and a curve is negative.\\

   By Remark \ref{rem:Sswapabc}, we may assume that $v$ is in the $(\alpha,\beta)$-sextant. Let $v_+$ and $v_-$ be the characters as that introduced in Section \ref{sec22}. Write $v_+$ as $n_+\alpha+m_+\beta$ and $v_-$ as $n_-\alpha+m_-\beta$, then $n_-m_+-n_+m_-=1$. It follows that
  \begin{align*}
      &\chi(v_+,v_-)-\chi(v_-,v_+) \\
      = & (-n_+n_--n_-m_+-m_-m_+)-(-n_+n_--n_+m_--m_-m_+)=-1.
  \end{align*}
As $\chi(v,v)< -4$, $v\neq \alpha+\beta$ or $\alpha$ or $\beta$. It follows that $-n_+n_--n_+m_--m_-m_+\leq -1$. So $\chi(v_+,v_-)\leq -2$.
  
  As $v\neq \alpha$ or $\beta$, both $v_+$ and $v_-$ are in the $(\alpha,\beta)$-sextant, $\phi_\sigma(v_+)-\phi_\sigma(v_-)\in(0,\tfrac{1}{3})$. For all pairs of $(E_1,E_2)\in \ms(v_-)\times \ms(v_+)$, we have $\Hom(E_i,E_j[k])=0$ for all $i,j\in\{1,2\}$ whenever $k\geq 0$ or $1$. It is also clear that $\Hom(E_2,E_1)=0$. So $\hom(E_2,E_1[1])=-\chi(v_+,v_-)\geq 2$. By Theorem \ref{thm:stratcubic3} and Lemma \ref{lem:znamb}, $\BN(v_-,v_+)\neq \ms(v_-)\times \ms(v_+)$. So when $(E_1,E_2)$ is general, $\Hom(E_1,E_2)=0$. 

  By Lemma \ref{lem:extendstabobj}, for every $0\neq f\in \Hom(E_2,E_1[1])$, the object $\Cone(f)[-1]\in \ms(v)$. So every  pair of $(E_1,E_2)\in \ms(v_-)\times \ms(v_+)\setminus \BN(v_-,v_+)$ satisfies all assumptions as that in Proposition \ref{prop:localneighbourofcurve}.  By Proposition \ref{prop:localneighbourofcurve}, we have $\omega|_{\ET(E_1,E_2)}\cong\cO_{\mathbf P^{r_v-1}}(-1)$ where $r_v=-\chi(v_+,v_-)$.
  
  Let  $(E_1,E_2)\in \ms(v_-)\times \ms(v_+)\setminus \BN(v_-,v_+)$  be with $c_2(E_1)+c_2(E_2)=c_v$. In particular, $\ET(E_1,E_2)\subset \ms(v,c_v)$ and $\omega_{\ms(v,c_v)}|_{\ET(E_1,E_2)}\cong \cO_{\mathbf P^{r_v-1}}(-1)$. Choose a line $\ell$ in $\ET(E_1,E_2)$, then $\omega_{\ms(v,c_v)}\cdot \ell=-1$. As $\ms(v,c_v)$ is smooth, the divisor $\omega_{\ms(v,c_v)}$ is primitive. The statement holds.
\end{proof}

\begin{Cor}\label{Cor:mrcquotient}
    Let $Y_3$ be smooth cubic threefold and $\Ku(Y_3)$ be its Kuznetsov component. Then for every primitive $v\in\Kn(\Ku(Y_3))$ with $\chi(v,v)\leq -4$, the maximal rationally connected quotient of $\ms(v)$ is the intermediate Jacobian $J(Y_3)$.
\end{Cor}
\begin{proof}
    This is a direct corollary of Theorem \ref{thm:fanofiber} together with the following facts. Every smooth Fano variety over $\C$ is rationally connected. Rational connectivity in a smooth proper family is both open and closed condition. 
\end{proof}

\section{Applications and further directions}

\subsection{Lagrangian subvarieties}\label{sec:lagrangian}
As another application of Proposition \ref{prop:existencegeneralcase} and the technique in Section \ref{sec2}, we prove \cite[Conjecture A.1]{FGLZ:EPW} in the very general cubic fourfold case. To state the result, we first recap the notion of the Kuznetsov component of a cubic fourfold. More details are referred to the paper \cite{FGLZ:EPW}.

Let $Y_4$ be a cubic fourfold and $j\colon Y_3\to Y_4$ be a smooth hyperplane section. The bounded derived category of coherent sheaves $\Db(Y_4)$ admits a semiorthogonal decomposition:
\begin{align*}
    \Db(Y_4)=\langle \cO_{Y_4}(-H), \Ku(Y_4),\cO_{Y_4},\cO_{Y_4}(H)\rangle,
\end{align*}
where $H$ is the hyperplane class. Denote by $\mathsf{pr}_{\Ku(Y_4)}$ the projection functor from $\Db(Y_4)$ to $\Ku(Y_4)$ given by $\mathbb R_{\cO_{Y_4}(-H)}\circ\mathbb L_{\cO_{Y_4}}\circ\mathbb L_{\cO_{Y_4}(H)}$. Here we write $\mathbb R_E$ (resp. $\mathbb L_E$) for the right (resp. left) mutation functor of the exceptional object $E$. More explicitly, for every object $G\in\Db(Y_4)$,
\begin{align*}
    \mathbb R_E(G)&:=\Cone(G\xrightarrow{\ev}E\otimes \mathrm{RHom}(G,E)^*)[-1]. \\ 
    \mathbb L_E(G)&:=\Cone(E\otimes \mathrm{RHom}(E,G)\xrightarrow{\mathsf{ev}}G).
\end{align*}

For every $E\in \Ku(Y_3)$, by adjunction $\RHom_{\cO_{Y_4}}(\cO_{Y_4},j_*E)=\RHom_{\cO_{Y_3}}(j^*\cO_{Y_4},E)=0$ and $\RHom_{\cO_{Y_4}}(\cO_{Y_4}(H),j_*E)=0$. It follows that \[\mathsf{pr}_{\Ku(Y_4)}(j_*E)=\mathbb R_{\cO_{Y_4}(-H)}(j_*E).\] 

Let $\sigma_3$ be the stability condition on $\Ku(Y_3)$ as in Section \ref{sec:kuz}. By \cite[Theorem 1.2]{BLMS:kuzcomponent}, there is a stability condition $\sigma_4$ on $\Ku(Y_4)$.

\def\pry{\ensuremath{\mathsf{pr}_{\Ku(Y_4)}}}
The following result \cite[Lemma A.2]{FGLZ:EPW} will be useful. 
\begin{Lem}[{\cite[Lemma A.2]{FGLZ:EPW}}]
    Let $E\in \Ku(Y_3)$, then the object $j^*\pry j_* E$ fits into the distinguished triangle:
    \begin{align}\label{eq:C5}
        \mathsf S_{\Ku(Y_3)}^{-1}E[2]\to j^*\pry j_* E \to E\xrightarrow{+}.
    \end{align}
\end{Lem}

We are ready to prove \cite[Conjecture A.1]{FGLZ:EPW} for very general cubic fourfolds.
\begin{Thm}\label{thm:C2}
 Let $Y_4$ be a smooth cubic fourfold with $\rk(\Kn(\Ku(Y_4)))=2$, $j\colon Y_3\to Y_4$ be a smooth hyperplane section. Then for every primitive character $v\in\Kn(\Ku(Y_3))$, there exists a non-empty open subset $U_v\subset M^s_{\sigma_3}(\Ku(Y_3),v)$ such that for every $E\in U_v$, the object $\pry j_*E$ is $\sigma_4$-stable. 
\end{Thm}

\begin{proof}
We prove the statement by induction on $-\chi(v,v)$. By Example \ref{eg511} and a direct computation (see \cite[Theorem B.1]{FGLZ:EPW}), the statement holds for $\pm\alpha$, $\pm\beta$ and $\pm\gamma$.  Assume that the statement holds for all $v'$ with $-\chi(v',v')\leq N-1$ for some $N\geq 2$.
 
 Let $v$ be a primitive character with $-\chi(v,v)=N\geq 2$. By Proposition and Definition \ref{pd:pick} and Lemma \ref{lem:22matrix}, $v=v_++v_-$ for some primitive $v_\pm$  with $-\chi(v_-,v_-),-\chi(v_+,v_+)\leq N-1$ and $\chi(v_+,v_-)<0$. 

 By induction, there exist $E_\pm\in U_{v_\pm}$. As $\chi(v_+,v_-)<0$ and $\gd(\sigma_3)<2$, $\exists \;0\neq f\in \Hom(E_+,E_-[1])$. By Lemma \ref{lem:extendstabobj}, the object $E_f=\Cone(f)[-1]$ is $\sigma_3$-stable with character $v$.

Applying the functors $\mathbb R_{\cO_{Y_4}(-H)}\circ j_*$ to the distinguished triangle
\begin{align*}
    E_-\to E_f\to E_+\xrightarrow{f}E_-[1],
\end{align*}
we get the distinguished triangles:
   \begin{align*}
    \pry j_*E_-\to \pry j_*E_f\to \pry j_*E_+\xrightarrow{\pry j_*f}\pry j_*E_-[1].
\end{align*}

As $Y_4$ is assumed with $\rk(\Kn(\Ku(Y_4)))=2$, the induced map
    \begin{align*}
        \varphi:=\mathsf{pr}_{\Ku(Y_4)}\circ j_*:\Kn(\Ku(Y_3))\to \Kn(\Ku(Y_4))
    \end{align*} is a group isomorphism preserving the orientation with respect to $\sigma_3$ and $\sigma_4$. In particular, for every non-zero $v,w\in \Kn(\Ku(Y_3))$ with no inner lattice point in $\triangle(v,w)$, there is no inner lattice point in $\triangle(\varphi(v),\varphi(w))$.

By induction, the objects $\pry j_*E_\pm$ are both $\sigma_4$-stable. By Lemma \ref{lem:extendstabobj}, to show that $\pry j_*E_f$ is $\sigma_4$-stable, we only need to show
\begin{align}\label{eq:Cvan}
    \Hom(\pry j_*E_+,\pry j_*E_f)=0
\end{align}

As $\pry j_*E_+\in\Ku(Y_4)$, we have $\RHom(\pry j_*E_+,\cO(-H))=0$. Applying the functor $\Hom(\pry j_*E_+,-)$ to the distinguished triangle
\begin{align*}
    \pry j_* E_f\to j_*E_f\xrightarrow{\ev}\cO_{Y_4}(-H)\otimes \RHom(\cO_{Y_4}(-H),j_*E_f)^*\xrightarrow{+},
\end{align*}
it follows that 
\begin{align}\label{eq:C11}
    \Hom(\pry j_*E_+,\pry j_*E_f)\cong\Hom(\pry j_*E_+,j_*E_f).
\end{align}

By adjunction, we have 
\begin{align}\label{eq:C12}
    \Hom(\pry j_*E_+,j_*E_f)\cong \Hom(j^*\pry j_*E_+,E_f).
\end{align}

By Proposition \ref{prop:serreinv}, $\phi_{\sigma_3}( \mathsf{S}_{\Ku(Y_3)}^{-1}E_+[2])>\phi_{\sigma_3}(E_+)>\phi_{\sigma_3}(E_f)$, we have \[\Hom(\mathsf{S}_{\Ku(Y_3)}^{-1}E_+[2],E_f)=\Hom(E_+,E_f)=0.\] Applying \eqref{eq:C5} to $E_+$ and  combining the two vanishings above, we get 
\begin{align}\label{eq:C13}
    \Hom(j^*\pry j_*E_+,E_f)=0.
\end{align}

Combining \eqref{eq:C11}, \eqref{eq:C12}, and \eqref{eq:C13}, we have the vanishing as that in \eqref{eq:Cvan}. The object $\pry j_*E_f$ is $\sigma_4$-stable. 

By the openness property of the stability conditions as in \cite{liurendapaper}, there is a non-empty open subset $U_v$ satisfying for every $E\in U_v$, $\pry j_*E$ is $\sigma_4$-stable. The statement holds by induction.
\end{proof}

Recall that the moduli space of stable objects on $\Ku(Y_4)$ is a smooth projective hyperk\"ahler variety \cite{liurendapaper}. By \cite[Theorem A.4]{FGLZ:EPW}, Theorem~\ref{thm:C2} has the following consequence on the existence of Lagrangian subvarieties of these moduli spaces.

\begin{Thm}[{\cite[Theorem A.4]{FGLZ:EPW}}]\label{thm:lagrangian_existence}
    Continuing with the setup of Theorem \ref{thm:C2}, the functor $\pry j_*$ induces a rational map $r\colon M_{\sigma_3}(\Ku(Y_3),v) \dashrightarrow M_{\sigma_4}(\Ku(Y_4),\pry j_*(v))$, such that the image is Lagrangian in $M_{\sigma_4}(\Ku(Y_4),\pry j_*(v))$. Furthermore, $r|_{U_v}$ is an immersion and $r$ is a birational map.
\end{Thm}

\begin{Rem}
    Note that in the case of Gushel--Mukai varieties, a much stronger result is proved in \cite{FGLZ:EPW}. In particular, they showed that for a very general GM fourfold and a smooth hyperplane section, an object in the Kuznetsov component of the GM threefold is semistable \textit{if and only if} its image in the Kuznetsov component of the GM fourfold is semistable. It is not clear to us whether the analogue of this stronger statement can hold in our case of cubic hypersurfaces.
\end{Rem}

\subsection{Hilbert schemes of curves on cubic threefolds}\label{sec:hilb_curves}

Curves with small degree and genus on the cubic threefold have been studied in \cite{HarrisRothStarr:curvesoncubic3fold}.  Denote by $\Hilb^{d,g}(Y_3)$ the open subscheme parametrizing irreducible smooth curves with degree $d$ and genus $g$ on $Y_3$ in the Hilbert scheme $\Hilb_{dt+1-g}(Y_3)$. It is shown explicitly case by case in \cite{HarrisRothStarr:curvesoncubic3fold} that $\Hilb^{d,g}(Y_3)$ is irreducible when $d\leq 5$, and smooth when $(d,g)\in\{(1,0),(2,0),(3,0),(3,1),(4,1)\}$. As an application of our results, we can describe the birational model of the ``main'' irreducible component of $\Hilb^{d,g}(Y_3)$ for some values of $d$ and $g$.

We start with a birational equivalence result, which strengthens Theorem~\ref{thm:connectedfiber} in special cases.

\begin{Prop}\label{prop:birfiber}
     Let $Y_3$ be a smooth cubic threefold and $\Ku(Y_3)$ be its Kuznetsov component. Let $v\in\Kn(\Ku(Y_3))$ be in either one of the following forms 
     \[2n\beta+\alpha; 2n\alpha+\beta; 2n\alpha+(2n-1)\beta\]
   for some $n\geq 1$. Then for general $c_v,c_v'\in J_{v}(Y_3)$, the spaces $\ms(v,c_v)$ and $\ms(v,c'_v)$ are  birational. 
   For every general  $c_1\in J_{2\beta+\alpha}(Y_3)$ and $c_2\in J_{2\alpha+\beta}(Y_3)$, the spaces $\ms(2\beta+\alpha,c_1)$ and $\ms(2\alpha+\beta,c_2)$ are birational. 
\end{Prop}
\begin{proof}
    By Proposition \ref{prop:stabbir1} and Remark \ref{rem:compatibleextandajmap}, for every $E\in \ms(\gamma)$ and general $c\in J_{2n\beta+\alpha}(Y_3)$, we have the birational map:
    \[e^R_{E[-1],2n\beta+\alpha}\colon\Pe{E[-1]}c\dashrightarrow \Pe{c+c_2(E)}E,\]
    where $c+c_2(E)\in J_{(2n-1)\beta+2\alpha}(Y_3)$. As $\chi(2n\beta+\alpha,\gamma)=1$, the space $\Pe{E[-1]}c$ is birational to $\ms(2n\beta+\alpha,c)$. As $(2n-1)\beta+2\alpha$ is primitive and $-\chi(\gamma,(2n-1)\beta+2\alpha)=2n-1$, the space $\Pe{c'}E$ is birational to $\ms((2n-1)\beta+2\alpha)\times \mathbf P^{2n-2}$. Therefore $\ms(2n\beta+\alpha,c)$ is birational to $\ms(2n\beta+\alpha,c+c_2(E)-c_2(E'))$ for all general $E,E'\in \ms(\gamma)$. Note that $J_0(Y_3)$ can be generated by $F(Y_3)$ as an abelian group, the statement holds for $v$ of the form $2n\beta+\alpha$. Note that when $n=1$, $2n-2=0$. The statement for the $2\alpha+\beta$ and $2\beta+\alpha$ part holds.

    For the $2n\alpha+\beta$ case, by Proposition \ref{prop:stabbir1}, for every $E\in \ms(\gamma)$, we have the birational map $e^L_{\beta+2n\alpha,E}$. As $\chi(\gamma,\beta+2n\alpha)=1$ and $2\beta+(2n-1)\alpha$ is primitive, by the same argument as that for $2n\beta+\alpha$, the statement holds.  For the $2n\alpha+(2n-1)\beta$ case, by Proposition \ref{prop:stabbir3}, for general $E\in \ms(\beta+\gamma)$, we have the birational map $e^{R}_{E[-1],2n\alpha+(2n-1)\beta}$. As $\chi(2n\alpha+(2n-1)\beta,\beta+\gamma)=2n-(2n-1)=1$ and $(2n+1)\alpha+(2n-3)\beta$ is primitive, by the same argument as that for $2n\beta+\alpha$, the statement holds.
\end{proof}

Now we have the following examples.

\begin{Ex}\label{eg:ajfiberhighdimension}
For $m\geq 0$, let $d=\tfrac{3}{2}m^2+\tfrac{3}{2}m+1$ and $g=m^3-m$. Denote by $C_{d,g}$  a smooth irreducible curve with degree $d$ and genus $g$ in $Y_3$. Note that $\cI_{C_{d,g}}(mH)\in\Ku(Y_3)$ if and only if $C_{d,g}$ is not on any hypersurface in $|mH|$ on $Y_3$. 

By a direct computation, the character of $\cI_{C_{d,g}}(mH)$ is $(m+1)\beta+m\gamma$. Note that the dimension of $\ms((m+1)\beta+m\gamma)$ is $3m^2+3m+2=2d$, which is the same as the expected dimension of $\Hilb^{d,g}(Y_3)$.

\begin{Ques}\label{ques:curve}
    Does there exist a smooth irreducible curve $C_{d,g}$ such that $\cI_{C_{d,g}}(mH)$ is $\sigma$-stable?
\end{Ques}
A positive answer to this question will show that a general point in the space $\ms((m+1)\beta+m\gamma)$ is the ideal sheaf of a curve with degree $d$ and genus $g$. In particular, $\ms((m+1)\beta+m\gamma)$ is birational to the irreducible component of $\Hilb^{d,g}(Y_3)$ that contains $C_{d,g}$. In this case, by Corollary \ref{Cor:mrcquotient}, when $m\geq 1$, the mrc quotient of $\Hilb^{\tfrac{3}{2}m^2+\tfrac{3}{2}m+1,m^3-m}(Y_3)$ is $J(Y_3)$.
\end{Ex}
\begin{Ex}[Rational quartic curves and $\ms(2\beta+\alpha)$] \label{ex:quartics}
Continuing from the last example, for every non-degenerate smooth rational quartic $C_{4,0}$ on $Y_3$, the sheaf $\cI_{C_{4,0}}(H)$ is slope stable by definition. By a standard wall-crossing argument, one can show that $\cI_{C_{4,0}}(H)$ is $\sigma_{\alpha,-1}$-stable for every $\alpha>0$ (see Appendix \ref{app:small_moduli} for more details on the notion). It follows that $\cI_{C_{4,0}}(H)$ is $\sigma$-stable. We get the birational map $\ms(2\beta+\gamma)\dashrightarrow \Hilb^{4,0}(Y_3)$.

On the other hand, by \cite{Kuz:V14} or \cite[Theorem 4.7]{Kuz:Fano3fold}, there is a family of smooth Fano threefolds $\{X_t\}$ parametrized by ${t\in M^s_\sigma(2\beta)}$, with index $1$ and genus $8$. For every $t$, we have $X_t$ is birational to $Y_3$, and the category $\Ku(Y_3)\cong \Ku(X_t)$. By \cite[Theorem 5.9]{JLZ} the object $\mathsf{pr}_{\Ku(X_t)}(\cI_p)=\Cone(\cE^{\oplus 4}_{2,X_t}\xrightarrow{\ev}\cI_p)$ is $\sigma$-stable for every $p\in X_t$. By direct computation, we have \[\dim \ms([\mathsf{pr}_{\Ku(X_t)}(\cI_p)])=1-\chi(\mathsf{pr}_{\Ku(X_t)}(\cI_p),\mathsf{pr}_{\Ku(X_t)}(\cI_p))=8.\] 

So up to taking the Serre functor, the character $[\mathsf{pr}_{\Ku(X_t)}(\cI_p)]$ is $2\beta+\gamma$ or $\beta+2\gamma$. In either case, we get a morphism $\mathsf{pr}_t\colon X_t\to \ms(v)\colon p\mapsto \mathsf{pr}_{\Ku(X_t)}(\cI_p)$. The morphism is clearly set-theoretically injective. As $c_2(\cI_p)$ is trivial for every $p\in X_t$, the map $\mathsf{pr}_t$ maps each $X_t$ to a fiber of the \AJ map. Now by Theorem \ref{thm:fanofiber} and Proposition \ref{prop:birfiber}, every general fiber $\ms(2\beta+\gamma,c_t)$ is isomorphic to an $X_t$.

In summary, the general fiber of the \AJ map from $\Hilb^{4,0}(Y_3)$ to $J(Y_3)$ is birational to $Y_3$. This provides a categorical interpretation for the result in \cite{IlievMarku:cubics} and    \cite[Theorem 8.2]{harris2002abeljacobi}.
\end{Ex}

\begin{Rem}\label{rem:curvetoku}
When $(d,g)$ is not of the form $(\tfrac{3}{2}m^2+\tfrac{3}{2}m+1,m^3-m)$ as that in the last example, in some cases, one can still choose $m$ and consider $\mathsf{pr}_{\Ku(Y_3)}(\cI_{C_{d,g}}(mH))$. Set $v(d,g,m)=[\mathsf{pr}_{\Ku(Y_3)}(\cI_{C_{d,g}}(mH))]$. In general, it is a tricky question to prove the stability of the object $\mathsf{pr}_{\Ku(Y_3)}(\cI_{C_{d,g}}(mH))$ for general curve $C_{d,g}$ on the principal irreducible component of $\Hilb^{d,g}(Y_3)$. For cases with small $d$, this has been proved, see \cite{BMMS:Cubics}. Whenever this holds for $(d,g,m)$, we get the following rational map which is compatible with the \AJ map:
\begin{align*}
    \mathsf{pr}_{\Ku(Y_3)}\colon\Hilb^{d,g}(Y_3) & \dashrightarrow \ms(v(d,g,m))\\
    \cI_{C_{d,g}}(mH)& \mapsto \mathsf{pr}_{\Ku(Y_3)}(\cI_{C_{d,g}}(mH)).
\end{align*}
In the case when $m$ satisfies $2d\geq \dim(\ms(v(d,g,m))$, the map $\mathsf{pr}_{\Ku(Y_3)}$ is dominant with rationally connected fiber. Here is the table for the choice of $m$ and the characters for small degrees:
\begin{center}
\begin{tabular}{||c c c c||} 
 \hline
 $d$ & $g$ & $m$ & $v(d,g,m)$ \\ [0.5ex] 
 \hline\hline
 1 & 0 & 0 & $\beta$ \\ 
 \hline
 2 & 0 & 1 & $\gamma$ \\
 \hline
 2 & 0 & 2 & $-\beta$ \\
 \hline
 3 & 0 & 1 & $\beta+\gamma$ \\
 \hline
 3 & 1 & 2 & 0 \\ 
 \hline
\end{tabular}
\begin{tabular}{||c c c c||} 
 \hline
 $d$ & $g$ & $m$ & $v(d,g,m)$ \\ [0.5ex] 
 \hline\hline
 4 & 0 & 1 & $2\beta+\gamma$ \\ 
 \hline
 4 & 0 & 2 & $-2\alpha-\beta$ \\
 \hline
 4 & 0 & 3 & $\alpha-2\gamma$ \\
 \hline
 4 & 1 & 2 & $-\alpha$ \\
 \hline
 5 & 0 & 3 & $-2\beta-\gamma$ \\ 
 \hline
\end{tabular}
\begin{tabular}{||c c c c||} 
 \hline
 $d$ & $g$ & $m$ & $v(d,g,m)$ \\ [0.5ex] 
 \hline\hline
 5 & 1 & 2 & $-2\alpha$ \\ 
 \hline
 5 & 1 & 3 & $-2\gamma$ \\
 \hline
 5 & 2 & 2 & $-\gamma$ \\
 \hline
 6 & 1 & 2 & $-3\alpha$ \\
 \hline
 6 & 1 & 3 & $-3\beta$ \\
 \hline
 7 & 2 & 3 & $-\alpha-3\beta$ \\ 
 \hline
\end{tabular}
\end{center}
\noindent However, it is not always the case that one can choose $m$ satisfying $2d\geq \dim(\ms(v(d,g,m))$. For example, for the missing cases in the table, when  $(d,g)=(6,0)$, $(7,0)$, or $(7,1)$, there seems no obvious way to get a dominant rational map from $\Hilb^{d,g}(Y_3)$ to $\ms(v)$.
\end{Rem}

\subsection{Further questions} \label{sec:questions}

For the cubic threefold case, while we have given several descriptions for the moduli spaces $\ms(v)$ and the general fibers $\ms(v,c)$ and their relations, some important information about them is still missing. In comparison to what we know for moduli of vector bundles on curves, we ask the following questions:

\begin{Ques}[Questions in the $Y_3$ case]\label{rem:qqy}
Let $\dim\ms(v)>5$ and $\ms(v,c)$ be a general fiber.
\begin{enumerate}[(1)]
    \item Is $\ms(v,c)$ with Picard rank one? Is there a Verlinde type formula for $\ms(v,c)$?
    \item What are the cohomology groups of $\ms(v,c)$?
    \item Is there a semiorthogonal decomposition for $\Db(\ms(v,c))$? 
    \item Is $\ms(v,c)$ always birational to $\ms(v,c')$ when both of them are of the expected dimension $-\chi(v,v)-4$?
    \item When $\dim\ms(v)$ is large, does the \AJ map admit some nice property, for example flatness?
    \item Is $\ms(v,c)$ unirational? Is it stably birational to $Y_3$?
\end{enumerate}      
\end{Ques}

The next interesting case is quartic double solids (index $2$ degree $2$) and Gushel--Mukai threefolds (index $1$ degree $10$), see \cite{APR}. In both cases, the Kuznetsov component is an Enriques category, i.e. the Serre functor is $\iota[2]$, with $\iota$ being an involution. The non-emptiness of the moduli spaces has been proved in \cite{PPZ_Enriques}, and is also covered by Theorem~\ref{thm:existinKurank2}. In this case, we expect our method to apply but to give a somewhat different result. Note that following the arguments in Section~\ref{sec:deg_K}, it suggests that the fiber of the \AJ map should have numerically trivial canonical class. We formulate this as the following question.

\begin{Ques}[Moduli spaces on $\Ku(Y_2)$]
For quartic double solids and Gushel--Mukai threefolds, what is the Kodaira dimension of a general fiber of the \AJ map?
\end{Ques}

From a broader point of view, we wonder if there are more examples of $\C$-linear triangulated categories, whose moduli spaces of stable objects behave like moduli of vector bundles on curves.
\begin{Ques}\label{rem:nccurve}
Search for more examples of triangulated categories $\cT$ satisfying the following conditions:
\begin{enumerate}
    \item $\cT$ is a full admissible subcategory of $\Db(X)$ for a smooth projective variety $X$;
    \item The Hodge structure on the topological K-theory $\text{K}_0^{\text{top}}
(\cT)$ (in the sense of \cite{Perry_2022}) is of pure type $(0,0)$;
    \item The Euler paring on $\Kn(\cT)$ is of signature $(r-2,2)$;
    \item $\cT$ admits a Serre invariant stability condition.
\end{enumerate}   
\end{Ques}
Note that here $\Kn(\cT)$ is not required to be of rank 2; as computed in Lemma \ref{lem:22matrix}, possible Euler forms of rank 2 are very restricted.

\appendix
\section{Euler form with numerical rank two}\label{app:bounds}

The following lemma computes all possible Euler forms for a non-commutative smooth projective variety with $\rk(\kn)=2$ and no non-zero numerical character fixed by the Serre functor. This was used in the proof of Theorem~\ref{thm:existinKurank2}.
\begin{Lem}\label{lem:22matrix}
    Let $Q(-,-)$ be a non-degenerate bilinear form on $\Lambda\cong\Z^{\oplus 2}$ satisfying \begin{enumerate}
        \item $Q(\mathbf x,\bx)\leq 0$;
        \item $\exists  \, D\in\mathrm{SL}(\Lambda)$ such that $\forall\, \bx,\mathbf y\in \Lambda$, $Q(\bx,\mathbf y)=Q(\mathbf y,D\bx)$ and $\forall \, 0\neq \bx\in\Lambda$, $D\bx\neq \pm \bx$.
    \end{enumerate}
    Then $D$ is of order $3$, $4$, or $6$. There exists a right-hand oriented $\Z$-linear basis of $\Lambda$ under which the matrix of $Q$ is in one of the following forms:
    \begin{align*}
       I_\pm= \begin{pmatrix}
            -n & \pm n\\0& -n
        \end{pmatrix};\;\;\; J_\pm=\begin{pmatrix}
            -n & \pm n\\ \mp n& -n
        \end{pmatrix};\;\;\; 
        K_\pm=\begin{pmatrix}
            -n & \pm n\\ \mp 2n & -n
        \end{pmatrix}
    \end{align*}
    for some $n\geq 1$.

    When $D$ is of order $3$ (resp. $4$ and $6$), the matrix of  $Q$ is of the form $K_\pm$ (resp. $J_\pm$ and $I_\pm$).
    
    When $Q$ is of the form $I_+$,    for every $\bv\in \Lambda$ with $Q(\bv,\bv)<-n$, $Q(\bv_+,\bv_-)<0$.

    When $Q$ is of the form $I_-$, for every $\bv\in \Lambda$ with $Q(\bv,\bv)<-3n$, $Q(\bv_+,\bv_-)<0$.
\end{Lem}
\begin{proof}
    Let $\bx\in \Lambda$ be with the maximum $Q(\bx,\bx)$ among all non-zero vectors in $\Lambda$. In particular, $\bx$ is primitive. By assumption ({\em a}), we may denote by $Q(\bx,\bx)=-n$ for some $n\geq 0$. As $S\bx\neq \pm \bx$, the vectors $\bx$ and $S\bx$ form a $\Z$-linear basis of $\Lambda$.
    
    It follows by assumption ({\em b}) that $Q(\bx,D\bx)=Q(\bx,\bx)=-n$. The matrix of $Q$ is of the form $\begin{pmatrix}
        -n & -n\\ t& -n
    \end{pmatrix}$ for some $t\in \Z$. As $Q$ is non-degenerate, the value $n\geq 1$.
    
   For a given non-zero vector $\bv=a\bx+bS\bx$, $Q(\bv,\bv)=-n(a^2+ab+b^2)+tab$.  As $Q(\bx\pm D\bx,\bx\pm D\bx)\leq -n$, the value $t\in[0,2n]$. 

    When $t=0$, the set  $\{\bx,D\bx\}$ (or $\{\bx,-D\bx\}$) forms a right-hand oriented basis of $\Lambda$ under which the matrix of $Q$ is of the form of $I_-$ (resp. $I_+$). By assumption ({\em b}), the transformation $D$ is of the form $\begin{pmatrix}
        0 & -1 \\ 1 & 1
    \end{pmatrix}$ (resp. $\begin{pmatrix}
        0 & 1 \\ -1 & 1
    \end{pmatrix}$), which is of order $6$.

      When $t=2n$, the set  $\{\bx,D\bx\}$ (or $\{\bx,-D\bx\}$) forms a right-hand oriented basis of $\Lambda$ under which the matrix of $Q$ is of the form of $K_-$ (resp. $K_+$). By assumption ({\em b}), the transformation $D$ is of the form $\begin{pmatrix}
        0 & -1 \\ 1 & -1
    \end{pmatrix}$ (resp. $\begin{pmatrix}
        0 & 1 \\ -1 & -1
    \end{pmatrix}$), which is of order $3$.

    When $t=2n$, $\{\bx,S\bx\}$ or $\{\bx,-S\bx\}$ is a right-hand oriented basis of $\Lambda$ under which the matrix of $Q$ is of the form of $K_\pm$.

     When $t\neq 0$ or $2n$, $Q(\bv,\bv)=-n$ when and only when $\bv=\pm \bx$ or $D\bx$. Note that we have $Q(D^2\bx,D^2\bx)=Q(D\bx,D^2\bx)=Q(D\bx,D\bx)=-n$. So $D^2\bx=-\bx$. It follows that $D$ is of order $4$, and $t=Q(D\bx,\bx)=Q(D\bx,-D^2\bx)=n$. So $\{\bx, D\bx\}$ (or $\{\bx,-D\bx\}$) is a right-hand oriented basis of $\Lambda$ under which the matrix of $Q$ is of the form of $J_-$ (resp. $J_+$).\\

    When $Q$ is of the form $I_\pm$, we compute all $\bv$ for which  $Q(\bv_+,\bv_-)\geq 0$. Note that we may always assume $n=1$. We write $\bv_+=(a,b)$ and $\bv_-=(c,d)$ as vectors under the basis for $I_\pm$. By assumption, $1=\bv_-\times \bv_+=bc-ad$ and $0\leq \bv_-\cdot \bv_+=ac+bd$. We may also assume that $b,d\geq 0$.

    When $Q$ is of the form of $I_-$, we have $Q(\bv_+,\bv_-)=-ac- ad-bd$. In the case of $d=0$, it follows that $b=c=1$ and $a\geq 0$. So $Q(\bv_+,\bv_-)\geq 0$ only when $a=0$. This is when $\bv=(1,1)$ with $Q(\bv,\bv)=-3$. In the case of $b=0$, it follows that $d=-a=1$ and $c\leq 0$. So $Q(\bv_+,\bv_-)\geq 0$ only when $c=0$ or $-1$. This is when $\bv=(-1,1)$ (or $(-2,1)$) with $Q(\bv,\bv)=-1$ (resp. $Q(\bv,\bv)=-3$). In the case of $b,d> 0$ and $a\geq 0$, it follows that $c>0$. So $Q(\bv_+,\bv_-)< 0$. In the case of $b,d> 0$ and $a< 0$, it follows that $c\leq 0$. If $|b|>|a|$ or $|c|\geq |d|$, then $Q(\bv_+,\bv_-)< 0$. Otherwise, as $bc-ad=1$, it follows that $b=d=1$, $c=0$, and $a=-1$. This is the case when $\bv=(-1,2)$ with $Q(\bv,\bv)=-3$. 

     When $Q$ is of the form of $I_+$, $Q(\bv_+,\bv_-)=-ac+ ad-bd$. In the case of $d=0$, it follows that $b=c=1$ and $a\geq 0$. So $Q(\bv_+,\bv_-)\geq 0$ only when $a=0$. This is when $\bv=(1,1)$ with $Q(\bv,\bv)=-1$. In the case of $b=0$, it follows that $d=-a=1$ and $c\leq 0$. So $Q(\bv_+,\bv_-)\leq ad=-1$. In the case of $b,d> 0$ and $a< 0$, it follows that $c\leq 0$. So $Q(\bv_+,\bv_-)\leq -bd< 0$. In the case of $b,d> 0$ and $a\geq 0$, it follows that $c>0$. As $bc-ad=1$, we have $|b|>|a|$ or $|c|\geq |d|$, in either case it follows that $Q(\bv_+,\bv_-)< 0$. 
\end{proof}

\section{The existence of stable objects with small characters}\label{app:small_moduli}

In this appendix, we prove the non-emptiness of several moduli spaces of small dimensions for various Fano threefolds. This was used in Section~\ref{sec:nonempty_fano3} in the proof of Theorem~\ref{thm:existinKurank2}.

We briefly recall the construction of the stability condition $\sigma$ on $\Ku(X)$. More details are referred to the original paper \cite{BLMS:kuzcomponent}. One of the key ingredients is the tilting construction as in \cite{Happel-al:tilting}.  
\begin{PropDef}[\cite{Happel-al:tilting}]\label{pd:tiltingheart}
 Let $Z$ be a weak stability function on $\cA$ satisfying the Harder--Narasimhan property. For any $t\in\R$, denote by
 \begin{align*}
     \cA^{>t}_{\mu_Z} & :=\{E \;|\; \text{All HN factors $A$ of $E$ have slope $\mu_Z(F)>t$}\}\subset \cA; \\
     \cA^{\leq t}_{\mu_Z} & :=\{E \;|\; \text{All HN factors $A$ of $E$ have slope $\mu_Z(F)\leq t$}\}\subset \cA.
 \end{align*}
 The pair $(\cA^{>t}_{\mu_Z},\cA^{\leq t}_{\mu_Z})$ forms a torsion pair in $\cA$. In particular, there is a heart of a bounded $t$-structure defined by
 \begin{align}
     \cA^t_{\mu_Z}=\langle \cA^{>t}_{\mu_Z},\cA^{\leq t}_{\mu_Z}[1]\rangle.
 \end{align}
\end{PropDef}

 Let $X$ be a smooth Fano threefold with $\Pic(X)=\Z H$. One first considers the standard heart $\Coh(X)$ and the stability function given by $Z=i\rk-H\ch_1$. In particular, the slope of a non-zero coherent sheaf $E$ is $\mu_H(E)=H\ch_1(E)/\rk(E)$ when $\rk(E)\neq 0$ and $+\infty$ when $\rk(E)=0$.

 For every $\beta\in \R$, one may define the first tilting heart as \begin{align*}
     \Coh^\beta(X):=\langle \Coh^{> \beta}_{\mu_H}(X),\Coh^{\leq \beta}_{\mu_H}(X)[1]\rangle. 
 \end{align*}
 Then for every parameter $\alpha>0$, one may further define a function on $\Kn(\Coh^\beta(X))$  given as 
 \begin{align*}
     Z_{\alpha,\beta}:= iH^2\ch^\beta_1+\tfrac{1}{2}\alpha^2H^3\rk-H\ch^\beta_2.
 \end{align*}
 Here we use the twisted Chern characters to simplify the notion: $\ch^\beta_1(E):=\ch_1(E)-\beta H\rk(E)$ and $\ch^\beta_2(E):=\ch_2(E)-\beta H\ch_1(E)+\tfrac{1}{2}\beta^2 H^2\rk(E)$. 

 By general theorem as in \cite{BMT:3folds-BG}, the function $Z_{\alpha,\beta}$ is a weak stability function satisfying Harder--Narasimhan property. Denote by $\sigma_{\alpha,\beta}=(Z_{\alpha,\beta},\Coh^\beta(X))$ the weak stability condition and $\mu_{\alpha,\beta}$ the slope of $Z_{\alpha,\beta}$.

We may further consider the tilting at $0$ with respect to $Z_{\alpha,\beta}$ on $\Coh^\beta$. More precisely, the heart of the bounded $t$-structure is given as
\begin{align*}
    \Coh^0_{\alpha,\beta}(X):=\langle \cA_{\mu_{\alpha,\beta}}^{>0},\cA_{\mu_{\alpha,\beta}}^{\leq0}[1]\rangle
\end{align*}
  The weak stability function is $Z^0_{\alpha,\beta}:=\tfrac{1}{i}Z_{\alpha,\beta}$. Intuitively, the data $\sigma^0_{\alpha,\beta}:=(\Coh^0_{\alpha,\beta}(X),Z^0_{\alpha,\beta})$ is just to make a ``homological shift'' on the weak stability condition $(\Coh^\beta(X),Z_{\alpha,\beta})$ by $\tfrac{1}{2}$. In general, such a non-integer shift on a weak stability condition will not produce a weak stability condition -- it only makes sense for a stability condition. Fortunately, however, by \cite[Proposition 2.15]{BLMS:kuzcomponent}, $\sigma^0_{\alpha,\beta}$ is a weak stability condition. 

The final step is to restrict some particular $\sigma^0_{\alpha,\beta}$ to the Kuznetsov component. More precisely, when $X$ is of index two, let $\beta=-\tfrac{1}{2}$, $\alpha\in(0,\tfrac{1}{2})$, 
 \begin{align}\label{eq:stabonky}
     \cA_\alpha:= \Coh^0_{\alpha,-\frac{1}{2}}(X)\cap \Ku(X)\text{, and } Z_\alpha:= Z^0_{\alpha,-\frac{1}{2}}=H^2\ch^{-\frac{1}{2}}_1+i(H\ch_2^{-\frac{1}{2}}-\tfrac{1}{2}\alpha^2 H^3\rk).
 \end{align}
When $X$ is of index one, let $\beta=\epsilon-1$ for some $0<\epsilon\ll1$ and $\alpha\in(0,\delta)$, where $\mu_{\delta,\epsilon-1}(\cE_X)=\mu_{\delta,\epsilon-1}(\cO_X(-H)[1])$ and 
\begin{align}\label{eq:stabonky1}
     \cA_\alpha:= \Coh^0_{\alpha,\epsilon-1}(X)\cap \Ku(X)\text{, and } Z_\alpha:= Z^0_{\alpha,\epsilon-1}.
 \end{align}
Then \cite[Theorem 1.1]{BLMS:kuzcomponent} states that $\sigma_\alpha=(\cA_\alpha,Z_\alpha)  $ is a stability condition on $\Ku(X)$. 

\begin{Rem}\label{rem:stabonku}
Here are some facts about $\sigma_\alpha$ that are worth explaining in detail.
\begin{enumerate}[(i)]
    \item  The numerical Grothendieck group $\Kn(\Ku(X))$ is with rank $2$. The central charge factors via $\lambda\colon\Kn(\Ku(X))\to \Lambda: v\mapsto(\rk(v),H^2\ch_1(v))$ as that in Assumption \ref{asp2}.
    
 \item When $X$ is of index two (resp. one), for different $\alpha,\alpha'\in(0,\tfrac{1}{2})$ (resp. different $\epsilon\ll 1$ and different $\alpha,\alpha'\in(0,\delta)$), there exists $\tilde g\in\glt$ such that $\sigma_\alpha= \sigma_{\alpha'}\cdot\tilde g$,  see \cite[Proposition 3.6]{PY}. In particular, the parameter $\alpha$ does not affect the stability of objects in $\Ku(\Kn(X))$. We will simply denote $\sigma$ for one fixed stability condition among all $\sigma_\alpha$'s.
\end{enumerate}   
\end{Rem}

\begin{Lem}
\label{lemma_lastcaseGM3}
Let $X$ be a Fano threefold with Picard rank $1$, index $1$ and genus $6$. Let $v \in \kn(\Ku(X))$ be an element with Chern character $3-2H+\frac{3}{10}H^2+\frac{7}{6}P$, where $P$ stands for the class of a point. Then $M^s_\sigma(v) \neq \emptyset$.   
\end{Lem}
\begin{proof}
We work with the alternative Kuznetsov component $\cA_X$ defined by
\[\Db(X)= \langle \cA_X, \cO_X, \cU_X^\vee \rangle.\]
Note that $\Ku(X)$ and $\cA_X$ are equivalent through $\mathbb{L}_{\cO_X}(- \otimes \cO_X(H)) \colon \Ku(X) \simeq \cA_X$. Moreover, this equivalence preserves the orbit of Serre invariant stability conditions by \cite{PR}, and numerically sends $v$ to $v' \in \kn(\cA_X)$ with $\ch(v')=-1+H-\frac{1}{5}H^2-\frac{5}{6}P$.

Let $\sigma$ be a Serre invariant stability condition on $\cA_X$. Consider $F \in \ms(w)$, where $\ch(w)=3-H-\frac{1}{5}H^2+\frac{5}{6}P$. This object exists since it corresponds to a stable object in $\Ku(X)$ with respect to a Serre invariant stability condition with Chern character $1 - \frac{3}{10}H^2+ \frac{1}{2}P$.

Its derived dual $F^\vee$ belongs to $\cA_X^\vee$, which sits in the semiorthogonal decomposition
\[\Db(X)= \langle \cA_X^\vee, \cU_X^\vee, \cO_X(H) \rangle.\]
Since $\Hom^\bullet(F^\vee,F^\vee)=\Hom^\bullet(F,F)$ and the fact that $F$ is stable, by \cite[Corollary 4.15]{Zhang} we have that $F^\vee$ is stable with respect to every Serre invariant stability condition on $\cA_X^\vee$.

Now $\mathbb{L}_{\cO_X}(F^\vee)$ belongs to $\cA_X$, has Chern character $v'$, and is stable with respect to Serre invariant stability conditions by \cite[Section 3.3]{PR}. This implies the non-emptiness in the statement.
\end{proof}

\begin{Lem}
\label{lemma_lastcasedVc}
Let $X$ be a Fano threefold with Picard rank $1$, index $2$ and degree $1$. Let $v \in \kn(\Ku(X))$ be an element with Chern character $1+H-\frac{3}{2}H^2-\frac{5}{6}P$. Then $M^s_\sigma(v) \neq \emptyset$.
\end{Lem}
\begin{proof}
Recall that $X$ is a sextic hypersurface in the weighted projective space $\mathbb{P}(1,1,1,2,3)$. Denote by $x_0, x_1, x_2, x_3, x_4$ the coordinates on $\mathbb{P}(1,1,1,2,3)$. The space $H^0(X, \cO_X(H))$ has dimension $3$ generated by the sections $x_0, x_1, x_2$, and $\cO_X(H)$ has base locus consisting of a point $y$ \cite[Proposition 3.1]{Isk}. In particular, $\cO_X(H)$ induces a rational map from $X$ to $\P^2$ which is defined away from $y$.

Let $C' \subset \P^2$ be a smooth conic. Then it corresponds to a smooth conic $C$ on $X$ not containing $y$. Let $\cI_C$ be the ideal sheaf of $C$. Then $\cI_C(H)$ has Chern character $1+H-\frac{3}{2}H^2-\frac{5}{6}P$. Moreover, we have $I_C(H) \in \Ku(X)$, since $H^0(X, \cO_X(H)) \to H^0(X, \cO_C(H))$ is an isomorphism. In the rest of the proof, we show that $\cI_C(H)$ is $\sigma$-stable for the stability conditions constructed in \cite{BLMS:kuzcomponent}.

First, those stability conditions are in the same $\glt$-orbit by \cite[Proposition 3.6]{PY}. Thus it is enough to show the stability for one of them. We will work with the stability condition $\sigma_\alpha$ for $\alpha<\frac{1}{2}$ (see \cite{BLMS:kuzcomponent}, \cite[Theorem 3.3]{PY}).

Let $\sigma_{\alpha, -\frac{1}{2}}$ be the weak stability condition on $\Db(X)$ obtained by tilting slope stability at $-\frac{1}{2}$. We claim that $\cI_C(H)$ is $\sigma_{\alpha, -\frac{1}{2}}$-stable for $\alpha \gg 0$. Indeed, it is a slope stable torsion-free sheaf with slope $1$ and truncated $-\frac{1}{2}$-twisted Chern character 
\[\ch(\cI_C(H))_{\leq 2}=(1, \frac{3}{2}H, -\frac{7}{8}H^2).\]
The claim follows from \cite[Lemma 2.7]{BMS:stabCY3s}.

Let us now compute the numerical walls for $\cI_C(H)$ with respect to $\sigma_{\alpha,-\frac{1}{2}}$ varying $\alpha$. Note that we can restrict to consider potential walls at $\alpha \geq \frac{1}{2}$, because for our purpose it is enough to show the stability of $\cI_C(H)$ with respect to some weak stability condition $\sigma_{\alpha,-\frac{1}{2}}$ with $\alpha < \frac{1}{2}$. By standard arguments of wall crossing, we get the following numerical walls and destabilizing classes for $\alpha \geq \frac{1}{2}$:
\begin{enumerate}
\item $\alpha=\frac{\sqrt{5}}{2}$, $\ch(v_1)_{\leq 2}=(1,0,0)$;
\item $\alpha=\sqrt{\frac{7}{8}}$, $\ch(v_1)_{\leq 2}=(1,0,-\frac{1}{8}H^2)$;
\item $\alpha=\sqrt{\frac{1}{2}}$, $\ch(v_1)_{\leq 2}=(1,0,-\frac{1}{4}H^2)$.
\end{enumerate}
Note that a $\sigma_{\alpha, -\frac{1}{2}}$-semistable object $A$ with $\ch(A)_{\leq 2}=(1,0,0)$ satisfies $A \cong \cO_X$ by Sublemma \ref{sublemma_classifywall}. Since $\Hom(\cI_C(H), \cO_X)=0$, the first wall in (a) should be of the form
\[0 \to \cO_X \to \cI_X(H) \to B \to 0,\]
which is impossible since $\cI_C(H) \in \Ku(X)$.

We claim that the numerical walls in (b) and (c) are not actual walls. Indeed, let $A$ be a $\sigma_{\alpha, -\frac{1}{2}}$-semistable object with $\ch(A)_{\leq 2}=\ch(v_1)_{\leq 2}$ as in (b) or (c). Then the destabilizing sequence at $\alpha$ has to be of the form
\begin{equation} \label{eq_dest}
0 \to A \to \cI_C(H) \to B \to 0.    
\end{equation}
Indeed, the subobject has to be a sheaf and $\cI_C(H)$ is torsion-free. Now consider the associated cohomology sequence 
\[0 \to \cH^{-1}(B) \to A \to \cI_C(H) \to \cH^0(B) \to 0\]
in $\Coh(X)$. Assume that $\cH^{-1}(B) \neq 0$. Since $A$ has rank $1$, the rank $r$ of $\cH^{-1}(B)$ can be either $0$ or $1$. As $\cI_C(H)$ is torsion-free, we have $r=0$, in contradiction with the fact that its slope semistable factors have slope $\leq -\frac{1}{2}$. Thus $\cH^{-1}(B)=0$, so $B$ is a sheaf. Thus \eqref{eq_dest} is a short exact sequence of sheaves. A standard computation shows that this violates the Gieseker stability of $\cI_C(H)$. We conclude that $\cI_C(H)$ remains $\sigma_{\alpha, -\frac{1}{2}}$-stable for some $\alpha < \frac{1}{2}$. 

Since $\mu_{\alpha, -\frac{1}{2}}(\cI_C(H))<0$, we have that $\cI_C(H)[1] \in \Coh^0_{\alpha, -\frac{1}{2}}(X)$. Note that for every $T \in \Coh(X)$ supported on points, we have
$\Hom(T, \cO_C(H))=\Hom(T, \cO_X(H)[1])=0$, thus $\Hom(T,\cI_C(H)[1])=0$. Then $\cI_C(H)[1]$ is $\sigma^0_{\alpha, -\frac{1}{2}}$-stable \cite[Proposition 4.1]{FP}, which implies $\cI_C(H)[1]$ is $\sigma_\alpha$-stable for some $\alpha < \frac{1}{2}$.
\end{proof}

\begin{SubLem}
\label{sublemma_classifywall}        
Let $A \in \Coh^{-\frac{1}{2}}(X)$ be a $\sigma_{\alpha, -\frac{1}{2}}$-semistable object with  $\ch(A)_{\leq 2}=(1,0,0)$. Then $A \cong \cO_X$.
\end{SubLem}
\begin{proof}
This is a rather standard remark, we include the proof here for the sake of completeness.      

By \cite[Conjecture 4.1]{BMS:stabCY3s}, \cite[Theorem 0.1]{Chunyi:Fano} we have $a:=\ch_3(A) \leq 0$. We also remark that $A$ is $\sigma_{\alpha, -\frac{1}{2}}$-stable, because it has minimal $H^2\ch_1^{-\frac{1}{2}}$. Since $A$ has discriminant $\Delta(A)=0$, by \cite[Corollary 3.11]{BMS:stabCY3s} it is a slope semistable sheaf. It follows that $\Hom(A, \cO_X[k])=0$ for every $k \neq 0, 1, 2$ by Serre duality and stability. Since $\chi(A, \cO_X)=1-a \geq 1$, it follows that $\hom(A, \cO_X)+ \hom(A, \cO_X[2])>0$. Since
\[\mu_{\alpha, 1}(\cO_X(2H))> \mu_{\alpha, 1}(A[1])\]
for small values of $\alpha>0$, we deduce that $\Hom(A, \cO_X[2])=0$ by Serre duality. Thus there exists a morphism $A \to \cO_X$, which is an isomorphism by stability.
\end{proof}

\bibliography{all}                      % .bib-Datei

\newcommand{\etalchar}[1]{$^{#1}$}
\begin{thebibliography}{BMMS12}

\bibitem[Alp13]{Alper:goodmoduli}
Jarod Alper.
\newblock Good moduli spaces for {A}rtin stacks.
\newblock {\em Ann. Inst. Fourier (Grenoble)}, 63(6):2349--2402, 2013.

\bibitem[APR22]{APR}
Matteo Altavilla, Marin Petkovi\'{c}, and Franco Rota.
\newblock Moduli spaces on the {K}uznetsov component of {F}ano threefolds of
  index 2.
\newblock {\em \'{E}pijournal G\'{e}om. Alg\'{e}brique}, 6:Art. 13, 31, 2022.

\bibitem[BBF{\etalchar{+}}24]{Arend:cubic3}
Arend Bayer, Sjoerd~Viktor Beentjes, Soheyla Feyzbakhsh, Georg Hein, Diletta
  Martinelli, Fatemeh Rezaee, and Benjamin Schmidt.
\newblock The desingularization of the theta divisor of a cubic threefold as a
  moduli space.
\newblock {\em Geom. Topol.}, 28(1):127--160, 2024.

\bibitem[Bea95]{beauville:cohomology}
Arnaud Beauville.
\newblock Sur la cohomologie de certains espaces de modules de fibr\'{e}s
  vectoriels.
\newblock In {\em Geometry and analysis ({B}ombay, 1992)}, pages 37--40. Tata
  Inst. Fund. Res., Bombay, 1995.

\bibitem[BKM24]{BKM}
Arend Bayer, Alexander Kuznetsov, and Emanuele Macr{\`{\i}}.
\newblock Mukai bundles on fano threefolds, 2024, arXiv:2402.07154.

\bibitem[BLM{\etalchar{+}}21]{liurendapaper}
Arend Bayer, Mart\'{\i} Lahoz, Emanuele Macr\`{i}, Howard Nuer, Alexander
  Perry, and Paolo Stellari.
\newblock Stability conditions in families.
\newblock {\em Publ. Math. Inst. Hautes \'{E}tudes Sci.}, 133:157--325, 2021.

\bibitem[BLMS23]{BLMS:kuzcomponent}
Arend Bayer, Mart\'{\i} Lahoz, Emanuele Macr\`{i}, and Paolo Stellari.
\newblock Stability conditions on {K}uznetsov components.
\newblock {\em Ann. Sci. \'{E}c. Norm. Sup\'{e}r. (4)}, 56(2):517--570, 2023.
\newblock With an appendix by Bayer, Lahoz, Macr\`{i}, Stellari and X. Zhao.

\bibitem[BM02]{Bridgeland-Maciocia:K3Fibrations}
Tom Bridgeland and Antony Maciocia.
\newblock Fourier-{M}ukai transforms for {$K3$} and elliptic fibrations.
\newblock {\em J. Algebraic Geom.}, 11(4):629--657, 2002, arXiv:math/9908022.

\bibitem[BM14]{BM:projectivity}
Arend Bayer and Emanuele Macr{\`{\i}}.
\newblock Projectivity and birational geometry of {B}ridgeland moduli spaces.
\newblock {\em J.\ Amer.\ Math.\ Soc.}, 27(3):707--752, 2014, arXiv:1203.4613.

\bibitem[BMMS12]{BMMS:Cubics}
Marcello Bernardara, Emanuele Macr{\`{\i}}, Sukhendu Mehrotra, and Paolo
  Stellari.
\newblock A categorical invariant for cubic threefolds.
\newblock {\em Adv. Math.}, 229(2):770--803, 2012, arXiv:0903.4414.

\bibitem[BMS16]{BMS:stabCY3s}
Arend Bayer, Emanuele Macr{\`{\i}}, and Paolo Stellari.
\newblock The space of stability conditions on abelian threefolds, and on some
  {C}alabi-{Y}au threefolds.
\newblock {\em Inventiones Mathematicae}, 206:1--65, 2016.

\bibitem[BMT14]{BMT:3folds-BG}
Arend Bayer, Emanuele Macr{\`{\i}}, and Yukinobu Toda.
\newblock Bridgeland stability conditions on threefolds {I}:
  {B}ogomolov-{G}ieseker type inequalities.
\newblock {\em J. Algebraic Geom.}, 23(1):117--163, 2014, arXiv:1103.5010.

\bibitem[Bot22]{Bottini}
Alessio Bottini.
\newblock Towards a modular construction of og10, 2022, arXiv:2211.09033.

\bibitem[Bri07]{Bridgeland:Stab}
Tom Bridgeland.
\newblock Stability conditions on triangulated categories.
\newblock {\em Ann. of Math. (2)}, 166(2):317--345, 2007, arXiv:math/0212237.

\bibitem[CG72]{CG:RationalityCubics}
C.~Herbert Clemens and Phillip~A. Griffiths.
\newblock The intermediate {J}acobian of the cubic threefold.
\newblock {\em Ann. of Math. (2)}, 95:281--356, 1972.

\bibitem[CS17]{CS:tour}
Alberto Canonaco and Paolo Stellari.
\newblock A tour about existence and uniqueness of dg enhancements and lifts.
\newblock {\em J. Geom. Phys.}, 122:28--52, 2017.

\bibitem[DN89]{DrezetNarasimhan:modulioncurves}
J.-M. Drezet and M.~S. Narasimhan.
\newblock Groupe de {P}icard des vari\'{e}t\'{e}s de modules de fibr\'{e}s
  semi-stables sur les courbes alg\'{e}briques.
\newblock {\em Invent. Math.}, 97(1):53--94, 1989.

\bibitem[FGLZ24]{FGLZ:EPW}
Soheyla Feyzbakhsh, Hanfei Guo, Zhiyu Liu, and Shizhuo Zhang.
\newblock Lagrangian families of bridgeland moduli spaces from gushel-mukai
  fourfolds and double epw cubes, 2024, 2404.11598.

\bibitem[FLM23]{zhiyu:contractKu}
Changping Fan, Zhiyu Liu, and Songtao~Kenneth Ma.
\newblock Stability manifolds of kuznetsov components of prime fano threefolds,
  2023, 2310.16950.

\bibitem[FP23]{FP}
Soheyla Feyzbakhsh and Laura Pertusi.
\newblock Serre-invariant stability conditions and {U}lrich bundles on cubic
  threefolds.
\newblock {\em \'{E}pijournal G\'{e}om. Alg\'{e}brique}, 7:Art. 1, 32, 2023.

\bibitem[GL24]{GL:newpaper}
Hangfei Guo and Zhiyu Liu.
\newblock Atomic sheaves on hyperk\"{a}hler manifolds via bridgeland moduli
  spaces, 2024, in preparation.

\bibitem[HRS96]{Happel-al:tilting}
Dieter Happel, Idun Reiten, and Sverre~O. Smal{\o}.
\newblock Tilting in abelian categories and quasitilted algebras.
\newblock {\em Mem. Amer. Math. Soc.}, 120(575):viii+ 88, 1996.

\bibitem[HRS02]{harris2002abeljacobi}
Joe Harris, Mike Roth, and Jason Starr.
\newblock Abel--jacobi maps associated to smooth cubic threefolds, 2002,
  math/0202080.

\bibitem[HRS05]{HarrisRothStarr:curvesoncubic3fold}
Joe Harris, Mike Roth, and Jason Starr.
\newblock Curves of small degree on cubic threefolds.
\newblock {\em Rocky Mountain J. Math.}, 35(3):761--817, 2005.

\bibitem[Huy23]{Huybrechts:cubichypersurfacebook}
Daniel Huybrechts.
\newblock {\em The geometry of cubic hypersurfaces}, volume 206 of {\em
  Cambridge Studies in Advanced Mathematics}.
\newblock Cambridge University Press, Cambridge, 2023.

\bibitem[IM00]{IlievMarku:cubics}
Atanas Iliev and Dimitri Markushevich.
\newblock The {A}bel-{J}acobi map for a cubic threefold and periods of {F}ano
  threefolds of degree 14.
\newblock {\em Doc. Math.}, 5:23--47 (electronic), 2000, arXiv:math/9910058.

\bibitem[Isk77]{Isk}
V.~A. Iskovskih.
\newblock Fano threefolds. {I}.
\newblock {\em Izv. Akad. Nauk SSSR Ser. Mat.}, 41(3):516--562, 717, 1977.

\bibitem[Isk79]{Is}
V.~A. Iskovskih.
\newblock Anticanonical models of three-dimensional algebraic varieties.
\newblock In {\em Current problems in mathematics, {V}ol. 12 ({R}ussian)},
  Itogi Nauki i Tekhniki, pages 59--157, 239 (loose errata). Akad. Nauk SSSR,
  Vsesoyuz. Inst. Nauchn. i Tekhn. Inform., Moscow, 1979.

\bibitem[JLLZ21]{JLLZ}
Augustinas Jacovskis, Zhiyu Liu, Xun Lin, and Shizhuo Zhang.
\newblock Categorical torelli theorems for gushel-mukai threefolds, 2021,
  arXiv:2108.02946, to appear in Journal of the London Math. Society.

\bibitem[JLZ22]{JLZ}
Augustinas Jacovskis, Zhiyu Liu, and Shizhuo Zhang.
\newblock Brill--noether theory for kuznetsov components and refined
  categorical torelli theorems for index one fano threefolds, 2022,
  arXiv:2207.01021.

\bibitem[Kin94]{King:QuiverStability}
Alastair King.
\newblock Moduli of representations of finite-dimensional algebras.
\newblock {\em Quart. J. Math. Oxford Ser. (2)}, 45(180):515--530, 1994.

\bibitem[KP21]{KuzPerry_serre}
Alexander Kuznetsov and Alexander Perry.
\newblock Serre functors and dimensions of residual categories, 2021,
  arXiv:2109.02026.

\bibitem[KS99]{KingSchofield:rationality}
Alastair King and Aidan Schofield.
\newblock Rationality of moduli of vector bundles on curves.
\newblock {\em Indag. Math. (N.S.)}, 10(4):519--535, 1999.

\bibitem[Kuz04]{Kuz:V14}
Alexander Kuznetsov.
\newblock Derived category of a cubic threefold and the variety {$V_{14}$}.
\newblock {\em Tr. Mat. Inst. Steklova}, 246(Algebr. Geom. Metody, Svyazi i
  Prilozh.):183--207, 2004, arXiv:math/0303037.

\bibitem[Kuz06]{Kuznetsov:Hyperplane}
Alexander Kuznetsov.
\newblock Hyperplane sections and derived categories.
\newblock {\em Izv. Ross. Akad. Nauk Ser. Mat.}, 70(3):23--128, 2006,
  arXiv:math/0503700.

\bibitem[Kuz09]{Kuz:Fano3fold}
Alexander Kuznetsov.
\newblock Derived categories of {F}ano threefolds.
\newblock {\em Tr. Mat. Inst. Steklova}, 264:116--128, 2009.

\bibitem[Kuz19]{Kuznetsov:fracCY}
Alexander Kuznetsov.
\newblock Calabi-{Y}au and fractional {C}alabi-{Y}au categories.
\newblock {\em J. Reine Angew. Math.}, 753:239--267, 2019.

\bibitem[Li19]{Chunyi:Fano}
Chunyi Li.
\newblock Stability conditions on {F}ano threefolds of {P}icard number 1.
\newblock {\em J. Eur. Math. Soc. (JEMS)}, 21(3):709--726, 2019.

\bibitem[LMS15]{LMS}
Mart\'{\i} Lahoz, Emanuele Macr\`{i}, and Paolo Stellari.
\newblock Arithmetically {C}ohen-{M}acaulay bundles on cubic threefolds.
\newblock {\em Algebr. Geom.}, 2(2):231--269, 2015.

\bibitem[LPZ23]{heart}
Chunyi Li, Laura Pertusi, and Xiaolei Zhao.
\newblock Derived categories of hearts on {K}uznetsov components.
\newblock {\em J. Lond. Math. Soc. (2)}, 108(6):2146--2174, 2023.

\bibitem[LZ22]{LiuZhang_note}
Zhiyu Liu and Shizhuo Zhang.
\newblock A note on {B}ridgeland moduli spaces and moduli spaces of sheaves on
  {$X_{14}$} and {$Y_3$}.
\newblock {\em Math. Z.}, 302(2):803--837, 2022.

\bibitem[Mac07]{Macri:curves}
Emanuele Macr{\`{\i}}.
\newblock Stability conditions on curves.
\newblock {\em Math. Res. Lett.}, 14(4):657--672, 2007, arXiv:0705.3794.

\bibitem[MS19]{MS_survey}
Emanuele Macr\`{i} and Paolo Stellari.
\newblock Lectures on non-commutative {K}3 surfaces, {B}ridgeland stability,
  and moduli spaces.
\newblock In {\em Birational geometry of hypersurfaces}, volume~26 of {\em
  Lect. Notes Unione Mat. Ital.}, pages 199--265. Springer, Cham, [2019]
  \copyright 2019.

\bibitem[MU83]{MU}
Shigeru Mukai and Hiroshi Umemura.
\newblock Minimal rational threefolds.
\newblock In {\em Algebraic geometry ({T}okyo/{K}yoto, 1982)}, volume 1016 of
  {\em Lecture Notes in Math.}, pages 490--518. Springer, Berlin, 1983.

\bibitem[Per22]{Perry_2022}
Alexander Perry.
\newblock The integral hodge conjecture for two-dimensional calabi–yau
  categories.
\newblock {\em Compositio Mathematica}, 158(2):287–333, 2022.

\bibitem[PPZ23]{PPZ_Enriques}
Alexander Perry, Laura Pertusi, and Xiaolei Zhao.
\newblock Moduli spaces of stable objects in enriques categories, 2023,
  arXiv:2305.10702.

\bibitem[PR23]{PR}
Laura Pertusi and Ethan Robinett.
\newblock Stability conditions on {K}uznetsov components of {G}ushel-{M}ukai
  threefolds and {S}erre functor.
\newblock {\em Math. Nachr.}, 296(7):2975--3002, 2023.

\bibitem[PS23]{PS_survey}
Laura Pertusi and Paolo Stellari.
\newblock Categorical {T}orelli theorems: results and open problems.
\newblock {\em Rend. Circ. Mat. Palermo (2)}, 72(5):2949--3011, 2023.

\bibitem[PY22]{PY}
Laura Pertusi and Song Yang.
\newblock Some remarks on {F}ano three-folds of index two and stability
  conditions.
\newblock {\em Int. Math. Res. Not. IMRN}, 2022(17):12940--12983, 2022.

\bibitem[Ram73]{ramanan1973moduli}
Sundararaman Ramanan.
\newblock The moduli spaces of vector bundles over an algebraic curve.
\newblock {\em Mathematische Annalen}, 200:69--84, 1973.

\bibitem[Shi22]{Shinder_2022}
Evgeny Shinder.
\newblock {\em Variation of Stable Birational Types of Hypersurfaces}, page
  296–313.
\newblock London Mathematical Society Lecture Note Series. Cambridge University
  Press, 2022.

\bibitem[Voi02]{Voisin:hogedbook}
Claire Voisin.
\newblock {\em Hodge theory and complex algebraic geometry. {I}}, volume~76 of
  {\em Cambridge Studies in Advanced Mathematics}.
\newblock Cambridge University Press, Cambridge, 2002.
\newblock Translated from the French original by Leila Schneps.

\bibitem[VP21]{villalobos:proj_moishezon}
David Villalobos-Paz.
\newblock Moishezon spaces and projectivity criteria, 2021, arXiv:2105.14630.

\bibitem[Zha20]{Zhang}
Shizhuo Zhang.
\newblock Bridgeland moduli spaces for gushel-mukai threefolds and kuznetsov's
  fano threefold conjecture, 2020, arXiv:2012.12193.

\end{thebibliography}
\bibliographystyle{halpha}

\end{document}